\documentclass[12pt, leqno]{amsart}
\usepackage[utf8]{inputenc}
\usepackage{xy, amsfonts,amsmath, mathrsfs, amssymb,latexsym, csquotes, pgf, tikz}

\usepackage{hyperref}
\definecolor{darkred}{rgb}{0.7,0,0}
\hypersetup{
    pdftoolbar=true,        
    pdfmenubar=true,        
    pdffitwindow=false,     
    pdfstartview={FitH},    
    pdfauthor={Mikhail Khovanov and Louis-Hadrien Robert},     
    pdfsubject={Foam evaluation and Kronheimer--Mrowka theories},   
    pdfcreator={Mikhail Khovanov and Louis-Hadrien Robert},   
    pdfproducer={Mikhail Khovanov and Louis-Hadrien Robert}, 
    pdfkeywords={}, 
    pdfnewwindow=true,      
    colorlinks=true,       
    linkcolor=darkred,          
    citecolor=teal,        
    filecolor=magenta,      
    urlcolor=violet,          
    linkbordercolor=darkred,
    citebordercolor=teal,
    urlbordercolor=violet,  
    linktocpage=true
}
\usepackage[hmargin=3cm,vmargin=3.5cm]{geometry}

\usetikzlibrary{arrows}
\usetikzlibrary{arrows.meta}
\usetikzlibrary{intersections}
\usetikzlibrary{decorations.markings}
\usetikzlibrary{decorations}
\usetikzlibrary{patterns}
\usepackage{tikz-3dplot}
\tikzset{yxplane/.style={canvas is xy plane at z=#1}}
\tikzset{>=latex}
\tikzset{->-/.style={decoration={
  markings,
  mark=at position .5 with {\arrow{>}}},postaction={decorate}}}
\tikzset{-<-/.style={decoration={
  markings,
  mark=at position .5 with {\arrow{<}}},postaction={decorate}}}
\usepackage{amscd}
\usepackage{psfrag}

\title{Foam evaluation and Kronheimer--Mrowka theories}
\author{Mikhail Khovanov} 
\address{Department of Mathematics, Columbia University, New York, NY 10027, USA}
\email{\href{mailto:khovanov@math.columbia.edu}{khovanov@math.columbia.edu}}
\author{Louis-Hadrien Robert}
\address{Université de Genève, 2-4 rue du Lièvre, Case postale 64, 1211 Genève 4, Switzerland}
\email{\href{mailto:louis-hadrien.robert@unige.ch}{louis-hadrien.robert@unige.ch}}

\date{August 29th, 2018}

\newcounter{res}[section]
\numberwithin{res}{section}
\newtheorem{prop}[res]{Proposition}
\newtheorem{theorem}[res]{Theorem}
\newtheorem{corollary}[res]{Corollary} 
\newtheorem{lemma}[res]{Lemma}
\newtheorem{conjecture}[res]{Conjecture}
\theoremstyle{definition}
\newtheorem{defn}[res]{Definition} 
\newtheorem{example}[res]{Example} 
\newtheorem{rmk}[res]{Remark} 
\newcommand{\brak}[1]{\ensuremath{\left\langle #1\right\rangle}}
\newcommand{\oplusop}[1]{{\mathop{\oplus}\limits_{#1}}}

\let\oldtocsection=\tocsection
\let\oldtocsubsection=\tocsubsection
\renewcommand{\tocsection}[2]{\hspace{0em}\oldtocsection{#1}{#2}}
\renewcommand{\tocsubsection}[2]{\hspace{1em}\oldtocsubsection{#1}{#2}}

\def\R{\mathbb R}

\def\Z{\mathbb Z}

\renewcommand\SS{\ensuremath{\mathbb{S}}}

\def\lra{\longrightarrow}
\def\Hom{\mathrm{Hom}}
\def\Id{\mathrm{Id}}

\def\lra{\longrightarrow}
\def\kk{\mathbf{k}}
\def\kf{\mathbf{k}'}
\def\lF{\langle F\rangle}
\def\lG{\langle \Gamma\rangle}
\def\lFc{\langle F,c\rangle }

\def\HMV{\mathrm{H}_{\mathrm{MV}}}

\def\adm{\mathrm{adm}}
\def\gdim{\mathrm{gdim}}
\def\rk{\mathrm{rk}}

\newcommand\Fo{\mathrm{Fo}}

\newcommand{\imagesfolder}{.}
\newcommand{\NB}[1]{\ensuremath{\vcenter{\hbox{#1}}}}

\newcommand{\tube}[1][0.7]{
\NB{\tikz[scale=#1]{
\begin{scope}
\draw (0,1) circle (0.5cm and 0.25cm);
\draw (0.5,-1) arc (0:-180:0.5cm and 0.25cm);
\draw[densely dotted] (0.5,-1) arc (0:180:0.5cm and 0.25cm);
\draw (0.5, -1) -- (0.5,1);
\draw (-.5, -1) -- (-.5,1);
\end{scope}}
}\!}

\newcommand{\cupxxxx}{
\begin{scope}
\draw (0,1) circle (0.5cm and 0.25cm);
\draw (0.5,1) arc (0:-180:0.5cm and 0.7cm);
\end{scope}}

\newcommand{\capxxxx}{
\begin{scope}
\draw (0.5,-1) arc (0:-180:0.5cm and 0.25cm);
\draw[densely dotted] (0.5,-1) arc (0:180:0.5cm and 0.25cm);
\draw (0.5,-1) arc (0:180:0.5cm and 0.7cm);
\end{scope}}

\newcommand{\captwo}[1][0.7]{
\NB{\tikz[scale=#1]{
\begin{scope}
\capxxxx
\fill (-0.1, -0.6) circle (0.5mm);
\fill (0.1, -0.6) circle (0.5mm);
\end{scope}}
}\!}

\newcommand{\capone}[1][0.7]{
\NB{\tikz[scale=#1]{
\begin{scope}
\capxxxx
\fill (0, -0.6) circle (0.5mm);
\end{scope}}
}\!}

\newcommand{\capzero}[1][0.7]{
\NB{\tikz[scale=#1]{
\begin{scope}
\capxxxx
\end{scope}}
}\!}

\newcommand{\cuptwo}[1][0.7]{
\NB{\tikz[scale=#1]{
\begin{scope}
\cupxxxx
\fill (-0.1, 0.6) circle (0.5mm);
\fill (0.1, 0.6) circle (0.5mm);
\end{scope}}
}\!}

\newcommand{\cupone}[1][0.7]{
\NB{\tikz[scale=#1]{
\begin{scope}
\cupxxxx
\fill (0, 0.6) circle (0.5mm);
\end{scope}}
}\!}

\newcommand{\cupzero}[1][0.7]{
\NB{\tikz[scale=#1]{
\begin{scope}
\cupxxxx
\end{scope}}
}\!}

\newcommand{\thsxxxx}{
\begin{scope}
\draw (0,1) circle (0.5cm and 0.25cm);
\draw (0.5,-1) arc (0:-180:0.5cm and 0.25cm);
\draw[densely dotted] (0.5,-1) arc (0:180:0.5cm and 0.25cm);
\draw (0.5,-1) arc (0:180:0.5cm and 0.7cm);
\draw (0.5,1) arc (0:-180:0.5cm and 0.7cm);
\end{scope}}

\newcommand{\thstwozero}[1][0.7]{
\NB{\tikz[scale=#1]{
\begin{scope}
\thsxxxx
\fill (-0.1, 0.6) circle (0.5mm);
\fill (0.1, 0.6) circle (0.5mm);
\end{scope}}
}\!}

\newcommand{\thsoneone}[1][0.7]{
\NB{\tikz[scale=#1]{
\begin{scope}
\thsxxxx
\fill (0, 0.6) circle (0.5mm);
\fill (0, -0.6) circle (0.5mm);
\end{scope}}
}\!}

\newcommand{\thszerotwo}[1][0.7]{
\NB{\tikz[scale=#1]{
\begin{scope}
\thsxxxx
\fill (-0.1, -0.6) circle (0.5mm);
\fill (0.1, -0.6) circle (0.5mm);
\end{scope}}
}\!}

\newcommand{\thsonezero}[1][0.7]{
\NB{\tikz[scale=#1]{
\begin{scope}
\thsxxxx
\fill (0, 0.6) circle (0.5mm);
\end{scope}}
}\!}

\newcommand{\thszeroone}[1][0.7]{
\NB{\tikz[scale=#1]{
\begin{scope}
\thsxxxx
\fill (0, -0.6) circle (0.5mm);
\end{scope}}
}\!}

\newcommand{\thszerozero}[1][0.7]{
\NB{\tikz[scale=#1]{
\begin{scope}
\thsxxxx
\end{scope}}
}\!}

\newcommand{\digonI}[1][0.7]{
\NB{\tikz[scale=#1]{
\begin{scope}
\draw (-1,1) -- (-0.5,1) .. controls +(0.25,-0.25) and +(-0.25,-0.25) .. (0.5,1) -- (1,1);
\draw (-1,-1) -- (-0.5,-1) .. controls +(0.25,-0.25) and +(-0.25,-0.25) .. (0.5,-1) -- (1,-1);
\draw (-0.5,1) .. controls +(0.25,0.25) and +(-0.25,0.25) .. (0.5,1);
\draw[densely dotted] (-0.5,-1) .. controls +(0.25,0.25) and +(-0.25,0.25) .. (0.5,-1);
\draw (0.5, -1) -- (0.5,1);
\draw (-.5, -1) -- (-.5,1);
\draw (1, -1) -- (1,1);
\draw (-1, -1) -- (-1,1);
\end{scope}}
}\!}

\newcommand{\thbxxxx}{
\draw (-1,1) -- (-0.5,1) .. controls +(0.25,-0.25) and +(-0.25,-0.25) .. (0.5,1) -- (1,1);
\draw (-1,-1) -- (-0.5,-1) .. controls +(0.25,-0.25) and +(-0.25,-0.25) .. (0.5,-1) -- (1,-1);
\draw (-0.5,1) .. controls +(0.25,0.25) and +(-0.25,0.25) .. (0.5,1);
\draw[densely dotted] (-0.5,-1) .. controls +(0.25,0.25) and +(-0.25,0.25) .. (0.5,-1);
\draw (1, -1) -- (1,1);
\draw (-1, -1) -- (-1,1);
\draw (0.5,-1) arc (0:180:0.5cm and 0.7cm);
\draw (0.5,1) arc (0:-180:0.5cm and 0.7cm);
}

\newcommand{\thbonezero}[1][0.7]{
\NB{\tikz[scale=#1]{
\begin{scope}
\thbxxxx
\fill (0, 1) circle (0.5mm);
\end{scope}}
}\!}

\newcommand{\thbzeroone}[1][0.7]{
\NB{\tikz[scale=#1]{
\begin{scope}
\thbxxxx
\fill (0, -1) circle (0.5mm);
\end{scope}}
}\!}

\newcommand{\digoncupxxxx}{
\draw (-1,1) -- (-0.5,1) .. controls +(0.25,-0.25) and +(-0.25,-0.25) .. (0.5,1) -- (1,1);
\draw (-0.5,1) .. controls +(0.25,0.25) and +(-0.25,0.25) .. (0.5,1);
\draw (1, 1) -- (1,0) -- (-1, 0) -- (-1,1);
\draw (0.5,1) arc (0:-180:0.5cm and 0.7cm);
}

\newcommand{\digoncupzero}[1][0.7]{
\NB{\tikz[scale=#1]{
\begin{scope}
\digoncupxxxx
\end{scope}}
}\!}

\newcommand{\digoncupone}[1][0.7]{
\NB{\tikz[scale=#1]{
\begin{scope}
\digoncupxxxx
\fill (0, 1) circle (0.5mm);
\end{scope}}
}\!}

\newcommand{\digoncapxxxx}{
\draw (-1,-1) -- (-0.5,-1) .. controls +(0.25,-0.25) and +(-0.25,-0.25) .. (0.5,-1) -- (1,-1);
\draw[densely dotted] (-0.5,-1) .. controls +(0.25,0.25) and +(-0.25,0.25) .. (0.5,-1);
\draw (1, -1) -- (1,0) -- (-1, 0) -- (-1,-1);
\draw (0.5,-1) arc (0:180:0.5cm and 0.7cm);
}

\newcommand{\digoncapzero}[1][0.7]{
\NB{\tikz[scale=#1]{
\begin{scope}
\digoncapxxxx
\end{scope}}
}\!}

\newcommand{\digoncapone}[1][0.7]{
\NB{\tikz[scale=#1]{
\begin{scope}
\digoncapxxxx
\fill (0, -1) circle (0.5mm);
\end{scope}}
}\!}

\newcommand{\squareI}[1][0.7]{
\NB{
\tikz[scale=#1]{
\tdplotsetmaincoords{70}{30}
  \begin{scope}[tdplot_main_coords]
  	\draw (-0.5, -0.5, 1) -- (0.5, -0.5, 1) -- (0.5, 0.5, 1) -- (-0.5, 0.5, 1) -- (-0.5, -0.5, 1);
    \draw (-0.5, -0.5, 1) -- (-1,-1, 1);
  	\draw ( 0.5, -0.5, 1) -- ( 1,-1, 1);
  	\draw (-0.5,  0.5, 1) -- (-1, 1, 1);
  	\draw ( 0.5,  0.5, 1) -- ( 1, 1, 1);
  	\draw (-0.5, -0.5,-1) -- (0.5, -0.5,-1) -- (0.5, 0.5,-1) -- (-0.5, 0.5,-1) -- (-0.5, -0.5,-1);
  	\draw (-0.5, -0.5,-1) -- (-1,-1,-1);
  	\draw ( 0.5, -0.5,-1) -- ( 1,-1,-1);
  	\draw (-0.5,  0.5,-1) -- (-1, 1,-1);
  	\draw ( 0.5,  0.5,-1) -- ( 1, 1,-1);
  	\draw (-0.5, -0.5, 1) -- +(0,0,-2);
  	\draw ( 0.5, -0.5, 1) -- +(0,0,-2);
  	\draw (-0.5,  0.5, 1) -- +(0,0,-2);
  	\draw ( 0.5,  0.5, 1) -- +(0,0,-2);
  	\draw (-1, -1, 1) -- +(0,0,-2);
  	\draw ( 1, -1, 1) -- +(0,0,-2);
  	\draw (-1,  1, 1) -- +(0,0,-2);
  	\draw ( 1,  1, 1) -- +(0,0,-2);   
  \end{scope}
}
}
}

\newcommand{\squaresmoothone}[1][0.7]{
\NB{
\tikz[scale=#1]{
\tdplotsetmaincoords{70}{30}
  \begin{scope}[tdplot_main_coords]
  	\draw (-0.5, -0.5, 1) -- (0.5, -0.5, 1) -- (0.5, 0.5, 1) -- (-0.5, 0.5, 1) -- (-0.5, -0.5, 1);
    \draw (-0.5, -0.5, 1) -- (-1,-1, 1);
  	\draw ( 0.5, -0.5, 1) -- ( 1,-1, 1);
  	\draw (-0.5,  0.5, 1) -- (-1, 1, 1);
  	\draw ( 0.5,  0.5, 1) -- ( 1, 1, 1);
  	\draw (-0.5, -0.5,-1) -- (0.5, -0.5,-1) -- (0.5, 0.5,-1) -- (-0.5, 0.5,-1) -- (-0.5, -0.5,-1);
  	\draw (-0.5, -0.5,-1) -- (-1,-1,-1);
  	\draw ( 0.5, -0.5,-1) -- ( 1,-1,-1);
  	\draw (-0.5,  0.5,-1) -- (-1, 1,-1);
  	\draw ( 0.5,  0.5,-1) -- ( 1, 1,-1);
  	\draw (-1, -1, 1) -- +(0,0,-2);
  	\draw ( 1, -1, 1) -- +(0,0,-2);
  	\draw (-1,  1, 1) -- +(0,0,-2);
  	\draw ( 1,  1, 1) -- +(0,0,-2);   
    \draw (-0.5, -0.5, 1) .. controls +(0,0,-1) and +(0,0,-1) .. (-0.5, 0.5, 1) coordinate[pos= 0.5] (T1);
  	\draw ( 0.5,  0.5, 1) .. controls +(0,0,-1) and +(0,0,-1) .. ( 0.5,-0.5, 1) coordinate[pos= 0.5] (T2);
  	\draw (-0.5, -0.5,-1) .. controls +(0,0, 1) and +(0,0, 1) .. (-0.5, 0.5,-1) coordinate[pos= 0.5] (B1);
  	\draw ( 0.5,  0.5,-1) .. controls +(0,0, 1) and +(0,0, 1) .. ( 0.5,-0.5,-1) coordinate[pos= 0.5] (B2);
    \draw[very thin] (T1) -- (T2);
    \draw[very thin] (B1) -- (B2);
\end{scope}
}
}
}

\newcommand{\squaresmoothtwo}[1][0.7]{
\NB{
\tikz[scale=#1]{
\tdplotsetmaincoords{70}{30}
  \begin{scope}[tdplot_main_coords]
  	\draw (-0.5, -0.5, 1) -- (0.5, -0.5, 1) -- (0.5, 0.5, 1) -- (-0.5, 0.5, 1) -- (-0.5, -0.5, 1);
    \draw (-0.5, -0.5, 1) -- (-1,-1, 1);
  	\draw ( 0.5, -0.5, 1) -- ( 1,-1, 1);
  	\draw (-0.5,  0.5, 1) -- (-1, 1, 1);
  	\draw ( 0.5,  0.5, 1) -- ( 1, 1, 1);
  	\draw (-0.5, -0.5,-1) -- (0.5, -0.5,-1) -- (0.5, 0.5,-1) -- (-0.5, 0.5,-1) -- (-0.5, -0.5,-1);
  	\draw (-0.5, -0.5,-1) -- (-1,-1,-1);
  	\draw ( 0.5, -0.5,-1) -- ( 1,-1,-1);
  	\draw (-0.5,  0.5,-1) -- (-1, 1,-1);
  	\draw ( 0.5,  0.5,-1) -- ( 1, 1,-1);
  	\draw (-1, -1, 1) -- +(0,0,-2);
  	\draw ( 1, -1, 1) -- +(0,0,-2);
  	\draw (-1,  1, 1) -- +(0,0,-2);
  	\draw ( 1,  1, 1) -- +(0,0,-2);   
    \draw (-0.5, -0.5, 1) .. controls +(0,0,-1) and +(0,0,-1) .. ( 0.5,-0.5, 1) coordinate[pos= 0.5] (T1);
  	\draw ( 0.5,  0.5, 1) .. controls +(0,0,-1) and +(0,0,-1) .. (-0.5, 0.5, 1) coordinate[pos= 0.5] (T2);
  	\draw (-0.5, -0.5,-1) .. controls +(0,0, 1) and +(0,0, 1) .. ( 0.5,-0.5,-1) coordinate[pos= 0.5] (B1);
  	\draw ( 0.5,  0.5,-1) .. controls +(0,0, 1) and +(0,0, 1) .. (-0.5, 0.5,-1) coordinate[pos= 0.5] (B2);
    \draw[very thin] (T1) -- (T2);
    \draw[very thin] (B1) -- (B2);
\end{scope}
}
}
}

\newcommand{\squareTOsmoothone}[1][0.7]{
\NB{
\tikz[scale=#1]{
\tdplotsetmaincoords{70}{30}
  \begin{scope}[tdplot_main_coords]
  	\draw (-0.5, -0.5,-1) -- (0.5, -0.5,-1) -- (0.5, 0.5,-1) -- (-0.5, 0.5,-1) -- (-0.5, -0.5,-1);
  	\draw (-0.5, -0.5,-1) -- (-1,-1,-1);
  	\draw ( 0.5, -0.5,-1) -- ( 1,-1,-1);
  	\draw (-0.5,  0.5,-1) -- (-1, 1,-1);
  	\draw ( 0.5,  0.5,-1) -- ( 1, 1,-1);
  	\draw (-1, -1, 0.5) -- +(0,0,-1.5);
  	\draw ( 1, -1, 0.5) -- +(0,0,-1.5);
  	\draw (-1,  1, 0.5) -- +(0,0,-1.5);
  	\draw ( 1,  1, 0.5) -- +(0,0,-1.5);
        \draw (-1, -1, 0.5) .. controls +( 0.4, 0.4, 0) and +( 0.4,-0.4, 0) .. (-1, 1, 0.5);
        \draw ( 1,  1, 0.5) .. controls +(-0.4,-0.4, 0) and +(-0.4, 0.4, 0) .. ( 1,-1, 0.5);
  	\draw[very thin] (-0.5, -0.5,-1) .. controls +(0,0, 1) and +(0,0, 1) .. (-0.5, 0.5,-1) coordinate[pos= 0.5] (B1);
  	\draw[very thin] ( 0.5,  0.5,-1) .. controls +(0,0, 1) and +(0,0, 1) .. ( 0.5,-0.5,-1) coordinate[pos= 0.5] (B2);
        \draw (B1) -- (B2);
\end{scope}
}
}
}

\newcommand{\squareTOsmoothtwo}[1][0.7]{
\NB{
\tikz[scale=#1]{
\tdplotsetmaincoords{70}{30}
  \begin{scope}[tdplot_main_coords]
  	\draw (-0.5, -0.5,-1) -- (0.5, -0.5,-1) -- (0.5, 0.5,-1) -- (-0.5, 0.5,-1) -- (-0.5, -0.5,-1);
  	\draw (-0.5, -0.5,-1) -- (-1,-1,-1);
  	\draw ( 0.5, -0.5,-1) -- ( 1,-1,-1);
  	\draw (-0.5,  0.5,-1) -- (-1, 1,-1);
  	\draw ( 0.5,  0.5,-1) -- ( 1, 1,-1);
  	\draw (-1, -1, 0.5) -- +(0,0,-1.5);
  	\draw ( 1, -1, 0.5) -- +(0,0,-1.5);
  	\draw (-1,  1, 0.5) -- +(0,0,-1.5);
  	\draw ( 1,  1, 0.5) -- +(0,0,-1.5);
        \draw (-1, -1, 0.5) .. controls +( 0.4, 0.4, 0) and +(-0.4, 0.4, 0) .. ( 1,-1, 0.5);
        \draw ( 1,  1, 0.5) .. controls +(-0.4,-0.4, 0) and +( 0.4,-0.4, 0) .. (-1, 1, 0.5);
  	\draw[very thin] (-0.5, -0.5,-1) .. controls +(0,0, 1) and +(0,0, 1) .. ( 0.5,-0.5,-1) coordinate[pos= 0.5] (B1);
  	\draw[very thin] ( 0.5,  0.5,-1) .. controls +(0,0, 1) and +(0,0, 1) .. (-0.5, 0.5,-1) coordinate[pos= 0.5] (B2);
        \draw (B1) -- (B2);
\end{scope}
}
}
}

\newcommand{\smoothoneTOsquare}[1][0.7]{
\NB{
\tikz[scale=#1]{
\tdplotsetmaincoords{70}{30}
  \begin{scope}[tdplot_main_coords]
  	\draw (-0.5, -0.5, 1) -- (0.5, -0.5, 1) -- (0.5, 0.5, 1) -- (-0.5, 0.5, 1) -- (-0.5, -0.5, 1);
        \draw (-0.5, -0.5, 1) -- (-1,-1, 1);
  	\draw ( 0.5, -0.5, 1) -- ( 1,-1, 1);
  	\draw (-0.5,  0.5, 1) -- (-1, 1, 1);
   	\draw ( 0.5,  0.5, 1) -- ( 1, 1, 1);
  	\draw (-1, -1, -0.5) -- +(0,0,1.5);
  	\draw ( 1, -1, -0.5) -- +(0,0,1.5);
  	\draw (-1,  1, -0.5) -- +(0,0,1.5);
  	\draw ( 1,  1, -0.5) -- +(0,0,1.5);
        \draw (-1, -1, -0.5) .. controls +( 0.4, 0.4, 0) and +( 0.4,-0.4, 0) .. (-1, 1, -0.5);
        \draw ( 1,  1, -0.5) .. controls +(-0.4,-0.4, 0) and +(-0.4, 0.4, 0) .. ( 1,-1, -0.5);
        \draw[very thin] (-0.5, -0.5, 1) .. controls +(0,0,-1) and +(0,0,-1) .. (-0.5, 0.5, 1) coordinate[pos= 0.5] (T1);
   	\draw[very thin] ( 0.5,  0.5, 1) .. controls +(0,0,-1) and +(0,0,-1) .. ( 0.5,-0.5, 1) coordinate[pos= 0.5] (T2);
        \draw (T1) -- (T2);
\end{scope}
}
}
}

\newcommand{\smoothtwoTOsquare}[1][0.7]{
\NB{
\tikz[scale=#1]{
\tdplotsetmaincoords{70}{30}
  \begin{scope}[tdplot_main_coords]
  	\draw (-0.5, -0.5, 1) -- (0.5, -0.5, 1) -- (0.5, 0.5, 1) -- (-0.5, 0.5, 1) -- (-0.5, -0.5, 1);
        \draw (-0.5, -0.5, 1) -- (-1,-1, 1);
  	\draw ( 0.5, -0.5, 1) -- ( 1,-1, 1);
  	\draw (-0.5,  0.5, 1) -- (-1, 1, 1);
  	\draw ( 0.5,  0.5, 1) -- ( 1, 1, 1);
  	\draw (-1, -1, -0.5) -- +(0,0, 1.5);
  	\draw ( 1, -1, -0.5) -- +(0,0, 1.5);
  	\draw (-1,  1, -0.5) -- +(0,0, 1.5);
  	\draw ( 1,  1, -0.5) -- +(0,0, 1.5);
        \draw (-1, -1, -0.5) .. controls +( 0.4, 0.4, 0) and +(-0.4, 0.4, 0) .. ( 1,-1, -0.5);
        \draw ( 1,  1, -0.5) .. controls +(-0.4,-0.4, 0) and +( 0.4,-0.4, 0) .. (-1, 1, -0.5);
        \draw[very thin] (-0.5, -0.5, 1) .. controls +(0,0,-1) and +(0,0,-1) .. (-0.5, 0.5, 1) coordinate[pos= 0.5] (T1);
    	\draw[very thin] ( 0.5,  0.5, 1) .. controls +(0,0,-1) and +(0,0,-1) .. ( 0.5,-0.5, 1) coordinate[pos= 0.5] (T2);
        \draw (T1) -- (T2);
\end{scope}
}
}
}

\newcommand{\triangleTOvertex}[1][0.7]{
\NB{
\tikz[scale=#1]{
\tdplotsetmaincoords{70}{130}
\begin{scope}[scale = 1.5, tdplot_main_coords]
  \tikzset{yxplane/.style={canvas is xy plane at z=#1}}
    \coordinate (AT) at ({cos(  0)}, {sin(  0)}, 0.3);
    \coordinate (BT) at ({cos(120)}, {sin(120)}, 0.3);
    \coordinate (CT) at ({cos(240)}, {sin(240)}, 0.3);
     \coordinate (OT) at ({0},{0}, 0.3);
     \draw (OT) -- (AT);
     \draw (OT) -- (BT);
     \draw (OT) -- (CT);
    \coordinate (AB) at ({cos(  0)}, {sin(  0)},-1);
    \coordinate (BB) at ({cos(120)}, {sin(120)}, -1);
    \coordinate (CB) at ({cos(240)}, {sin(240)}, -1);
    \coordinate (aB) at ({0.5*cos(  0)}, {0.5*sin(  0)}, -1);
    \coordinate (bB) at ({0.5*cos(120)}, {0.5*sin(120)}, -1);
    \coordinate (cB) at ({0.5*cos(240)}, {0.5*sin(240)}, -1);
    \draw (aB) -- (bB);
    \draw (bB) -- (cB);
    \draw (cB) -- (aB);
    \draw (aB) -- (AB);
    \draw (bB) -- (BB);
    \draw (cB) -- (CB);
  \coordinate (M1) at (0,0, -0.3);
  \coordinate (M2) at (0,0,  0.3);
  \draw (M1) -- (M2);
  \draw (AB) -- (AT);
  \draw (BB) -- (BT);
  \draw (CB) -- (CT);
  \draw (aB) -- (M1);
  \draw (bB) -- (M1);
  \draw (cB) -- (M1);
  \draw (AT) -- (M2);
  \draw (BT) -- (M2);
  \draw (BT) -- (M2);
\end{scope}
}}}

\newcommand{\vertexTOtriangle}[1][0.7]{
\NB{
\tikz[scale=#1]{
\tdplotsetmaincoords{70}{130}
\begin{scope}[scale = 1.5, tdplot_main_coords]
  \tikzset{yxplane/.style={canvas is xy plane at z=#1}}
    \coordinate (AT) at ({cos(  0)}, {sin(  0)}, 1);
    \coordinate (BT) at ({cos(120)}, {sin(120)}, 1);
    \coordinate (CT) at ({cos(240)}, {sin(240)}, 1);
    \coordinate (aT) at ({0.5*cos(  0)}, {0.5*sin(  0)},1);
    \coordinate (bT) at ({0.5*cos(120)}, {0.5*sin(120)},1);
    \coordinate (cT) at ({0.5*cos(240)}, {0.5*sin(240)},1);
    \draw (aT) -- (bT);
    \draw (bT) -- (cT);
    \draw (cT) -- (aT);
     \draw (aT) -- (AT);
     \draw (bT) -- (BT);
     \draw (cT) -- (CT);
    \coordinate (AB) at ({cos(  0)}, {sin(  0)},-0.3);
    \coordinate (BB) at ({cos(120)}, {sin(120)}, -0.3);
    \coordinate (CB) at ({cos(240)}, {sin(240)}, -0.3);
  \coordinate (M1) at (0,0, -0.3);
  \coordinate (M2) at (0,0,  0.3);
  \draw (M1) -- (M2);
  \draw (AB) -- (AT);
  \draw (BB) -- (BT);
  \draw (CB) -- (CT);
  \draw (AB) -- (M1);
  \draw (BB) -- (M1);
  \draw (CB) -- (M1);
  \draw (aT) -- (M2);
  \draw (bT) -- (M2);
  \draw (cT) -- (M2);
\end{scope}
}}}

\newcommand{\vertexIdottedxxx}{
\tdplotsetmaincoords{70}{130}
\begin{scope}[tdplot_main_coords]
\begin{scope}[yxplane=1]
    \coordinate (OT) at ( 0, 0);
    \coordinate (AT) at ({cos(  0)}, {sin(  0)});
    \coordinate (BT) at ({cos(120)}, {sin(120)});
    \coordinate (CT) at ({cos(240)}, {sin(240)});
    \draw (OT) -- (AT);
    \draw (OT) -- (BT);
    \draw (OT) -- (CT);
  \end{scope}
\begin{scope}[yxplane=-1]
    \coordinate (OB) at (0, 0);
    \coordinate (AB) at ({cos(  0)}, {sin(  0)});
    \coordinate (BB) at ({cos(120)}, {sin(120)});
    \coordinate (CB) at ({cos(240)}, {sin(240)});
    \draw (OB) -- (AB);
    \draw (OB) -- (BB);
    \draw (OB) -- (CB);
  \end{scope}
  \begin{scope}[yxplane=-0.85]
    \coordinate (AM) at ({cos(  0)*0.8}, {sin(  0)*0.8});
  \end{scope}
  \begin{scope}[yxplane=0]
    \coordinate (BM) at ({cos(120)*0.5}, {sin(120)*0.5});
  \end{scope}
  \begin{scope}[yxplane=0.85]
    \coordinate (CM) at ({cos(240)*0.6}, {sin(240)*0.6});
  \end{scope}
   \draw (OB) -- (OT);
   \draw (AB) -- (AT);
   \draw (BB) -- (BT);
   \draw (CB) -- (CT);
\end{scope}
}

\newcommand{\vertexIdottedzerozeroone}[1][0.7]{
\NB{
\tikz[scale=#1]{
\vertexIdottedxxx
\fill (CM) circle (0.7mm);
}}}

\newcommand{\vertexIdottedonetrianglezero}[1][0.7]{
\NB{
\tikz[scale=#1]{
\vertexIdottedxxx
\node[scale = 0.7] at (BM) {$\triangle$} ;
\fill (AM) circle (0.7mm);
}}}

\newcommand{\vertexIdottedzerotriangleone}[1][0.7]{
\NB{
\tikz[scale=#1]{
\vertexIdottedxxx
\node[scale = 0.7] at (BM) {$\triangle$} ;
\fill (CM) circle (0.7mm);
}}}

\newcommand{\vertexInodot}[1][0.7]{
\NB{
\tikz[scale=#1]{
\vertexIdottedxxx
}}}

\newcommand{\vertexIdottedzeroonezero}[1][0.7]{
\NB{
\tikz[scale=#1]{
\vertexIdottedxxx
\fill (BM) circle (0.7mm);
}}}

\newcommand{\vertexIdottedonezerozero}[1][0.7]{
\NB{
\tikz[scale=#1]{
\vertexIdottedxxx
\fill (AM) circle (0.7mm);
}}}

\newcommand{\decorateddisk}[3]{
\NB{
\tikz[scale=#3]{
\begin{scope}
  \coordinate (A) at (0,0);
  \coordinate (B) at (2,0);
  \coordinate (C) at (2.5,0.5);
  \coordinate (D) at (0.5,0.5);
  \node[scale = #2] (M) at (1.25, 0.25) {$#1$};
  \draw (C) -- (D) -- (A) -- (B) -- (C);
\end{scope}}}}

\newcommand{\define}{\stackrel{\mbox{\scriptsize{def}}}{=}}

\newcommand{\catF}{\mathsf{Foams}}
\newcommand{\D}{\ensuremath{\mathscr{D}}}

\newcommand{\qdim}{\ensuremath{\mathrm{qdim}}}

\begin{document} 

\begin{abstract}  We introduce and study combinatorial equivariant analogues of the Kronheimer--Mrowka homology theory of planar trivalent graphs. 
\end{abstract}

\maketitle
\baselineskip 14pt

\tableofcontents

\section{Introduction} 
\label{sec:intro}

This paper uses an unoriented version of the
Robert--Wagner foam evaluation formula~\cite{RW1}, specialized to three colors, to
construct and study combinatorial relatives of  Kronheimer--Mrowka homology theories 
for planar unoriented trivalent graphs. 
 Kronheimer--Mrowka defined their homology $J^{\sharp}$ in much greater generality, for 
trivalent graphs embedded in oriented 3-manifolds~\cite{KM1, KM2, KM3}.
Their theory comes 
from the $SO(3)$ gauge theory for 3-orbifolds. As a special case, it gives a functorial 
homology theory for trivalent graphs embedded in $\R^3$; such graphs generalize 
unoriented knots and links. Further restricting to graphs in $\R^2$ and using the 
embedding $\R^2\subset \R^3$, Kronheimer and Mrowka obtain a homology theory
for planar trivalent graphs.  

This homology theory is defined over the two-element field $\kk$ and consists of a single vector space, without an additional grading. 
Conjecturally, for a planar graph $K$, the dimension  $\dim_{\kk}(J^{\sharp}(K))$ is equal to 
its number of Tait colorings, which are 3-colorings of edges 
of $K$ such that edges that share a vertex carry distinct colors. 
Kronheimer--Mrowka nonvanishing results, utilizing Gabai's sutured theory, and the 
conjecture would imply the four-color theorem~\cite{ApHa,ApHaKo,ApHaBook, Thomas}, providing an alternative approach 
to the theorem and relating it to gauge theory and topology in low dimensions. 

Robert--Wagner foam evaluation formula \cite{RW1}
relates to a different kind of link homology, namely to a family of bigraded homology groups for links in $\R^3$ that has $sl(n)$ specializations of the HOMFLYPT polynomial as its Euler characteristic. Their formula allows to build equivariant $sl(n)$ homology for 
links with components labeled by arbitrary fundamental representations, starting 
with the values of closed $sl(n)$-foams. As shown by Ehrig, Tubbenhauer, and Wedrich 
\cite{FunctorialitySLN}, it also leads to a proof of the full functoriality of these 
$sl(n)$ link homologies. 

In this paper we specialize and extend the 
Robert--Wagner formula to provide a combinatorial approach to Kronheimer--Mrowka 
homology theory for planar graphs
and define an equivariant combinatorial 
version of the theory. In our combinatorial definition there are no complexes present, and homology (or state space) of a graph is given as the quotient of a free $R$-module by the kernel of a bilinear 
form on it, for a certain commutative ring $R$. 

One considers cobordisms, also called foams, in $\R^3$ between 
trivalent planar graphs, or webs. A closed foam $F$ in $\R^3$ is a cobordism from the empty 
web to itself. Specializing the Robert--Wagner formula to three colors and extending it to 
unoriented graphs and cobordisms requires reducing coefficients modulo two, as in the 
Kronheimer--Mrowka framework, and working over the two-element field $\kk\simeq \mathbb{F}_2$. 
A portion of the orientability property is retained, since foam 
3-colorings that we consider give rise to orientable surfaces when facets colored by 
any one out of the three colors are dropped from $F$. Evaluation of closed foams 
requires extending the ground field $\kk$ to the ring $R=\kk[E_1,E_2,E_3]$ of symmetric polynomials in 
three variables $X_1,X_2,X_3$ over $\kk$, producing what is usually called an equivariant theory (here $E_1, E_2, E_3$ are the elementary symmetric functions in $X_1, X_2, X_3$). 
In the absence 
of geometric interpretation, equivariance refers to  
the homology of the empty web being isomorphic to the equivariant  $SO(3)$ cohomology of a point 
(with coefficients in $\kk$), and to similar isomorphisms for the simplest planar graphs. 

In the equivariant theory, closed foams evaluate to elements of $R$, via the formula (\ref{eq:eval}) in Section~\ref{subsec:pf_eval}. This formula and many of its implications can be written in greater generality, for pre-foams. 
Pre-foams are compact two-dimensional CW-complexes with points that can have neighborhoods of the three types shown in Figure~\ref{fig:3localmodel}. 

Foams are pre-foams that are equipped with an embedding into $\R^3$. We consider colorings of facets of a pre-foam into three colors such that along any singular edge all three colors meet (pre-admissible colorings). We then single out a class of \emph{admissible} pre-foams, with the condition that bicolored surfaces for each 
pre-admissible coloring are orientable. The pre-foam underlying a foam is always admissible. 

We define a version of Robert--Wagner evaluation 
formula, our formula (\ref{eq:eval}),
to assign a rational symmetric function $\brak{\Gamma}$ in variables $X_1, X_2, X_3$ to a pre-foam $\Gamma$. We 
show that, for admissible pre-foams, $\brak{\Gamma}$ is a polynomial, so takes values in $R$. 

In Section~\ref{sec:rel-btwn-ev} we derive a number of local skein formulas for evaluation of pre-foams and foams. 

Ring $R$ of symmetric functions in three variables, where evaluations take values, maps 
surjectively onto the field $\kk$, by killing 
everything in $R$ in positive degrees. This leads to the corresponding evaluation $\brak{F}_{\kk}$ of closed foams, now taking values in $\kk$. 
In Section~\ref{sec:combinatorialKM} we compare Kronheimer--Mrowka's conjectural algorithm for pre-foam evaluation~\cite{KM1} with the evaluation $\brak{F}_{\kk}$ and show that the two evaluations coincide on foams (that is, on pre-foams embedded in $\R^3$). Consequently, we prove Kronheimer--Mrowka Conjecture~8.9 in~\cite{KM1}, restricted to the case of foams, see Theorem~\ref{thm:JflatFoams}. We also give an example of a pre-foam not embeddable in $\R^3$ for which the algorithm does not result in a well-defined value, implying that the conjecture does not hold in full generality for all pre-foams. Restriction to foams is a very natural assumption, and Kronheimer--Mrowka essentially restrict to this case in the discussion following their conjecture.      
    
We use the foam evaluation formula (\ref{eq:eval}) to associate a graded $R$-module $\brak{\Gamma}$, called \emph{the state space of}
$\Gamma$, to a planar trivalent graph $\Gamma$. 
This is a standard construction, see  \cite{BHMV,SL3,RW1}, where generators of $\brak{\Gamma}$ are all possible foams from the empty graph to $\Gamma$, and a linear combination of generators is zero iff composing 
these generators with any foam from $\Gamma$ to the empty graph, evaluating resulting closed foams to elements of $R$, and forming the corresponding linear combination always produces zero.  

In Proposition~\ref{prop:fin-gen} of Section~\ref{sec:fin-gen} we prove that $R$-module $\brak{\Gamma}$ is finitely-generated 
for any planar graph $\Gamma$. 
In Section~\ref{sub:dir_sum} we derive direct sum decompositions for 
$\brak{\Gamma}$ to simplify it when a planar graph $\Gamma$ has a facet with at most four edges. In particular, this allows to decompose $\brak{\Gamma}$ for any bipartite graph, showing that in this case it coincides with the state space of Mackaay--Vaz~\cite{MaV} equivariant $sl(3)$ link homology modulo two. 

It's unclear whether $\brak{\Gamma}$ is a free graded module for any $\Gamma$. We consider base changes, that is, homomorphisms $\psi: R\lra S$ from $R$ to rings $S$. The bilinear form that defines $\brak{\Gamma}$ is modified to get a bilinear form over $S$, such that the quotient by the kernel is the $S$-state space of $\Gamma$, denoted $\brak{\Gamma}_S$. 

Suitable base changes give rise to simpler theories. In Section~\ref{sec:facet_neg} we consider base change 
$\psi_{\D}: R \lra R[\D^{-1}]$ given by inverting the 
discriminant $\D = E_1 E_2 + E_3 \in R$ of the 
polynomial $X^3 + E_1 X^2 + E_2 X + E_3$. 

Kronheimer--Mrowka 4-periodic complex~\cite[Diagram (21)]{KM1} exists in our set-up as well, see Section~\ref{sec:four_end}. 
We prove that the base change $\psi_{\D}$ makes 
this complex exact and that 
the corresponding state spaces $\lG_{\D}$, which 
are $R[\D^{-1}]$-modules, are projective of 
rank equal to the number of Tait colorings of 
web (planar trivalent graph) $\Gamma$.

A naive conjecture for state spaces $\lG$ is that they are free graded $R$-modules
of rank equal to the number of Tait colorings of $\Gamma$, but the authors are not confident enough to propose it, and have not verified it even for the dodecahedral graph, the one-skeleton of the dodecahedron stretched out on the plane. 

\vspace{0.1in}

{\bf Acknowledgments}  M.K. was partially supported by NSF grants DMS-1406065, DMS-1664240, and DMS-1807425. L.-H.R. was supported by NCCR SwissMAP, funded by the Swiss National Science Foundation. We are grateful to the Simons Center for Geometry and Physics for hosting Categorification in Mathematical Physics workshop in April 2018 which started this collaboration.

\section{Pre-foams and their evaluations}
\label{sec:prefoams}

\subsection{Pre-foams and colorings} \ 

An (open) tripod $T$ is a topological space obtained by identifying three copies of the 
semi-open interval $[0,1)$ along the three $0$ points. A tripod has a singular point 
and three intervals emanating from it. The subspace $\{0\}\times (0,1)$ in the direct product $T\times (0,1)$ of a tripod and an open interval is called 
a \emph{seam} or a \emph{singular edge} of the product. 

\begin{defn}
A (closed) \emph{pre-foam} $F$ is a compact 2-dimensional CW-complex with a PL-structure such that each 
point has an open  neighborhood that is either an open disc, the product of 
a tripod and an open interval, or 
the cone over 1-skeleton of a tetrahedron, see Figure~\ref{fig:3localmodel}. 
\end{defn}

\begin{figure}[ht]
\centering
\begin{tikzpicture}
\input{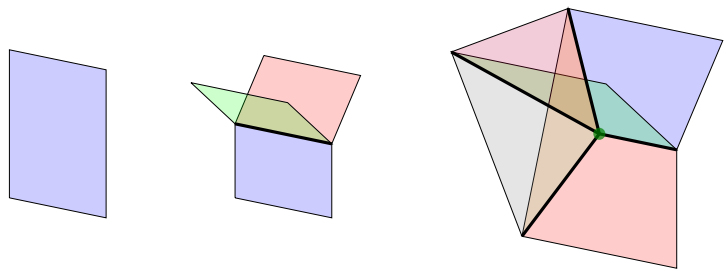}
\end{tikzpicture}
\caption{The three local models of a pre-foam. Facets are colored just  to 
make the figure clearer. The seams are depicted in bold black. Standard neighborhoods of a smooth point, a 
seam point, and a seam vertex are depicted from left to right.}
\label{fig:3localmodel}
\end{figure}

We call points of the first type \emph{regular} or \emph{smooth} points of $F$, 
points of the second type \emph{seam points}, and points of the third type \emph{seam 
vertices}. 

The subspace of seam points and seam vertices in $F$ is denoted $s(F)$. 
It's a topological space and can also be thought of as a four-valent graph  
that may contain circles (seams that close on themselves). The vertices of the graph 
$s(F)$ are the seam vertices of $F$. We use $s(F)$ to denote both this subspace of $F$ 
and the corresponding graph. The set of seam vertices is denoted $v(F)$, and 
connected components of $s(F)\setminus v(F)$ are called \emph{seams}. Each seam 
is homeomorphic to either an open interval or a circle. 

The space 
$F\setminus s(F)$ is an open surface and we denote by $f(F)$ the set of 
its connected  components, also called \emph{facets} of $F$. It is a finite set. 
A connected component may be a compact 
surface (which then does not bound any seams) or a non-compact surface, 
which is a facet bounding one or more seams. 

Standard neighborhood $N(v)$ of a seam vertex $v$ can be visualized as the cone over the 1-skeleton 
of a tetrahedron. The six connected components of $N(v)\setminus s(F)$ are six 
portions of facets of $F$ that contain $v$ in the closure. We call these six components 
the \emph{corners} at $v$. 

Closed pre-foam $F$ may be decorated by a finite collections of points (dots). 
Dots can float freely on any facet of $F$ but cannot cross seams or enter seam vertices. 
A collection of several dots on a facet may be denoted by a single dot with the label the number 
of dots it represents.  

A \emph{coloring} $c$ of $F$ is a map $f(F)\lra \{1,2,3\}$, that is, an assignment of a number 
from $1$ to $3$ to each facet of $F$. 

A coloring $c$ is called \emph{pre-admissible} if the three facets at each seam of $F$ are 
colored by three distinct colors. A \emph{colored pre-foam} is a pair $(F,c)$ where $F$ is a pre-foam and $c$ a  
pre-admissible coloring of $F$. A facet of a colored pre-foam whose color is $i$ is called an \emph{$i$-facet}.

\begin{figure}[ht]
	\centering
    	\begin{tikzpicture}
    		\input{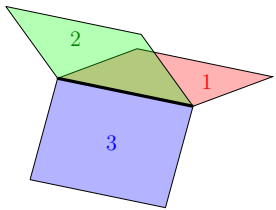}
    	\end{tikzpicture}
        \caption{A pre-admissible coloring in a neighborhood of a seam point. All three colors appear.}\label{fig:colatseam}
\end{figure}

To give an example of a pre-foam without pre-admissible colorings, we can take a tripod times 
an interval, $T\times [0,1]$, and select a homeomorphism $h$ between $T\times \{0\}$ and 
$T\times \{ 1\}$ that nontrivially permutes the three legs of $T$ using either a transposition 
or a 3-cycle. Gluing $T\times [0,1]$ onto itself via this homeomorphism produces 
a surface $S'$ with a singular circle. Adding either two disks (in case $h$ is a transposition
of the three edges) 
or one disk (when $h$ is a  3-cycle) to $S'$ produces a pre-foam $S$ with a singular circle 
and either two or one facets, homeomorphic to open disks. This pre-foam has no 
pre-admissible colorings. 

If $v$ is a seam vertex of $F$ and $c$ a pre-admissible coloring, the six 
corners at $v$ are colored by three colors so that opposite 
corners carry identical colors. 

\begin{figure}[ht]
\centering
\begin{tikzpicture}
\input{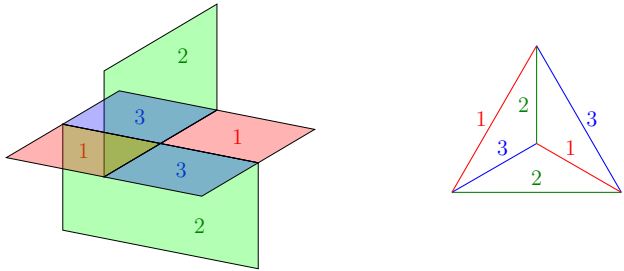}
\end{tikzpicture}
\caption{On the left: a pre-admissible coloring of a foam in the neighborhood of a seam vertex. On the right: the induced Tait coloring of the graph obtained by intersecting the foam with a small 2-sphere centered at the seam vertex. A Tait coloring of a graph is a 3-coloring of edges so 
that no two edges of the same color share a vertex. 
}
\label{fig:colatsingpoint}
\end{figure}

For each $1\le i < j \le 3$ denote by $F_{ij}(c)$ the closure of the union of $i$- 
and $j$-colored facets of $(F,c)$. Also let 
$F_{ji}(c) = F_{ij}(c)$. 

\begin{prop} The set $F_{ij}(c)$ is a closed compact surface that contains the 
graph $s(F)$. 
\end{prop} 

\begin{proof}
Each seam of $F$ is adjacent to an $i$-facet and a $j$-facet, 
so that $s(F)\subset F_{ij}(c)$ and each seam point has a neighborhood in $F_{ij}(c)$ homeomorphic to $\R^2$. Points of $F_{ij}(c)$ that are not seam vertices in $F$ also  
have neighborhoods homeomorphic to $\R^2$. Likewise, each seam vertex $v$ of $F$, 
necessarily in $F_{ij}(c)$, has a neighborhood in  $F_{ij}(c)$ homeomorphic 
to $\R^2$, since a pre-admissible coloring of a neighborhood of a seam vertex is unique 
up to permutation of the colors, 
and the intersection of $F_{ij}(c)$ with the standard open neighborhood of $v$ is 
homeomorphic to $\R^2$. 
\end{proof}

\begin{figure}[ht]
	\begin{tikzpicture}
		\input{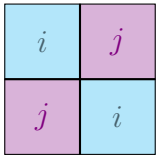}
	\end{tikzpicture}
    \caption{Neighborhood of a seam vertex in the surface $F_{ij}(c)$.}
    \label{fig:singinFij}
\end{figure}

A coloring $c$ is called \emph{admissible} if the surfaces $F_{ij}(c)$ are orientable 
for all $1\le i< j\le 3$. We denote by $\adm(F)$ the set of admissible colorings of $F$. 

If $F$ has an admissible coloring and a connected component which is a surface $S$, 
then $S$ is orientable. If a pre-foam $F$ is an unorientable surface, then $F$ has 
pre-admissible colorings but no admissible coloring. 

\begin{prop} 
\label{prop:admis2bipartite}
If $F$ has an admissible coloring, the graph $s(F)$ is bipartite. 
\end{prop} 

\begin{proof}
In our convention, a bipartite graph may have circles (edges that close upon themselves). To prove that $s(F)$ is bipartite, it is enough to prove that any cycle in $s(F)$ has an even number of vertices. 
Let $c$ be an admissible coloring of $F$, and $C:=(v_1,v_2,\dots, v_{n})$ be a cycle in $C(F)$. The cycle $C$ consists of edges 
$e_1, \dots, e_n$, with $e_i$ connecting vertices $v_i$ and $v_{i+1}$, indices taken modulo $n$. 

We consider a tubular neighborhood $N$ of $C$ in the pre-foam. 
The boundary $\partial N$ is a trivalent graph $G$ with a Tait coloring $c_G$ induced by $c$. Along an edge of $C$, the graph $G$ consists of three \enquote{parallel} edges colored by three distinct colors. The structure of $G$ in a neighborhood of a seam vertex $v_i \in C$ is depicted in Figure~\ref{fig:cyclenbhvertex}. As they approach $v_i$, two out of the three \enquote{parallel} edges terminate in trivalent vertices that are connected by an edge (an arc in the lower half of the tube in Figure~\ref{fig:cyclenbhvertex}). Then two new edges 
start out at these vertices to continue in parallel with the cycle $C$. 

\begin{figure}[ht]
\centering
\begin{tikzpicture}
\input{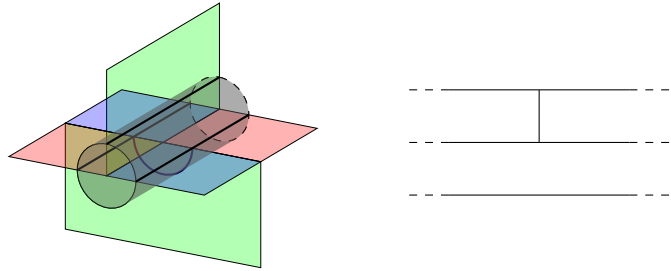}
\end{tikzpicture}
\caption{On the left: intersection of $N$ with the pre-foam $F$ in a neighborhood of a vertex of $C$. On the right: the graph $G$ in a neighborhood of a vertex of $C$.}\label{fig:cyclenbhvertex}
\end{figure}

With the coloring $c$ and the induced Tait coloring $c_G$ of $G$ fixed, 
the vertices of the cycle $C$ can be partitioned into three types, 
determined by the color of the edge that continues uninterrupted past the vertex, see Figure~\ref{fig:3types}. The edge connecting two 
vertices of $G$ corresponding to the vertex in 
the cycle is colored by the same color as the uninterrupted edge. 

\begin{figure}[ht]
\centering
\begin{tikzpicture}[xscale= 1.2, yscale=0.7]
\begin{scope}
  \coordinate (A1) at (-1,-1);
  \coordinate (A2) at (-1, 0);
  \coordinate (A3) at (-1, 1);
  \coordinate (B1) at ( 1,-1);
  \coordinate (B2) at ( 1, 0);
  \coordinate (B3) at ( 1, 1);
  \coordinate (C2) at ( 0, 0);
  \coordinate (C3) at ( 0, 1);
  \draw[red]                    (A1) -- (B1) node[pos= 0.5, below, scale = 0.7, red] {$1$};
  \draw[dashed, red]            (A1) -- ++(-0.5,0);
  \draw[dashed, red]            (B1) -- ++( 0.5,0);
  \draw[green!50!black]         (A2) -- (C2) node[pos= 0.5, below, scale = 0.7,green!50!black] {$2$};
  \draw[dashed, green!50!black] (A2) -- ++(-0.5,0);
  \draw[blue]                   (A3) -- (C3) node[pos= 0.5, above,scale = 0.7, blue] {$3$};
  \draw[dashed, blue]           (A3) -- ++(-0.5,0);
  \draw[red]                    (C2) -- (C3) node[pos= 0.5, left ,scale = 0.7, red] {$1$};
  \draw[green!50!black]         (C3) -- (B3) node[pos= 0.5, above,scale = 0.7, green!50!black] {$2$};
  \draw[dashed, green!50!black] (B3) -- ++( 0.5,0);
  \draw[blue]                   (C2) -- (B2) node[pos= 0.5, below, scale = 0.7, blue] {$3$};
  \draw[dashed, blue]           (B2) -- ++( 0.5,0);
\end{scope}

\begin{scope}[xshift = 4cm]
  \coordinate (A1) at (-1,-1);
  \coordinate (A2) at (-1, 0);
  \coordinate (A3) at (-1, 1);
  \coordinate (B1) at ( 1,-1);
  \coordinate (B2) at ( 1, 0);
  \coordinate (B3) at ( 1, 1);
  \coordinate (C2) at ( 0, 0);
  \coordinate (C3) at ( 0, 1);
  \draw[green!50!black]         (A1) -- (B1) node[pos= 0.5, below, scale = 0.7, green!50!black] {$2$};
  \draw[dashed, green!50!black] (A1) -- ++(-0.5,0);
  \draw[dashed, green!50!black] (B1) -- ++( 0.5,0);
  \draw[blue]                   (A2) -- (C2) node[pos= 0.5, below, scale = 0.7,blue] {$3$};
  \draw[dashed, blue]           (A2) -- ++(-0.5,0);
  \draw[red]                    (A3) -- (C3) node[pos= 0.5, above,scale = 0.7, red] {$1$};
  \draw[dashed, red]            (A3) -- ++(-0.5,0);
  \draw[green!50!black]         (C2) -- (C3) node[pos= 0.5, left ,scale = 0.7, green!50!black] {$2$};
  \draw[blue]                   (C3) -- (B3) node[pos= 0.5, above,scale = 0.7, blue] {$3$};
  \draw[dashed, blue]           (B3) -- ++( 0.5,0);
  \draw[red]                    (C2) -- (B2) node[pos= 0.5, below, scale = 0.7, red] {$1$};
  \draw[dashed, red]            (B2) -- ++( 0.5,0);
\end{scope}

\begin{scope}[xshift = 8cm]
  \coordinate (A1) at (-1,-1);
  \coordinate (A2) at (-1, 0);
  \coordinate (A3) at (-1, 1);
  \coordinate (B1) at ( 1,-1);
  \coordinate (B2) at ( 1, 0);
  \coordinate (B3) at ( 1, 1);
  \coordinate (C2) at ( 0, 0);
  \coordinate (C3) at ( 0, 1);
  \draw[blue]                   (A1) -- (B1) node[pos= 0.5, below, scale = 0.7, blue] {$3$};
  \draw[dashed, blue]           (A1) -- ++(-0.5,0);
  \draw[dashed, blue]           (B1) -- ++( 0.5,0);
  \draw[red]                    (A2) -- (C2) node[pos= 0.5, below, scale = 0.7,red] {$1$};
  \draw[dashed, red]            (A2) -- ++(-0.5,0);
  \draw[green!50!black]         (A3) -- (C3) node[pos= 0.5, above,scale = 0.7, green!50!black] {$2$};
  \draw[dashed, green!50!black] (A3) -- ++(-0.5,0);
  \draw[blue]                   (C2) -- (C3) node[pos= 0.5, left ,scale = 0.7, blue] {$3$};
  \draw[red]                    (C3) -- (B3) node[pos= 0.5, above,scale = 0.7, red] {$1$};
  \draw[dashed, red]            (B3) -- ++( 0.5,0);
  \draw[green!50!black]         (C2) -- (B2) node[pos= 0.5, below, scale = 0.7, green!50!black] {$3$};
  \draw[dashed, green!50!black] (B2) -- ++( 0.5,0);
\end{scope}
\end{tikzpicture}
\caption{The three possible induced colorings of $G$ in a neighborhood of a seam vertex $v$ of $C$. This determines the type of the seam vertex $v$ (with respect to the coloring $c$ and the cycle $C$). From left to right, the seam vertex $v$ has type $1$, $2$ and $3$.}
\label{fig:3types}
\end{figure}
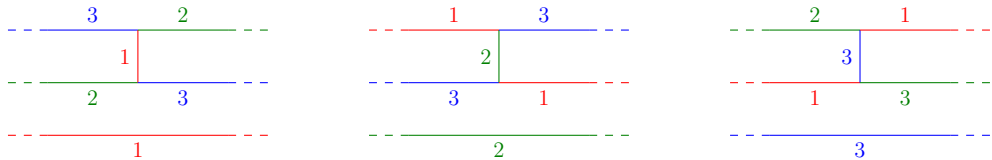

\begin{figure}[ht]
\centering
\begin{tikzpicture}
\begin{scope}
\input{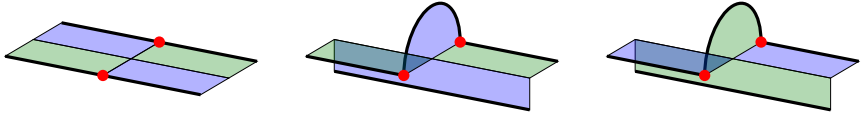}
\end{scope}

\end{tikzpicture}
\caption{The surface $\Sigma$ in the neighborhood of vertices of type 1, 2, 3 going from left to right. The boundary of $\Sigma$ is depicted in bold black.}
\label{fig:3typescol}
\end{figure}

We consider the intersection $\Sigma=F_{23}(c)\cap N$. This surface deformation retracts onto the cycle $C$, so it is either an annulus or a Möbius band. Since $c$ is admissible, it is an annulus. Its boundary has two connected components, and both  components consist of a collection of edges attached along some vertices (depicted in red in Figure~\ref{fig:3typescol}). 

To each seam vertex of type $2$ and $3$ there correspond two vertices on the boundary of $\Sigma$ that belong to same boundary component, see Figure~\ref{fig:3typescol}. 
To each seam vertex in $C$ of type $1$ there correspond two vertices on the boundary of $\Sigma$: one on each of the boundary components, see the leftmost picture on Figure~\ref{fig:3typescol}. Since on the boundary components the colors of the edges alternate, each of the two components has an even number of vertices. This proves that $C$ has an even number of seam vertices of type $1$. 

Permuting the colors in the previous argument shows that there is an even number of seam vertices in $C$ of type $2$  (resp. of type $3$). Hence $C$ has an even number of vertices.
\end{proof}

\begin{defn} A pre-foam $F$ is called \emph{admissible} if any pre-admissible coloring of $F$ 
is admissible. 
\end{defn} 

\begin{rmk}
The empty pre-foam $\emptyset$ has a unique coloring, which is admissible. 
\end{rmk}

\begin{rmk} \label{rmk:Klein}
Here is an example of a pre-foam which has an admissible coloring and a pre-admissible but 
not admissible coloring. Glue four disks to four parallel disjoint loops in a Klein bottle so as
to form a pre-foam $F$ with $8$ facets, including four disks and four annuli. The closure of the union 
of the four annuli is the original Klein bottle. 

The pre-foam $F$ has a pre-admissible but not admissible coloring, given by coloring the disks with $1$ and the remaining four annuli by $2$ and $3$ alternatively. For an admissible coloring, color the disks by $3$, $3$, $1$ and $1$ in this order and complete by coloring the annuli with $1$, $2$, $3$ and $2$, starting with the annuli adjoint 
to the 3-colored disks. This is depicted in Figure~\ref{fig:kleinbottle}.
\begin{figure}
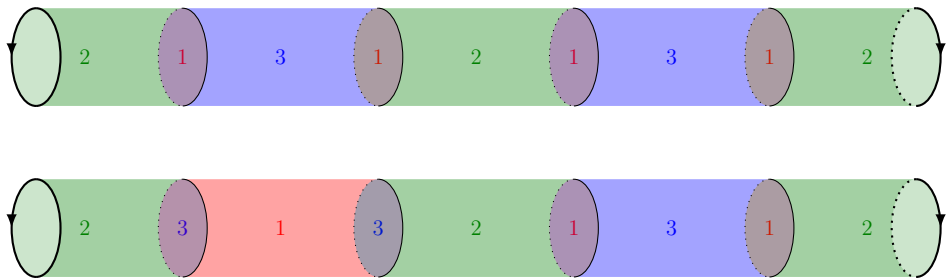

\centering
\tikz[scale = 1.3]{\begin{scope}[xscale = 0.5, yscale = 0.5]
  \begin{scope}
  \begin{scope}[black!50!green]
    \fill[opacity = 0.2] (-1,1) -- (2,1) arc (90:270:0.5 and 1) -- (-1, -1) arc (270:90:0.5 and 1);
    \draw[black, thick, ->- ] (-1,1) arc (90:270:0.5 and 1);
    \fill[opacity = 0.2] (-1,1) -- (2,1) arc (90:-90:0.5 and 1) --  (-1, -1) arc (-90:90:0.5 and 1);
    \draw[thick, black] (-1,1) arc (90:-90:0.5 and 1);
    \node[scale =0.7] at (0,0) {$2$};
    \fill[red, opacity = 0.2] (2,0) circle (0.5 and 1) node [scale = 0.7,red, opacity =1] {$1$};
  \end{scope}
  \begin{scope}[blue, xshift = 4cm]
    \fill[opacity = 0.2] (-2,1) -- (2,1) arc (90:270:0.5 and 1) -- (-2, -1) arc (270:90:0.5 and 1);
    \draw[ black, dotted] (-2,1) arc (90:270:0.5 and 1);
    \fill[opacity = 0.2] (-2,1) -- (2,1) arc (90:-90:0.5 and 1) -- (-2, -1) arc (-90:90:0.5 and 1);
    \draw[ black] (-2,1) arc (90:-90:0.5 and 1);
    \node[scale =0.7] at (0,0) {$3$};
    \fill[red, opacity = 0.2] (2,0) circle (0.5 and 1) node [scale = 0.7,red, opacity =1] {$1$};
  \end{scope}
  \begin{scope}[black!50!green,  xshift = 8cm]
    \fill[opacity = 0.2] (-2,1) -- (2,1) arc (90:270:0.5 and 1) -- (-2, -1) arc (270:90:0.5 and 1);
    \draw[ black, dotted] (-2,1) arc (90:270:0.5 and 1);
    \fill[opacity = 0.2] (-2,1) -- (2,1) arc (90:-90:0.5 and 1) -- (-2, -1) arc (-90:90:0.5 and 1);
    \draw[ black] (-2,1) arc (90:-90:0.5 and 1);
    \node[scale =0.7] at (0,0) {$2$};
    \fill[red, opacity = 0.2] (2,0) circle (0.5 and 1) node [scale = 0.7,red, opacity =1] {$1$};
  \end{scope}
  \begin{scope}[blue,  xshift = 12cm]
    \fill[opacity = 0.2] (-2,1) -- (2,1) arc (90:270:0.5 and 1) -- (-2, -1) arc (270:90:0.5 and 1);
    \draw[ black, dotted] (-2,1) arc (90:270:0.5 and 1);
    \fill[opacity = 0.2] (-2,1) -- (2,1) arc (90:-90:0.5 and 1) -- (-2, -1) arc (-90:90:0.5 and 1);
    \draw[ black] (-2,1) arc (90:-90:0.5 and 1);
    \node[scale =0.7] at (0,0) {$3$};
    \fill[red, opacity = 0.2] (2,0) circle (0.5 and 1) node [scale = 0.7,red, opacity =1] {$1$};
  \end{scope}
  \begin{scope}[black!50!green,  xshift = 16cm]
    \fill[opacity = 0.2] (-2,1) -- (1,1) arc (90:270:0.5 and 1) -- (-2, -1) arc (270:90:0.5 and 1);
    \draw[ black, dotted] (-2,1) arc (90:270:0.5 and 1);
    \fill[opacity = 0.2] (-2,1) -- (1,1) arc (90:-90:0.5 and 1) -- (-2, -1) arc (-90:90:0.5 and 1);
    \draw[ black] (-2,1) arc (90:-90:0.5 and 1);
    \node[scale =0.7] at (0,0) {$2$};
    \draw[thick, black, ->-] (1,1) arc (90:-90:0.5 and 1);
    \draw[thick, black, dotted] (1,1) arc (90:270:0.5 and 1);
  \end{scope}
  \end{scope}

 \begin{scope}[yshift = -3.5cm ]
  \begin{scope}[black!50!green]
    \fill[opacity = 0.2] (-1,1) -- (2,1) arc (90:270:0.5 and 1) -- (-1, -1) arc (270:90:0.5 and 1);
    \draw[black, thick, ->- ] (-1,1) arc (90:270:0.5 and 1);
    \fill[opacity = 0.2] (-1,1) -- (2,1) arc (90:-90:0.5 and 1) --  (-1, -1) arc (-90:90:0.5 and 1);
    \draw[thick, black] (-1,1) arc (90:-90:0.5 and 1);
    \node[scale =0.7] at (0,0) {$2$};
    \fill[blue, opacity = 0.2] (2,0) circle (0.5 and 1) node [scale = 0.7,blue, opacity =1] {$3$};
  \end{scope}
  \begin{scope}[red, xshift = 4cm]
    \fill[opacity = 0.2] (-2,1) -- (2,1) arc (90:270:0.5 and 1) -- (-2, -1) arc (270:90:0.5 and 1);
    \draw[ black, dotted] (-2,1) arc (90:270:0.5 and 1);
    \fill[opacity = 0.2] (-2,1) -- (2,1) arc (90:-90:0.5 and 1) -- (-2, -1) arc (-90:90:0.5 and 1);
    \draw[ black] (-2,1) arc (90:-90:0.5 and 1);
    \node[scale =0.7] at (0,0) {$1$};
    \fill[blue, opacity = 0.2] (2,0) circle (0.5 and 1) node [scale = 0.7,blue, opacity =1] {$3$};
  \end{scope}
  \begin{scope}[black!50!green,  xshift = 8cm]
    \fill[opacity = 0.2] (-2,1) -- (2,1) arc (90:270:0.5 and 1) -- (-2, -1) arc (270:90:0.5 and 1);
    \draw[ black, dotted] (-2,1) arc (90:270:0.5 and 1);
    \fill[opacity = 0.2] (-2,1) -- (2,1) arc (90:-90:0.5 and 1) -- (-2, -1) arc (-90:90:0.5 and 1);
    \draw[ black] (-2,1) arc (90:-90:0.5 and 1);
    \node[scale =0.7] at (0,0) {$2$};
    \fill[red, opacity = 0.2] (2,0) circle (0.5 and 1) node [scale = 0.7,red, opacity =1] {$1$};
  \end{scope}
  \begin{scope}[blue,  xshift = 12cm]
    \fill[opacity = 0.2] (-2,1) -- (2,1) arc (90:270:0.5 and 1) -- (-2, -1) arc (270:90:0.5 and 1);
    \draw[ black, dotted] (-2,1) arc (90:270:0.5 and 1);
    \fill[opacity = 0.2] (-2,1) -- (2,1) arc (90:-90:0.5 and 1) -- (-2, -1) arc (-90:90:0.5 and 1);
    \draw[ black] (-2,1) arc (90:-90:0.5 and 1);
    \node[scale =0.7] at (0,0) {$3$};
    \fill[red, opacity = 0.2] (2,0) circle (0.5 and 1) node [scale = 0.7,red, opacity =1] {$1$};
  \end{scope}
  \begin{scope}[black!50!green,  xshift = 16cm]
    \fill[opacity = 0.2] (-2,1) -- (1,1) arc (90:270:0.5 and 1) -- (-2, -1) arc (270:90:0.5 and 1);
    \draw[ black, dotted] (-2,1) arc (90:270:0.5 and 1);
    \fill[opacity = 0.2] (-2,1) -- (1,1) arc (90:-90:0.5 and 1) -- (-2, -1) arc (-90:90:0.5 and 1);
    \draw[ black] (-2,1) arc (90:-90:0.5 and 1);
    \node[scale =0.7] at (0,0) {$2$};
    \draw[thick, black, ->-] (1,1) arc (90:-90:0.5 and 1);
    \draw[thick, black, dotted] (1,1) arc (90:270:0.5 and 1);
  \end{scope}
  \end{scope}

\end{scope}}
\caption{On the top: a pre-admissible but not admissible coloring of the pre-foam $F$, on the bottom: an admissible coloring of the pre-foam $F$. On top and on bottom, the two bold oriented circles are meant to identified.
}\label{fig:kleinbottle}
\end{figure}
\end{rmk}

In what follows, we will be mainly interested in admissible pre-foams. It is worth noticing that the seam graph of an admissible pre-foam has an even number of vertices, since it is a 4-regular bipartite graph. 

Denote by $|Y|$ the cardinality of a set $Y$, so that $|d(F)|$ and $|v(F)|$ is the number of dots, respectively the number of seam vertices of a pre-foam $F$. 

\begin{defn} \label{def:degree}
The \emph{degree} $\deg(F)$ of a pre-foam $F$ is an integer given by 
\begin{align*}
\deg(F) &= 2 \ |d(F)| - 2 \ \chi(F) - \chi(s(F))  \\
 &= 2\ |d(F)| - 2 \sum_{f \in f(F)} \chi(f) + 3\ |v(F)|.
\end{align*}
\end{defn}
$\chi(f)$ is the Euler characteristic of the open facet $f$. 
The second expression follows from the identities
\begin{align*}
 \chi(F) & = \sum_{f \in f(F)} \chi(f) + \chi(s(F)), \\
\chi(s(F)) &  = - |v(F)|,
\end{align*}
since $s(F)$ has twice as many (non circular) edges as vertices.

\begin{rmk}\label{rmk:degree}
Suppose that  a foam $F$ carries no dots 
and admits a pre-admissible coloring $c$. Then 
\begin{align} \label{eq:degfromsurfaces}
\deg(F)= - \left( \chi(F_{12}(c)) + \chi(F_{13}(c)) + \chi(F_{23}(c)) \right).
\end{align}
Indeed, we have
\begin{align*}
\chi(F_{12}(c))= \chi(s(F)) + \sum_{\substack{f\in f(F) \\ \textrm{$f$ colored by $1$ or $2$}}} \chi(f) ,
\end{align*}
likewise for $\chi(F_{13}(c))$ and $\chi(F_{23}(c))$. The identity~(\ref{eq:degfromsurfaces}) follows.
 
Note that $\deg(F)$ is even if and only if $F$ has an even number of seam vertices. In particular, 
 if $F$ has an admissible coloring then $\deg(F) \in 2\Z$. 

For an admissible foam $F$, 
\begin{align} \label{eq:degfromsurfaces2}
\deg(F)= 2 \ |d(F)| - \left( \chi(F_{12}(c)) + \chi(F_{13}(c)) + \chi(F_{23}(c)) \right),
\end{align}
for any admissible coloring $c$. 
\end{rmk}

\subsection{Kempe moves for pre-foam colorings}\ 

For a given admissible pre-foam $F$, the group $S_3$ acts naturally on $\adm(F)$ by permuting 
the colorings. As we will now see, there are other, more local, coloring modifications. In this subsection, $i$, $j$ and $k$ denote the three elements of $\{1,2,3\}$, but not necessarily on this order.

\begin{defn}
\label{def:kempemove}
Let us consider an admissible pre-foam $F$ and a coloring $c$ of $F$. The surface $F_{jk}(c)$ may have several connected components, let $\Sigma$ be one of them. We can define a coloring $c'$ of $F$ by swapping the colors $j$ and $k$ in all the facets of $F$ which are contained in $\Sigma$. 
We say that $c$ and $c'$ are related by a \emph{$jk$-Kempe move along $\Sigma$}. 
\end{defn}

\begin{rmk} 
\begin{enumerate}
\item 	One could define Kempe moves for non-admissible pre-foams, however, when performing such a move, 
one may not remain in the set of admissible colorings, see Remark~\ref{rmk:Klein} and Figure~\ref{fig:kleinbottle} for an example 
of an admissible coloring related by a Kempe move to a pre-admissible but not admissible coloring. 
\item Note that a $jk$-Kempe move does not change the set of facets of $F$ colored by $i$.
\end{enumerate}
\end{rmk}

\begin{lemma}\label{lem:colKempe}
	Let $F$ be an admissible pre-foam, and $S$ a subset of facets of $F$. Consider the set $\adm(F, i=S)$ of all colorings of $F$ such that the set of facets colored by $i$ is exactly $S$. Suppose that $\adm(F, i=S)$ is not empty. Then $F\setminus \bigcup_{f\in S} f$ is a surface $\Sigma$ with some number $n$ of connected components 
(necessarily orientable). The set  $\adm(F, i=S)$ contains $2^n$ colorings, and they are all related 
to one another by finite sequences of $jk$-Kempe moves along connected components of $\Sigma$.
\end{lemma}

\begin{proof}
That $\Sigma = F\setminus \bigcup_{f\in S} f$ is a surface is clear, since if $c$ is an element of $\adm(F, i=S)$ then $F\setminus \bigcup_{f\in S} f= F_{jk}(c)$. Let us denote by $\Sigma_1, \dots \Sigma_n$ the connected components of $\Sigma$ and for each $a$ in $\{1,\dots, n\}$, choose a facet $f_a$ contained in $\Sigma_a$. A coloring in $\adm(F, i=S)$ is completely determined by its value on the facets $f_a$. Since a $jk$-Kempe move along $\Sigma_a$ changes the color of $f_a$, there are precisely $2^n$ colorings in $\adm(F, i=S)$ and they relate to one another by finite sequences of $jk$-Kempe moves along the $\Sigma_a$'s.
\end{proof}

\begin{lemma}\label{lem:ell}
	Let $F$ be an admissible pre-foam, $c$ an admissible coloring, and $\Sigma$ a connected component of $F_{jk}(c)$. Denote by $c'$ the coloring obtained from $c$ by the $jk$-Kempe move along $\Sigma$. Then:
    \begin{enumerate}
    \item The surfaces $F_{jk}(c)$ and $F_{jk}(c')$ are equal.
   	\item The surface $F_{ij}(c')$ is the closure in $F$ of the symmetric difference of $F_{ij}(c)$ and $\Sigma$. The surface $F_{ik}(c')$ is the closure in $F$ of the symmetric difference of $F_{ik}(c)$ and $\Sigma$. 
    \item\label{it:ellSigma} There exists an integer $\ell_\Sigma(c)$, such that 
    \begin{align*}
    \chi(F_{ij}(c')) = \chi(F_{ij}(c)) + \ell_\Sigma(c)	 \quad \textrm{and} \quad
    \chi(F_{ik}(c')) =  \chi(F_{ik}(c)) - \ell_\Sigma(c).
    \end{align*}
    Moreover,  $\ell_\Sigma(c)$ only depend on $\Sigma$ and on the restriction of $c$ to $\Sigma$. 
    \end{enumerate}
\end{lemma}

\begin{rmk}
  With the above notations, $\ell_{\Sigma}(c) = - \ell_{ \Sigma}(c')$. 
\end{rmk}

\begin{proof}
The only non-trivial point is $(\ref{it:ellSigma})$, which follows directly from formula (\ref{eq:degfromsurfaces2}) in Remark~\ref{rmk:degree}. Since $\chi(F_{jk}(c)) = \chi(F_{jk }(c'))$, equality 
\[\chi(F_{ij}(c)) + \chi(F_{ik}(c)) = 
    \chi(F_{ij}(c')) + \chi(F_{ik}(c'))\] 
holds.
That $\ell_{\Sigma}(c)$ depends only on $\Sigma$ and on the restriction of $c$ to $\Sigma$ follows, since the Euler characteristic of a surface can be computed locally.   
\end{proof}

\subsection{Pre-foam evaluation} 
\label{subsec:pf_eval}\ 

Let $R' = \kk[X_1,X_2,X_3]$ be the graded ring of polynomials in three variables 
with coefficients in $\kk$ and $\deg(X_i)=2$, $1\le i \le 3$. Denote by $R$ the subring of $R'$ that consists of 
all symmetric polynomials in $X_1, X_2, X_3$. Thus, $R \cong \kk[E_1,E_2,E_3]$, 
where 
\begin{align*}
 E_1 & =  X_1 + X_2 + X_3 , \\
 E_2 & =  X_1 X_2 + X_1 X_3 + X_2 X_3, \\
 E_3 & =  X_1 X_2 X_3 
 \end{align*} 
are the three elementary symmetric functions in  $X_1, X_2, X_3$. Our degree conventions imply  that $\deg(E_i)= 2i$ for $i=1,2,3$.  

Let 
\[R''=R'[(X_1+X_2)^{-1},(X_1+X_3)^{-1}, (X_2 + X_3)^{-1}].\]
This ring is obtained by inverting elements $X_i+X_j$ of $R$, for $1\le i < j \le 3$.

The ring $R''$ contains subrings $R''_{ij} = R[(X_i+X_k)^{-1}, (X_j+X_k)^{-1}]$
given by inverting two elements out of the above three, and not inverting $X_i+X_j$. 
\begin{lemma} 
\[ R' = R''_{12} \cap R''_{13}\cap  R''_{23}.\]
\end{lemma} 
That is, the ring $R'$ is the intersection of the above three rings. 
\begin{proof}Immediate from the division properties of multi-variable polynomials. 
\end{proof}
Thus, there are inclusions of rings 
\begin{align}\label{eq:ring_inc}
R \ \subset \ R' \ \subset \ \ R''_{ij}\ \subset \ R''
\end{align}

For a pre-foam $F$ and $c\in \adm(F)$, let 
\begin{align}\label{eq:P}
P(F,c) =\prod_{f  \in f(F)}  X^{d(f)}_{c(f)} 
\end{align}
be the monomial which is the product of $X_i$'s, over all facets $f$ of $F$, 
with the index $c(f)$, which is the color of the facet $f$, and the exponent $d(f)$ -- the 
number of dots on the facet $f$. 

For instance, if $(F,c)$ has two facets colored $1$ and decorated by three and no dots, respectively, 
one facet colored $2$ decorated by four dots, and two facets colored $3$ with two and 
three 
dots, respectively, then $P(F,c) = X_1^{3+0} X_2^4 X_3^{2+3} = X_1^3 X_2^4 X_3^5.$ 

With $F$ and $c$ as above, let 
\begin{align}\label{eq:Q}
Q(F,c) = \prod_{1\le i < j \le 3} (X_i + X_j)^{\frac{\chi(F_{ij}(c))}{2}}  \ \in R''. 
\end{align}
Here $\chi(S)$ denotes the Euler characteristic of a surface $S$. Since $c$ is 
admissible, closed surfaces $F_{ij}(c)$ are orientable and have even Euler characteristic. 
Consequently, $\frac{\chi(F_{ij}(c))}{2}$ is an integer. 

Now, given a pre-foam $F$ and $c\in \adm(F)$, define the \emph{evaluation} 
$\lFc$ by 
\begin{align}\label{eq:evalcol}
\lFc
= \frac{P(F,c)}{Q(F,c)} \ \in R''. 
\end{align}  
Note that, if none of the orientable surfaces $F_{ij}(c)$, for $1\le i < j \le 3$, contains a connected component which 
is a two-sphere, the integers $\chi(F_{ij}(c))$ are non-positive, and $\lFc$ belongs 
to the ring of polynomials $R'=\kk[X_1,X_2,X_3]$. It's possible for $\lFc$ to be 
a polynomial even if some components of $\chi(F_{ij}(c))$ are spheres, as long as 
the Euler characteristic of each $F_{ij}(c)$ is non-positive. 

Finally, we define the evaluation of a pre-foam $F$ as the sum of evaluations 
$\lFc$ over all admissible colorings of $F$: 
\begin{align} \lF
= \sum_{c \in \adm(F)} \lFc .
\label{eq:eval}
\end{align}
$\lF$ is an element of the ring $R''$. More precisely, it's an element of 
its $S_3$-invariant subring $(R'')^{S_3}$, under the permutation action of $S_3$ on the  
generators $X_1, X_2, X_3$. The invariance is implied by the action of $S_3$ on 
colorings, since $\sigma(\lFc) = \langle F, \sigma(c)\rangle$  for $\sigma \in S_3$. 

\begin{example}\label{exa:evaluation}
\begin{enumerate}
\item The empty pre-foam $\emptyset$ has a unique admissible coloring, and $\langle \emptyset \rangle = 1$. 
\item If $F$ has no admissible colorings, $\lF = 0$. Attaching two disks with disjoint interiors to the 2-torus standardly 
embedded in $\R^3$, one along meridian and one along longitude, yields a foam with a single seam vertex. This foam has no 
admissible colorings and evaluates to $0$ for any dot assignment. 

\item If $F$ is a $2$-sphere with $n$ dots, it has three colorings, one for each color.  
For the coloring of $F$ by color $1$, surfaces $F_{12}(c)$ and $F_{13}(c)$ are both $2$-spheres, 
while $F_{23}(c)$ is the empty surface, and 
\[\lFc = \frac{X_1^n}{(X_1+X_2)(X_1+X_3)} .\]
We have 
\begin{align*} 
\lF & =  \frac{X_1^n}{(X_1+X_2)(X_1+X_3)} + \frac{X_2^n}{(X_1+X_2)(X_2+X_3)}  + 
\frac{X_3^n}{(X_1+X_3)(X_2+X_3)} \\
  & =    \frac{X_1^n(X_2+X_3)+X_2^n(X_1+X_3)+X_3^n(X_1+X_2)}
  {(X_1+X_2)(X_1+X_3)(X_2+X_3)} \\
  & =  s_{n-2,0,0}(X_1,X_2,X_3) = h_{n-2}(X_1,X_2,X_3) = \sum_{i+j+k = n-2} X_1^i X_2^j X_3^k . 
  \end{align*} 
Adding signs to the last ratio above (which does not change the expression, since 
we are in characteristic two) makes it a ratio of an antisymmetrizer with the 
exponent $(n,1,0)$ and antisymmetrizer with the exponent $(2,1,0)$, that is, 
a Schur function $s_{\lambda}(X_1,X_2,X_3)$ for the partition
 $\lambda=(n-2,1-1,0-0)= (n-2,0,0)$. In characteristic $0$ this Schur function is the character of the $(n-2)$ symmetric power of the fundamental representation $V$ of $sl(3)$, hence equals the complete symmetric function $h_{n-2}(X_1,X_2,X_3)$. 
 
\begin{corollary} \label{cor:sphere-eval} A two-sphere with zero or one dot evaluates to $0$, a two-dotted 
 two-sphere evaluates to $1$, a three-dotted to $E_1=X_1+X_2+X_3$, and four-dotted to $E_1^2+E_2$. 
 \end{corollary}
 
\item\label{it:theta-prefoam} A theta-prefoam $\Theta$ consists of three disks glued together along three boundary circles. 
It can be visualized as a 2-sphere with an additional disk glued in along the equatorial circle. 
A theta-foam $\Theta$ is a theta-prefoam standardly embedded in $\R^3$ (see Figure~\ref{fig:theta-foam}).

\begin{figure}
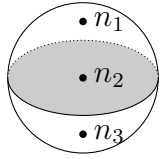

\centering
\tikz[scale=1]{\begin{scope}
  \draw (0,0) circle (1cm);
  \fill[gray!40!white] (0,0) circle (1cm and 0.5cm);
  \draw (-1, 0) arc (-180:0: 1cm and 0.5cm);
  \draw[densely dotted] (-1, 0) arc (180:0: 1cm and 0.5cm);
  \fill (0,0,0) circle (0.5mm) node[right] {$n_2$};
  \fill (0,0.75,0) circle (0.5mm) node[right] {$n_1$};
  \fill (0,-0.75,0) circle (0.5mm) node[right] {$n_3$};
\end{scope}}
\caption{A theta-foam whose facets are decorated with $n_1$, $n_2$ and $n_3$ dots.}\label{fig:theta-foam}
\end{figure}

Assume that facets of $\Theta$ are decorated by $n_1\ge n_2\ge n_3$ dots. 
We compute the evaluation 
\begin{align}\label{eq:theta-eval}
\langle \Theta \rangle \ = \ \frac{\sum_{\sigma\in S_3} X_{\sigma(1)}^{n_1}X_{\sigma(2)}^{n_2}X_{\sigma(3)}^{n_3}}{(X_1+X_2)(X_1+X_3)(X_2+X_3)} = s_{n_1-2, n_2-1, n_3}(X_1,X_2,X_3). 
\end{align}
There are six admissible colorings, with surfaces $F_{ij}(c)$ over all $i,j,c$ being 2-spheres. 
The evaluation is the Schur function 
$s_{\lambda}$, for the partition $(n_1-2, n_2-1, n_3)$. In particular, the evaluation is zero if 
any two numbers among $n_1, n_2, n_3$ are equal. If $n_1+n_2+n_3\le 3$, the only possible triple 
of dots with a nontrivial evaluation is $(2,1,0)$, which evaluates to $1\in \kk$. 
\item If $\Gamma$ is a trivalent graph, then $\Gamma\times \SS^1$ is a foam. The set of admissible colorings of this foam is naturally in bijection with the set of Tait coloring of $\Gamma$. For any coloring of this foam, the bicolored surfaces are collections of tori. Hence the evaluation of this foam is equal to the number of Tait coloring of $\Gamma$ modulo $2$. Since $S^3$ acts on coloring by permuting $1,2$ and $3$, the evaluation of $\Gamma \times \SS^1$ is $0$ unless $\Gamma$ is a (maybe empty) collection of circles. In this last case the evaluation is $1$.  
\end{enumerate}
\end{example}

\begin{theorem} \label{thm:evl-sym-pol}The evaluation $\lF$ of an admissible pre-foam $F$ is an element of the ring $R$ of symmetric 
polynomials in $X_1,X_2,X_3$, homogeneous of degree $\deg(F)$. 
\end{theorem} 
We view $0\in R$ as a homogeneous polynomial of an arbitrary degree. 

\begin{proof} 
The statement about the degree follows from the definition of the evaluation and Remark~\ref{rmk:degree}. Note that if the degree of a pre-foam is not even, Remark~\ref{rmk:degree} implies that $F$ has no admissible colorings and its evaluation is then automatically $0$.

That the evaluation is symmetric in $X_1$, $X_2$ and $X_3$ follows directly from the permutation action of $S_3$ on the set of admissible colorings of $F$. 

From now on, when we speak about colorings, we'll mean pre-admissible colorings. In particular, if a pre-foam is admissible, all its colorings are admissible as well. 

The theorem says that the evaluation of an admissible pre-foam lies in the subring $R$ of $R''$. 
The evaluation is clearly symmetric in $X_1, X_2, X_3$, so belongs 
to $(R'')^{S_3}$. It suffices to 
show that $\lF$ belongs to the subring $R_{12}''$ of $R''$. By $S_3$-symmetry 
we can then conclude that it belongs to the subrings $R_{13}''$, $R_{23}''$ as 
well, hence to the intersection of these three subrings with the subring 
 $(R'')^{S_3}$. Intersection of these four subrings is $R$. 
 
The argument is essentially the same as in \cite[Proposition~2.18]{RW1}. In order to be self-contained and since in our context the proof is simpler, we repeat it here. It is a direct consequence of the following lemma.
\begin{lemma}\label{lem:inR12}
Let $F$ be an admissible pre-foam and $S$ a subset of $f(F)$. Then
\[
\sum_{c\in \adm(F, 3=S)} \lFc \in R_{12}''.
\]
\end{lemma}
Indeed, we have
\[
\lF = \sum_{S\subseteq f(F)} \,\,\sum_{c\in \adm(F, 3=S)} \lFc ,
\]
and, therefore, $\lF$ is in $R_{12}''$, concluding 
proof of Theorem~\ref{thm:evl-sym-pol}.
  \end{proof}

\begin{proof}[Proof of Lemma~\ref{lem:inR12}]
If $\adm(F, 3=S)$ is empty, the statement is obvious. Suppose that $c_0$ is a coloring in $\adm(F, 3=S)$. Let us denote by $\Sigma_{1}, \dots,\Sigma_n$ the connected components of $\Sigma :=  
F_{12}(c_0) = F \setminus \bigcup_{f\in S} f$. For $a$ in $\{1, \dots, n\}$ and  $i \in \{1,2\}$ let $p_i(a)$ be the total number of dots located on facets of $\Sigma_a$ colored by $i$ (by $c_0$), and $p_3$ be the number of dots located on facets in $S$.
It follows from Lemmas~\ref{lem:colKempe} and \ref{lem:ell}(\ref{it:ellSigma}) that
\[
\sum_{c\in \adm(F, 3=S)} \lFc = \frac{\displaystyle{{X_3^{p_3}\prod_{a=1}^n\left( X_1^{p_1(a)} X_2^{p_2(a)} + X_2^{p_1(a)} X_1^{p_2(a)} \left(\frac{X_1+X_3}{ X_2+ X_3 } \right)^{\ell_{\Sigma_a(c)}} \right)}}}{(X_1 + X_2)^{{\chi(\Sigma)/2}} (X_1+X_3)^{\chi(F_{13}(c_0))/2}(X_2+X_3)^{\chi(F_{23}(c_0))/2}}
.
\]
If ${\chi(\Sigma)/2}$ is non-positive, the statement is obvious. Suppose that it is positive.
Each factor in the product is divisible by $(X_1+ X_2)$. Indeed, identifying $X_1$ and $X_2$ one gets $0$ (since the ground field has characteristic $2$). Since ${\chi(\Sigma)/2} = \sum_{a=1}^n \chi({\Sigma_a})/2 \leq n$, we can use the above factors $(X_1+X_2)$ to cancel $(X_1 + X_2)^{{\chi(\Sigma)/2}}$ in the denominator. Hence $\sum_{c\in \adm(F, 3=S)} \lFc$ is in $R_{12}''$.
\end{proof}

\subsection{Closed foams}\ 

\begin{defn} A (closed) foam $F$ is a (closed) pre-foam together with a piecewise linear 
embedding into $\R^3$. 
\end{defn} 

In this section, \emph{foam} will refer to a closed foam, and likewise for pre-foams. Later, we'll allow foams and pre-foams to have boundary and will refer to them as foams and pre-foams with boundary.  

A coloring of a foam $F$ is a coloring of the underlying pre-foam. Unless there's a
possibility of confusion, we denote the pre-foam underlying $F$ also by $F$. 

\begin{prop} Any (closed) foam $F$ is admissible. 
\end{prop} 
\begin{proof} For any pre-admissible coloring $c$ of $F$, the closed surfaces $F_{ij}(c)$ are 
embedded in $\R^3$, hence orientable, implying the admissibility of $c$.
\end{proof}

Note that, for a foam $F$, the evaluation $\lF$ does not depend on the embedding 
of the pre-foam of $F$ into $\R^3$. 

\begin{corollary} The evaluation $\lF$ of a foam $F$ in $\R^3$ is a symmetric 
polynomial in $X_1,X_2,X_3$ homogeneous of degree $\deg(F)$. 
\end{corollary} 

The evaluation is multiplicative for the disjoint union of pre-foams, $\langle F_1\sqcup F_2 \rangle = 
\langle F_1 \rangle \cdot \langle F_2 \rangle$.

A pre-foam $F$ is called \emph{connected} if it's a connected topological space. 
Suppose that a connected pre-foam $F$ has an involutive homeomorphism $\alpha$ that respects the number of dots on each facet, fixes at least one facet pointwise, and induces a non-trivial involution on the set of facets of $F$. Then 
the induced involution on $\adm(F)$ will have no fixed 
points. Since $\lFc = \langle F,\alpha(c)\rangle$
and $\mathrm{char}(\kk)=2$, the evaluation of $F$ is equal to $0$. An example is the theta-prefoam, see Example (\ref{it:theta-prefoam}) in (\ref{exa:evaluation}), when some 
of $n_1, n_2, n_3$ are equal.

\subsection{Relations between evaluations} 
\label{sec:rel-btwn-ev}\ 

In what follows we will speak about local relations satisfied by evaluations of  admissible pre-foams. This is to be understood as follows: Given a collection of pre-foams which are all admissible and are identical except in a ball where they are given by the terms of a local relation, evaluations of these pre-foams should satisfy the given identity, with coefficients in $R$. Note that since all foams (i.e., pre-foams embedded in $\R^3$) are admissible, these relations can be also thought of as local relation on foams.

\begin{prop}[Neck-cutting relation]
  \label{prop:neckcutting}
  The following local identity holds:
  \begin{align*}
    \brak{\tube[0.7]} = & \brak{\thstwozero[0.7]} + \brak{\thsoneone[0.7]} +
    \brak{\thszerotwo[0.7]} \\ & + E_1\left(\brak{\thsonezero[0.7]}+ \brak{\thszeroone[0.7]} \right) + E_2\brak{\thszerozero[0.7]}.
    \end{align*}
\end{prop}
\begin{proof}
Let us denote admissible pre-foam on the left-hand side by $F$. Six admissible pre-foams appearing on the right-hand side  are the same except from the dots distributed on them. We denote these pre-foams by $G_1, \dots, G_6$. A coloring $c$ of $G_1$ induces canonically a coloring of $G_2, \dots, G_6$, still denoted $c$, and any coloring of the $G_i$'s, for $i\in\{2,\dots, 6\}$, is obtained as such. 
By definition of the evaluation of pre-foams, the identity we intend to prove is equivalent to:
\begin{align*}
  \sum_{c \in \adm(F)}\!\!\brak{F,c} =  \!\!\sum_{c\in \adm(G_1)} \!\!\left(\brak{G_1, c} + \brak{G_2, c} + \brak{G_3, c} + E_1\left(\brak{G_4, c} + \brak{G_5, c}\right) + E_2\brak{G_6, c}\right).
\end{align*}
Let $c$ be a coloring of $G_1$. Denote by $\langle 
G,c\rangle'$ the sum on the right-hand side of the above equation for a fixed $c$.  
There are two types of colorings $c$: 
\begin{itemize}
\item The two half-spheres of $G_1$ have the same color. Then coloring $c$ of $G_1$ induces canonically a coloring of $F$, also denoted by $c$. Note that all colorings of $F$ are obtained in this way. Denote by $\adm_1(G_1)$ this set of colorings of $G_1$; there is a canonical bijection between 
$\adm_1(G_1)$ and $\adm(F)$. 
\item The two half-spheres have different colors. Denote by $\adm_2(G_1)$ this set of colorings of $G_1$.
\end{itemize}  
Suppose first that $c$ is in $\adm_1(G_1)$. 
Up to an $S_3$-symmetry, we may assume that the two half-spheres are colored by $1$. Then 
\begin{align*}
  \begin{cases}
    P(F,c) = P(G_6,c), \\
    P(G_1, c) = P(G_2, c) = P(G_3, c) = X_1^2 P(G_6,c),\\
    P(G_4,c) = P(G_5,c) = X_1 P(G_6,c), \\
    Q(G_6,c) = Q(F,c)(X_1+ X_2)(X_1+ X_3), \\
    Q(G_1,c )=     Q(G_2,c )=     Q(G_3,c )=     Q(G_4,c )=     Q(G_5,c )=     Q(G_6,c ).
  \end{cases}
\end{align*}
Hence, 
\begin{align*}
  \langle G, c\rangle' 
   & = \frac{\left(3X_1^2+ 2E_1X_1  + E_2\right)P(G_6,c)}{Q(G_6,c)} \\
   & = \frac{\left(X_1^2+ X_1X_2 + X_2X_3 + X_1X_3\right)P(F,c)}{Q(G_6,c)} \\
   & = \frac{(X_1 +X_2)(X_1+ X_3)P(F,c)}{(X_1 +X_2)(X_1+ X_3)Q(F,c)} \\
   & = \frac{P(F,c)}{Q(F,c)}.
\end{align*}

Suppose now that $c$ is in $\adm_2(G_1)$. Up to an  $S_3$-symmetry,  we may suppose that the upper half-sphere is colored by $1$ while the lower half-sphere is colored by $2$. Then 
\begin{align*}
  \begin{cases}
    P(G_1, c) =  X_1^2P(G_6, c), \ 
    P(G_2, c) =  X_1X_2P(G_6,c), \  
    P(G_3, c) =  X_2^2P(G_6, c), \\ 
    P(G_4, c) =  X_1P(G_6,c), \ 
    P(G_5, c) =  X_2P(G_6, c), \\
    Q(G_i,c )=   Q(G_j,c ), \ \forall i,j\in \{ 1, \dots, 6\}.
  \end{cases}
\end{align*}
Hence, 
\begin{align*}
  & \langle G, c\rangle'
   = \frac{\left(X_1^2 + X_1X_2+ X_2^2+ (X_1+ X_2+ X_3)(X_1+ X_2)  + E_2\right)P(G_6,c)}{Q(G_6,c)} \\
  & = \frac{\left((X_1+ X_2)^2 + X_1X_2 + (X_1+ X_2)^2+ X_3(X_1+ X_2)  + X_1X_2 + (X_1+ X_2)X_3  \right)P(G_6,c)}{Q(G_6,c)} \\
  &= 0
\end{align*}
and  
\begin{align*}
\sum_{c\in \adm(G_1)} \!\!\langle G, c\rangle' 
 =\! \sum_{c\in \adm_1(G_1)}\!\!\langle G, c\rangle'  
 \,+ \!\!\sum_{c\in \adm_2(G_1)}\!\! \langle G, c\rangle' 
 = \!\sum_{c\in \adm_1(G_1)}\!\! \brak{F, c} \,+ 0 
 = \!\sum_{c\in \adm(F)}\!\! \brak{F, c}.
\end{align*}
\end{proof}

\begin{prop}[Digon relation]\label{prop:digonrel}
 The following local identity holds:
 \begin{align*}
 \brak{\digonI}=  \brak{\thbonezero} + \brak{\thbzeroone}.
 \end{align*}
\end{prop}
Note that, on the right-hand side of the relation, one dot lies on a half-bubble facet that faces the reader, while the other 
lies on a facet away from the reader.

\begin{proof}
The proof is very similar to the previous one. Let us denote by $F$ the foam on the left hand side and by $G_1$ and $G_2$ the foams on the right hand side, that differ only by placement of a dot. Any coloring $c$ of $G_1$ induces a coloring of $G_2$ (also denoted $c$), and all colorings of $G_2$ are obtained as such. For a coloring $c$ of $G_1$ there are two possibilities:
\begin{itemize}
	\item The two facets of the two half-bubbles toward the reader have the same color. In this case, $c$ induces a coloring of $F$, still denoted $c$, and all colorings of $F$ are obtained as such. We denote the set of such colorings by $\adm_{1}(G_1)$; it's in bijection with $\adm(F)$. 
    \item The two facets of the two half-bubbles toward the reader have different colors. We denote by $\adm_2(G_1)$ the set of such colorings.
\end{itemize}
Let $c$ be a coloring of $G_1$. Suppose that it is in $\adm_1(G_1)$. Up to $S_3$-symmetries,  we may suppose that the facets of the half-bubbles toward the reader are colored $1$, the other two facets of the half-bubbles are colored $2$, and the remaining "big" facet is colored $3$.
We have:
\begin{align*}
  \begin{cases}
    P(G_1, c) =  X_2P(F, c), \\
    P(G_2, c) =  X_1P(F,c), \\
    Q(G_1,c )= Q(G_2,c )=   (X_1+X_2)Q(F,c ),
    \end{cases}
\end{align*}
so that  

\begin{align*}
\brak{G_1,c} + \brak{G_2,c} = \frac{P(G_1,c)}{Q(G_1,c)} +
\frac{P(G_2,c)}{Q(G_2,c)} 
= \frac{(X_1+X_2) P(F,c)}{(X_1+X_2)Q(F,c)}  = \brak{F,c}.
\end{align*}

Suppose now that $c$ is in $\adm_2(G_1)$. Up to an $S_3$-symmetry, we may suppose that the facet of the upper half-bubble toward the reader and the facet away from the reader of the lower half-bubble are colored by $1$, the other two facets of the half-bubbles are colored by $2$, and the remaining "big" facet is colored by $3$.
Then 
\begin{align*}
  \begin{cases}
    P(G_1, c) =  X_2P(F, c), \\
    P(G_2, c) =  X_2P(F,c), \\
    Q(G_1,c )= Q(G_2,c ),
    \end{cases}
\end{align*}
so that $\brak{G_1,c} + \brak{G_2,c} = 0$,  
and we conclude exactly as is the previous proposition.
\end{proof}

\begin{prop}[Square relation]
 \label{prop:squarerel}
 The following local identity holds:
 \begin{align*}
 \brak{\squareI}=  \brak{\squaresmoothone} + \brak{\squaresmoothtwo}.
 \end{align*}
\end{prop}

\begin{proof}

Let us denote by $F$ the foam on the left-hand side and by $G_1$ and $G_2$ the foams on the right-hand side, respectively.
Note that locally foams $G_1$ and $G_2$ are obtained from each other by $\pi/2$  vertical axis rotation.   In order to describe colorings of these foams, we slice them along three horizontal planes: at the top,  in the middle, and at the bottom, and collect the slices in three frames of a movie:
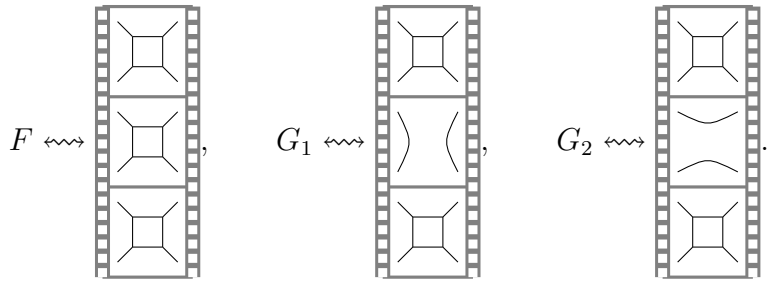
\begin{align*}
F  \leftrightsquigarrow \NB{\tikz[scale = 0.4]{\begin{scope}
    \draw[draw= gray, line width = 2mm] (-1.5,-4.5) -- +(0,9); 
    \draw[draw= gray, line width = 2mm] ( 1.5,-4.5) -- +(0,9); 
    \draw[draw= gray, very thick] (-1.5,-4.5) -- +(3,0); 
    \draw[draw= gray, very thick] (-1.5,-1.5) -- +(3,0); 
    \draw[draw= gray, very thick] (-1.5, 1.5) -- +(3,0); 
    \draw[draw= gray, very thick] (-1.5, 4.5) -- +(3,0); 
    \draw[draw= white, dotted, line width =1.2mm]  (-1.5,-4.5) -- +(0,9); 
    \draw[draw= white, dotted, line width =1.2mm]  ( 1.5,-4.5) -- +(0,9); 
  \end{scope}
\foreach \x in {-3,0,...,3}{
  \begin{scope}[yshift=\x cm]
    \coordinate (A) at (-1,-1);
    \coordinate (B) at (1,-1);
    \coordinate (C) at (1,1);
    \coordinate (D) at (-1,1);
    \coordinate (a) at (-.5,-.5);
    \coordinate (b) at (.5,-.5);
    \coordinate (c) at (.5,.5);
    \coordinate (d) at (-.5,.5);
    \draw (A) -- (a);
    \draw (c) -- (C);
    \draw (B) -- (b);
    \draw (d) -- (D);
    \draw (c) -- (d);
    \draw (a) -- (b);
    \draw (a) -- (d);
    \draw (b) -- (c);
  \end{scope}
}

}}, \qquad 
G_1 \leftrightsquigarrow \NB{\tikz[scale = 0.4]{\begin{scope}
    \draw[draw= gray, line width = 2mm] (-1.5,-4.5) -- +(0,9); 
    \draw[draw= gray, line width = 2mm] ( 1.5,-4.5) -- +(0,9); 
    \draw[draw= gray, very thick] (-1.5,-4.5) -- +(3,0); 
    \draw[draw= gray, very thick] (-1.5,-1.5) -- +(3,0); 
    \draw[draw= gray, very thick] (-1.5, 1.5) -- +(3,0); 
    \draw[draw= gray, very thick] (-1.5, 4.5) -- +(3,0); 
    \draw[draw= white, dotted, line width =1.2mm]  (-1.5,-4.5) -- +(0,9); 
    \draw[draw= white, dotted, line width =1.2mm]  ( 1.5,-4.5) -- +(0,9); 
  \end{scope}
\foreach \x in {-3,3}{
  \begin{scope}[yshift=\x cm]
    \coordinate (A) at (-1,-1);
    \coordinate (B) at (1,-1);
    \coordinate (C) at (1,1);
    \coordinate (D) at (-1,1);
    \coordinate (a) at (-.5,-.5);
    \coordinate (b) at (.5,-.5);
    \coordinate (c) at (.5,.5);
    \coordinate (d) at (-.5,.5);
    \draw (A) -- (a);
    \draw (c) -- (C);
    \draw (B) -- (b);
    \draw (d) -- (D);
    \draw (c) -- (d);
    \draw (a) -- (b);
    \draw (a) -- (d);
    \draw (b) -- (c);
  \end{scope}
}

  \begin{scope}
    \coordinate (A) at (-1,-1);
    \coordinate (B) at (1,-1);
    \coordinate (C) at (1,1);
    \coordinate (D) at (-1,1);
    \draw (A) .. controls (-0.5, 0) and (-0.5, 0) .. (D);
    \draw (C) .. controls ( 0.5, 0) and ( 0.5, 0) .. (B);
  \end{scope}

}}, \qquad   
G_2 \leftrightsquigarrow \NB{\tikz[scale = 0.4]{\begin{scope}
    \draw[draw= gray, line width = 2mm] (-1.5,-4.5) -- +(0,9); 
    \draw[draw= gray, line width = 2mm] ( 1.5,-4.5) -- +(0,9); 
    \draw[draw= gray, very thick] (-1.5,-4.5) -- +(3,0); 
    \draw[draw= gray, very thick] (-1.5,-1.5) -- +(3,0); 
    \draw[draw= gray, very thick] (-1.5, 1.5) -- +(3,0); 
    \draw[draw= gray, very thick] (-1.5, 4.5) -- +(3,0); 
    \draw[draw= white, dotted, line width =1.2mm]  (-1.5,-4.5) -- +(0,9); 
    \draw[draw= white, dotted, line width =1.2mm]  ( 1.5,-4.5) -- +(0,9); 
  \end{scope}
\foreach \x in {-3,3}{
  \begin{scope}[yshift=\x cm]
    \coordinate (A) at (-1,-1);
    \coordinate (B) at (1,-1);
    \coordinate (C) at (1,1);
    \coordinate (D) at (-1,1);
    \coordinate (a) at (-.5,-.5);
    \coordinate (b) at (.5,-.5);
    \coordinate (c) at (.5,.5);
    \coordinate (d) at (-.5,.5);
    \draw (A) -- (a);
    \draw (c) -- (C);
    \draw (B) -- (b);
    \draw (d) -- (D);
    \draw (c) -- (d);
    \draw (a) -- (b);
    \draw (a) -- (d);
    \draw (b) -- (c);
  \end{scope}
}

  \begin{scope}
    \coordinate (A) at (-1,-1);
    \coordinate (B) at (1,-1);
    \coordinate (C) at (1,1);
    \coordinate (D) at (-1,1);
    \draw (A) .. controls (0,-0.5) and (0,-0.5) .. (B);
    \draw (C) .. controls (0, 0.5) and (0, 0.5) .. (D);
  \end{scope}

}}.
\end{align*}
Up to an $S_3$-symmetry there are three local types of colorings of $F$, denoted $\adm_m(F)$, $\adm_v(F)$ and $\adm_h(F)$, where the letter $m$ stands for \emph{m}onochrome, $v$ for \emph{v}ertical, and $h$ for \emph{h}orizontal, depending on how the four facets on the sides of $F$ are colored: either in the same color (monochrome) or 'horizontally', or 'vertically' when viewed in the cross-section presentation of the coloring:
\begin{align*}
c\in \adm_m(F) \leftrightsquigarrow \NB{\tikz[scale = 0.4]{\begin{scope}
    \draw[draw= gray, line width = 2mm] (-1.5,-4.5) -- +(0,9); 
    \draw[draw= gray, line width = 2mm] ( 1.5,-4.5) -- +(0,9); 
    \draw[draw= gray, very thick] (-1.5,-4.5) -- +(3,0); 
    \draw[draw= gray, very thick] (-1.5,-1.5) -- +(3,0); 
    \draw[draw= gray, very thick] (-1.5, 1.5) -- +(3,0); 
    \draw[draw= gray, very thick] (-1.5, 4.5) -- +(3,0); 
    \draw[draw= white, dotted, line width =1.2mm]  (-1.5,-4.5) -- +(0,9); 
    \draw[draw= white, dotted, line width =1.2mm]  ( 1.5,-4.5) -- +(0,9); 
  \end{scope}
\foreach \x in {-3,0,...,3}{
  \begin{scope}[yshift=\x cm]
    \coordinate (A) at (-1,-1);
    \coordinate (B) at (1,-1);
    \coordinate (C) at (1,1);
    \coordinate (D) at (-1,1);
    \coordinate (a) at (-.5,-.5);
    \coordinate (b) at (.5,-.5);
    \coordinate (c) at (.5,.5);
    \coordinate (d) at (-.5,.5);
    \draw[red] (A) -- (a) node[pos =0.5, left, scale= 0.5] {$1$};
    \draw[red] (c) -- (C) node[pos =0.5, right, scale= 0.5] {$1$};
    \draw[red] (B) -- (b) node[pos =0.5, right, scale= 0.5] {$1$};
    \draw[red] (d) -- (D) node[pos =0.5, left, scale= 0.5] {$1$};
    \draw[blue] (c) -- (d) node[pos =0.5, above, scale= 0.5] {$3$};
    \draw[blue] (a) -- (b) node[pos =0.5, below, scale= 0.5] {$3$};
    \draw[green!50!black] (a) -- (d) node[pos =0.5, left, scale= 0.5] {$2$};
    \draw[green!50!black] (b) -- (c) node[pos =0.5, right, scale= 0.5] {$2$};
  \end{scope}
}}}, \,\,
c\in \adm_h(F) \leftrightsquigarrow \NB{\tikz[scale = 0.4]{\begin{scope}
    \draw[draw= gray, line width = 2mm] (-1.5,-4.5) -- +(0,9); 
    \draw[draw= gray, line width = 2mm] ( 1.5,-4.5) -- +(0,9); 
    \draw[draw= gray, very thick] (-1.5,-4.5) -- +(3,0); 
    \draw[draw= gray, very thick] (-1.5,-1.5) -- +(3,0); 
    \draw[draw= gray, very thick] (-1.5, 1.5) -- +(3,0); 
    \draw[draw= gray, very thick] (-1.5, 4.5) -- +(3,0); 
    \draw[draw= white, dotted, line width =1.2mm]  (-1.5,-4.5) -- +(0,9); 
    \draw[draw= white, dotted, line width =1.2mm]  ( 1.5,-4.5) -- +(0,9); 
  \end{scope}
\foreach \x in {-3,0,...,3}{
  \begin{scope}[yshift=\x cm]
    \coordinate (A) at (-1,-1);
    \coordinate (B) at (1,-1);
    \coordinate (C) at (1,1);
    \coordinate (D) at (-1,1);
    \coordinate (a) at (-.5,-.5);
    \coordinate (b) at (.5,-.5);
    \coordinate (c) at (.5,.5);
    \coordinate (d) at (-.5,.5);
    \draw[green!50!black] (A) -- (a) node[pos =0.5, left, scale= 0.5] {$2$};
    \draw[red] (c) -- (C) node[pos =0.5, right, scale= 0.5] {$1$};
    \draw[green!50!black] (B) -- (b) node[pos =0.5, right, scale= 0.5] {$2$};
    \draw[red] (d) -- (D) node[pos =0.5, left, scale= 0.5] {$1$};
    \draw[green!50!black] (c) -- (d) node[pos =0.5, above, scale= 0.5] {$2$};
    \draw[red] (a) -- (b) node[pos =0.5, below, scale= 0.5] {$1$};
    \draw[blue] (a) -- (d) node[pos =0.5, left, scale= 0.5] {$3$};
    \draw[blue] (b) -- (c) node[pos =0.5, right, scale= 0.5] {$3$};
  \end{scope}
}}}, \,\, 
c\in \adm_v(F) \leftrightsquigarrow \NB{\tikz[scale = 0.4]{\begin{scope}
    \draw[draw= gray, line width = 2mm] (-1.5,-4.5) -- +(0,9); 
    \draw[draw= gray, line width = 2mm] ( 1.5,-4.5) -- +(0,9); 
    \draw[draw= gray, very thick] (-1.5,-4.5) -- +(3,0); 
    \draw[draw= gray, very thick] (-1.5,-1.5) -- +(3,0); 
    \draw[draw= gray, very thick] (-1.5, 1.5) -- +(3,0); 
    \draw[draw= gray, very thick] (-1.5, 4.5) -- +(3,0); 
    \draw[draw= white, dotted, line width =1.2mm]  (-1.5,-4.5) -- +(0,9); 
    \draw[draw= white, dotted, line width =1.2mm]  ( 1.5,-4.5) -- +(0,9); 
  \end{scope}
\foreach \x in {-3,0,...,3}{
  \begin{scope}[yshift=\x cm]
    \coordinate (A) at (-1,-1);
    \coordinate (B) at (1,-1);
    \coordinate (C) at (1,1);
    \coordinate (D) at (-1,1);
    \coordinate (a) at (-.5,-.5);
    \coordinate (b) at (.5,-.5);
    \coordinate (c) at (.5,.5);
    \coordinate (d) at (-.5,.5);
    \draw[red] (A) -- (a) node[pos =0.5, left, scale= 0.5] {$1$};
    \draw[green!50!black] (c) -- (C) node[pos =0.5, right, scale= 0.5] {$2$};
    \draw[green!50!black] (B) -- (b) node[pos =0.5, right, scale= 0.5] {$2$};
    \draw[red] (d) -- (D) node[pos =0.5, left, scale= 0.5] {$1$};
    \draw[blue] (c) -- (d) node[pos =0.5, above, scale= 0.5] {$3$};
    \draw[blue] (a) -- (b) node[pos =0.5, below, scale= 0.5] {$3$};
    \draw[green!50!black] (a) -- (d) node[pos =0.5, left, scale= 0.5] {$2$};
    \draw[red] (b) -- (c) node[pos =0.5, right, scale= 0.5] {$1$};
  \end{scope}
}}}.
\end{align*}
Up to an $S_3$-symmetry there are three local types of colorings of $G_1$, denoted $\adm_{m_1}(G_1)$, $\adm_{m_2}(G_1)$ and $\adm_v(G_1)$:
\begin{align*}
c\in \adm_{m_1}(G_1) \leftrightsquigarrow \NB{\tikz[scale = 0.4]{\begin{scope}
    \draw[draw= gray, line width = 2mm] (-1.5,-4.5) -- +(0,9); 
    \draw[draw= gray, line width = 2mm] ( 1.5,-4.5) -- +(0,9); 
    \draw[draw= gray, very thick] (-1.5,-4.5) -- +(3,0); 
    \draw[draw= gray, very thick] (-1.5,-1.5) -- +(3,0); 
    \draw[draw= gray, very thick] (-1.5, 1.5) -- +(3,0); 
    \draw[draw= gray, very thick] (-1.5, 4.5) -- +(3,0); 
    \draw[draw= white, dotted, line width =1.2mm]  (-1.5,-4.5) -- +(0,9); 
    \draw[draw= white, dotted, line width =1.2mm]  ( 1.5,-4.5) -- +(0,9); 
  \end{scope}
\foreach \x in {-3,3}{
  \begin{scope}[yshift=\x cm]
    \coordinate (A) at (-1,-1);
    \coordinate (B) at (1,-1);
    \coordinate (C) at (1,1);
    \coordinate (D) at (-1,1);
    \coordinate (a) at (-.5,-.5);
    \coordinate (b) at (.5,-.5);
    \coordinate (c) at (.5,.5);
    \coordinate (d) at (-.5,.5);
    \draw[red] (A) -- (a) node[pos =0.5, left, scale= 0.5] {$1$};
    \draw[red] (c) -- (C) node[pos =0.5, right, scale= 0.5] {$1$};
    \draw[red] (B) -- (b) node[pos =0.5, right, scale= 0.5] {$1$};
    \draw[red] (d) -- (D) node[pos =0.5, left, scale= 0.5] {$1$};
    \draw[blue] (c) -- (d) node[pos =0.5, above, scale= 0.5] {$3$};
    \draw[blue] (a) -- (b) node[pos =0.5, below, scale= 0.5] {$3$};
    \draw[green!50!black] (a) -- (d) node[pos =0.5, left, scale= 0.5] {$2$};
    \draw[green!50!black] (b) -- (c) node[pos =0.5, right, scale= 0.5] {$2$};
  \end{scope}
}

  \begin{scope}
    \coordinate (A) at (-1,-1);
    \coordinate (B) at (1,-1);
    \coordinate (C) at (1,1);
    \coordinate (D) at (-1,1);
    \draw[red] (A) .. controls (-0.5, 0) and (-0.5, 0) .. (D) node[pos =0.5, left, scale= 0.5] {$1$};
    \draw[red] (C) .. controls ( 0.5, 0) and ( 0.5, 0) .. (B) node[pos =0.5, right, scale= 0.5] {$1$};
  \end{scope}}}, \,\, 
c\in \adm_{m_2}(G_1) \leftrightsquigarrow \NB{\tikz[scale = 0.4]{\begin{scope}
    \draw[draw= gray, line width = 2mm] (-1.5,-4.5) -- +(0,9); 
    \draw[draw= gray, line width = 2mm] ( 1.5,-4.5) -- +(0,9); 
    \draw[draw= gray, very thick] (-1.5,-4.5) -- +(3,0); 
    \draw[draw= gray, very thick] (-1.5,-1.5) -- +(3,0); 
    \draw[draw= gray, very thick] (-1.5, 1.5) -- +(3,0); 
    \draw[draw= gray, very thick] (-1.5, 4.5) -- +(3,0); 
    \draw[draw= white, dotted, line width =1.2mm]  (-1.5,-4.5) -- +(0,9); 
    \draw[draw= white, dotted, line width =1.2mm]  ( 1.5,-4.5) -- +(0,9); 
  \end{scope}
  \begin{scope}[yshift=3 cm]
    \coordinate (A) at (-1,-1);
    \coordinate (B) at (1,-1);
    \coordinate (C) at (1,1);
    \coordinate (D) at (-1,1);
    \coordinate (a) at (-.5,-.5);
    \coordinate (b) at (.5,-.5);
    \coordinate (c) at (.5,.5);
    \coordinate (d) at (-.5,.5);
    \draw[red] (A) -- (a) node[pos =0.5, left, scale= 0.5] {$1$};
    \draw[red] (c) -- (C) node[pos =0.5, right, scale= 0.5] {$1$};
    \draw[red] (B) -- (b) node[pos =0.5, right, scale= 0.5] {$1$};
    \draw[red] (d) -- (D) node[pos =0.5, left, scale= 0.5] {$1$};
    \draw[blue] (c) -- (d) node[pos =0.5, above, scale= 0.5] {$3$};
    \draw[blue] (a) -- (b) node[pos =0.5, below, scale= 0.5] {$3$};
    \draw[green!50!black] (a) -- (d) node[pos =0.5, left, scale= 0.5] {$2$};
    \draw[green!50!black] (b) -- (c) node[pos =0.5, right, scale= 0.5] {$2$};
  \end{scope}
  \begin{scope}[yshift= -3cm]
    \coordinate (A) at (-1,-1);
    \coordinate (B) at (1,-1);
    \coordinate (C) at (1,1);
    \coordinate (D) at (-1,1);
    \coordinate (a) at (-.5,-.5);
    \coordinate (b) at (.5,-.5);
    \coordinate (c) at (.5,.5);
    \coordinate (d) at (-.5,.5);
    \draw[red] (A) -- (a) node[pos =0.5, left, scale= 0.5] {$1$};
    \draw[red] (c) -- (C) node[pos =0.5, right, scale= 0.5] {$1$};
    \draw[red] (B) -- (b) node[pos =0.5, right, scale= 0.5] {$1$};
    \draw[red] (d) -- (D) node[pos =0.5, left, scale= 0.5] {$1$};
    \draw[green!50!black] (c) -- (d) node[pos =0.5, above, scale= 0.5] {$2$};
    \draw[green!50!black] (a) -- (b) node[pos =0.5, below, scale= 0.5] {$2$};
    \draw[blue] (a) -- (d) node[pos =0.5, left, scale= 0.5] {$3$};
    \draw[blue] (b) -- (c) node[pos =0.5, right, scale= 0.5] {$3$};
  \end{scope}
  \begin{scope}
    \coordinate (A) at (-1,-1);
    \coordinate (B) at (1,-1);
    \coordinate (C) at (1,1);
    \coordinate (D) at (-1,1);
    \draw[red] (A) .. controls (-0.5, 0) and (-0.5, 0) .. (D) node[pos =0.5, left, scale= 0.5] {$1$};
    \draw[red] (C) .. controls ( 0.5, 0) and ( 0.5, 0) .. (B) node[pos =0.5, right, scale= 0.5] {$1$};
  \end{scope}}}, \,\, 
c\in \adm_v(G_1) \leftrightsquigarrow \NB{\tikz[scale = 0.4]{\begin{scope}
    \draw[draw= gray, line width = 2mm] (-1.5,-4.5) -- +(0,9); 
    \draw[draw= gray, line width = 2mm] ( 1.5,-4.5) -- +(0,9); 
    \draw[draw= gray, very thick] (-1.5,-4.5) -- +(3,0); 
    \draw[draw= gray, very thick] (-1.5,-1.5) -- +(3,0); 
    \draw[draw= gray, very thick] (-1.5, 1.5) -- +(3,0); 
    \draw[draw= gray, very thick] (-1.5, 4.5) -- +(3,0); 
    \draw[draw= white, dotted, line width =1.2mm]  (-1.5,-4.5) -- +(0,9); 
    \draw[draw= white, dotted, line width =1.2mm]  ( 1.5,-4.5) -- +(0,9); 
  \end{scope}
\foreach \x in {-3,3}{
  \begin{scope}[yshift=\x cm]
    \coordinate (A) at (-1,-1);
    \coordinate (B) at (1,-1);
    \coordinate (C) at (1,1);
    \coordinate (D) at (-1,1);
    \coordinate (a) at (-.5,-.5);
    \coordinate (b) at (.5,-.5);
    \coordinate (c) at (.5,.5);
    \coordinate (d) at (-.5,.5);
    \draw[red] (A) -- (a) node[pos =0.5, left, scale= 0.5] {$1$};
    \draw[green!50!black] (c) -- (C) node[pos =0.5, right, scale= 0.5] {$2$};
    \draw[green!50!black] (B) -- (b) node[pos =0.5, right, scale= 0.5] {$2$};
    \draw[red] (d) -- (D) node[pos =0.5, left, scale= 0.5] {$1$};
    \draw[blue] (c) -- (d) node[pos =0.5, above, scale= 0.5] {$3$};
    \draw[blue] (a) -- (b) node[pos =0.5, below, scale= 0.5] {$3$};
    \draw[green!50!black] (a) -- (d) node[pos =0.5, left, scale= 0.5] {$2$};
    \draw[red] (b) -- (c) node[pos =0.5, right, scale= 0.5] {$1$};
  \end{scope}
}

  \begin{scope}
    \coordinate (A) at (-1,-1);
    \coordinate (B) at (1,-1);
    \coordinate (C) at (1,1);
    \coordinate (D) at (-1,1);
    \draw[red] (A) .. controls (-0.5, 0) and (-0.5, 0) .. (D) node[pos =0.5, left, scale= 0.5] {$1$};
    \draw[green!50!black] (C) .. controls ( 0.5, 0) and ( 0.5, 0) .. (B) node[pos =0.5, right, scale= 0.5] {$2$};
  \end{scope}}}.
\end{align*}
Up to an $S_3$-symmetry there are three local types of colorings of $G_2$, denoted $\adm_{m_1}(G_2)$, $\adm_{m_2}(G_2)$ and $\adm_h(G_1)$:
\begin{align*}
c\in \adm_{m_1}(G_2) \leftrightsquigarrow \NB{\tikz[scale = 0.4]{\begin{scope}
    \draw[draw= gray, line width = 2mm] (-1.5,-4.5) -- +(0,9); 
    \draw[draw= gray, line width = 2mm] ( 1.5,-4.5) -- +(0,9); 
    \draw[draw= gray, very thick] (-1.5,-4.5) -- +(3,0); 
    \draw[draw= gray, very thick] (-1.5,-1.5) -- +(3,0); 
    \draw[draw= gray, very thick] (-1.5, 1.5) -- +(3,0); 
    \draw[draw= gray, very thick] (-1.5, 4.5) -- +(3,0); 
    \draw[draw= white, dotted, line width =1.2mm]  (-1.5,-4.5) -- +(0,9); 
    \draw[draw= white, dotted, line width =1.2mm]  ( 1.5,-4.5) -- +(0,9); 
  \end{scope}
\foreach \x in {-3,3} {
 \begin{scope}[yshift=\x cm]
    \coordinate (A) at (-1,-1);
    \coordinate (B) at (1,-1);
    \coordinate (C) at (1,1);
    \coordinate (D) at (-1,1);
    \coordinate (a) at (-.5,-.5);
    \coordinate (b) at (.5,-.5);
    \coordinate (c) at (.5,.5);
    \coordinate (d) at (-.5,.5);
    \draw[red] (A) -- (a) node[pos =0.5, left, scale= 0.5] {$1$};
    \draw[red] (c) -- (C) node[pos =0.5, right, scale= 0.5] {$1$};
    \draw[red] (B) -- (b) node[pos =0.5, right, scale= 0.5] {$1$};
    \draw[red] (d) -- (D) node[pos =0.5, left, scale= 0.5] {$1$};
    \draw[blue] (c) -- (d) node[pos =0.5, above, scale= 0.5] {$3$};
    \draw[blue] (a) -- (b) node[pos =0.5, below, scale= 0.5] {$3$};
    \draw[green!50!black] (a) -- (d) node[pos =0.5, left, scale= 0.5] {$2$};
    \draw[green!50!black] (b) -- (c) node[pos =0.5, right, scale= 0.5] {$2$};
  \end{scope}
}
  \begin{scope}
    \coordinate (A) at (-1,-1);
    \coordinate (B) at (1,-1);
    \coordinate (C) at (1,1);
    \coordinate (D) at (-1,1);
    \draw[red] (A) .. controls (0,-0.5) and (0,-0.5) .. (B) node[pos =0.5, below, scale= 0.5] {$1$};
    \draw[red] (C) .. controls (0, 0.5) and (0, 0.5) .. (D) node[pos =0.5, above, scale= 0.5] {$1$};
  \end{scope}}}, \,\, 
c\in \adm_{m_2}(G_1) \leftrightsquigarrow \NB{\tikz[scale = 0.4]{\begin{scope}
    \draw[draw= gray, line width = 2mm] (-1.5,-4.5) -- +(0,9); 
    \draw[draw= gray, line width = 2mm] ( 1.5,-4.5) -- +(0,9); 
    \draw[draw= gray, very thick] (-1.5,-4.5) -- +(3,0); 
    \draw[draw= gray, very thick] (-1.5,-1.5) -- +(3,0); 
    \draw[draw= gray, very thick] (-1.5, 1.5) -- +(3,0); 
    \draw[draw= gray, very thick] (-1.5, 4.5) -- +(3,0); 
    \draw[draw= white, dotted, line width =1.2mm]  (-1.5,-4.5) -- +(0,9); 
    \draw[draw= white, dotted, line width =1.2mm]  ( 1.5,-4.5) -- +(0,9); 
  \end{scope}
  \begin{scope}[yshift=3 cm]
    \coordinate (A) at (-1,-1);
    \coordinate (B) at (1,-1);
    \coordinate (C) at (1,1);
    \coordinate (D) at (-1,1);
    \coordinate (a) at (-.5,-.5);
    \coordinate (b) at (.5,-.5);
    \coordinate (c) at (.5,.5);
    \coordinate (d) at (-.5,.5);
    \draw[red] (A) -- (a) node[pos =0.5, left, scale= 0.5] {$1$};
    \draw[red] (c) -- (C) node[pos =0.5, right, scale= 0.5] {$1$};
    \draw[red] (B) -- (b) node[pos =0.5, right, scale= 0.5] {$1$};
    \draw[red] (d) -- (D) node[pos =0.5, left, scale= 0.5] {$1$};
    \draw[blue] (c) -- (d) node[pos =0.5, above, scale= 0.5] {$3$};
    \draw[blue] (a) -- (b) node[pos =0.5, below, scale= 0.5] {$3$};
    \draw[green!50!black] (a) -- (d) node[pos =0.5, left, scale= 0.5] {$2$};
    \draw[green!50!black] (b) -- (c) node[pos =0.5, right, scale= 0.5] {$2$};
  \end{scope}
  \begin{scope}[yshift= -3cm]
    \coordinate (A) at (-1,-1);
    \coordinate (B) at (1,-1);
    \coordinate (C) at (1,1);
    \coordinate (D) at (-1,1);
    \coordinate (a) at (-.5,-.5);
    \coordinate (b) at (.5,-.5);
    \coordinate (c) at (.5,.5);
    \coordinate (d) at (-.5,.5);
    \draw[red] (A) -- (a) node[pos =0.5, left, scale= 0.5] {$1$};
    \draw[red] (c) -- (C) node[pos =0.5, right, scale= 0.5] {$1$};
    \draw[red] (B) -- (b) node[pos =0.5, right, scale= 0.5] {$1$};
    \draw[red] (d) -- (D) node[pos =0.5, left, scale= 0.5] {$1$};
    \draw[green!50!black] (c) -- (d) node[pos =0.5, above, scale= 0.5] {$2$};
    \draw[green!50!black] (a) -- (b) node[pos =0.5, below, scale= 0.5] {$2$};
    \draw[blue] (a) -- (d) node[pos =0.5, left, scale= 0.5] {$3$};
    \draw[blue] (b) -- (c) node[pos =0.5, right, scale= 0.5] {$3$};
  \end{scope}
  \begin{scope}
    \coordinate (A) at (-1,-1);
    \coordinate (B) at (1,-1);
    \coordinate (C) at (1,1);
    \coordinate (D) at (-1,1);
    \draw[red] (A) .. controls (0,-0.5) and (0,-0.5) .. (B) node[pos =0.5, below, scale= 0.5] {$1$};
    \draw[red] (C) .. controls (0, 0.5) and (0, 0.5) .. (D) node[pos =0.5, above, scale= 0.5] {$1$};
  \end{scope}}}, \,\, 
c\in \adm_h(G_2) \leftrightsquigarrow \NB{\tikz[scale = 0.4]{\begin{scope}
    \draw[draw= gray, line width = 2mm] (-1.5,-4.5) -- +(0,9); 
    \draw[draw= gray, line width = 2mm] ( 1.5,-4.5) -- +(0,9); 
    \draw[draw= gray, very thick] (-1.5,-4.5) -- +(3,0); 
    \draw[draw= gray, very thick] (-1.5,-1.5) -- +(3,0); 
    \draw[draw= gray, very thick] (-1.5, 1.5) -- +(3,0); 
    \draw[draw= gray, very thick] (-1.5, 4.5) -- +(3,0); 
    \draw[draw= white, dotted, line width =1.2mm]  (-1.5,-4.5) -- +(0,9); 
    \draw[draw= white, dotted, line width =1.2mm]  ( 1.5,-4.5) -- +(0,9); 
  \end{scope}
\foreach \x in {-3,3} {
  \begin{scope}[yshift=\x cm]
    \coordinate (A) at (-1,-1);
    \coordinate (B) at (1,-1);
    \coordinate (C) at (1,1);
    \coordinate (D) at (-1,1);
    \coordinate (a) at (-.5,-.5);
    \coordinate (b) at (.5,-.5);
    \coordinate (c) at (.5,.5);
    \coordinate (d) at (-.5,.5);
    \draw[green!50!black] (A) -- (a) node[pos =0.5, left, scale= 0.5] {$2$};
    \draw[red] (c) -- (C) node[pos =0.5, right, scale= 0.5] {$1$};
    \draw[green!50!black] (B) -- (b) node[pos =0.5, right, scale= 0.5] {$2$};
    \draw[red] (d) -- (D) node[pos =0.5, left, scale= 0.5] {$1$};
    \draw[green!50!black] (c) -- (d) node[pos =0.5, above, scale= 0.5] {$2$};
    \draw[red] (a) -- (b) node[pos =0.5, below, scale= 0.5] {$1$};
    \draw[blue] (a) -- (d) node[pos =0.5, left, scale= 0.5] {$3$};
    \draw[blue] (b) -- (c) node[pos =0.5, right, scale= 0.5] {$3$};
  \end{scope}
}
  \begin{scope}
    \coordinate (A) at (-1,-1);
    \coordinate (B) at (1,-1);
    \coordinate (C) at (1,1);
    \coordinate (D) at (-1,1);
    \draw[green!50!black] (A) .. controls (0,-0.5) and (0,-0.5) .. (B) node[pos =0.5, below, scale= 0.5] {$2$};
    \draw[red] (C) .. controls (0, 0.5) and (0, 0.5) .. (D) node[pos =0.5, above, scale= 0.5] {$1$};
  \end{scope}}}.
\end{align*}
Note that we have the following canonical bijections:
\begin{align*}
&\adm_{m}(F) \simeq \adm_{m_1}(G_1)\simeq \adm_{m_1}(G_2), \qquad &&\adm_{m_2}(G_1) \simeq\adm_{m_2}(G_2) 
,\\
&\adm_{v}(F) \simeq \adm_{v}(G_1), &&\adm_{h}(F) \simeq \adm_{h}(G_2). 
\end{align*}
For $c \in \adm_{m}(F) \simeq \adm_{m_1}(G_1)\simeq \adm_{m_1}(G_2)$ we have
\begin{align*}
\begin{cases}
	P(F,c) = P(G_1,c) = P(G_2,c),\\
    Q(F,c) = \frac{X_1+X_3}{X_2+X_3} Q(G_1,c) = \frac{X_1+X_2}{X_2+X_3}Q(G_2,c),
\end{cases}
\end{align*}
and 
\begin{align*}
\brak{G_1,c} + \brak{G_2,c} = \left( \frac{X_1+ X_3}{X_2+X_3} + \frac{X_1+ X_2}{X_2+X_3}\right)\brak{F,c} = \brak{F,c}.
\end{align*}
In the above computation we assumed that the coloring $c$ was given by the movie depicted earlier. Modifying $c$ by permuting the colors changes the indices of $X$ in the computation but not the final result $\brak{G_1,c} + \brak{G_2,c} = \brak{F,c}$. 

For $c \in \adm_{v}(F) \simeq \adm_{v}(G_1)$ (resp. $c \in \adm_{h}(F) \simeq \adm_{h}(G_2)$) we have:
\begin{align*}
\begin{cases}
	P(F,c) = P(G_1,c) \quad \textrm{(resp. $P(F,c) = P(G_2,c)$)},\\
    Q(F,c) = Q(G_1,c) \quad \textrm{(resp. $Q(F,c) = Q(G_2,c)$)}.
\end{cases}
\end{align*}
This gives:
\begin{align*}
\brak{G_1,c}=\brak{F,c} \quad \left( \textrm{resp. \,}\brak{G_2,c} = \brak{F,c} \right) . 
\end{align*}

For $c \in \adm_{m_2}(G_1) \simeq \adm_{m_2}(G_2)$ we have
\[ P(G_1,c) = P(G_2,c), \quad 
    Q(G_1,c) = Q(G_2,c), \]
so that 
$\brak{G_1,c} + \brak{G_2,c} = 0.$

We sum up the case-by-case study:
\begin{align*}
\begin{cases}
 \brak{F,c} = \brak{G_1, c} + \brak{G_2, c} &\textrm{for $c \in \adm_m(F)\simeq \adm_{m_1}(G_1)\simeq \adm_{m_1}(G_2)$,}\\
 \brak{F,c} = \brak{G_1, c} &\textrm{for $c \in \adm_v(F) \simeq \adm_v(G_1)$,}\\
 \brak{F,c} = \brak{G_2, c} &\textrm{for $c \in \adm_h(F) \simeq \adm_h(G_2)$,}\\
  \brak{G_1, c} + \brak{G_2, c} = 0 &\textrm{for $c \in \adm_{m_2}(G_1) \simeq \adm_{m_2}(G_2).$}
\end{cases}
\end{align*}
This gives
	\brak{F} = \brak{G_1} + \brak{G_2}.
\end{proof}

\begin{prop}[Trivalent bubble relation]\label{prop:trivalentbubble}
 The following local identity holds:

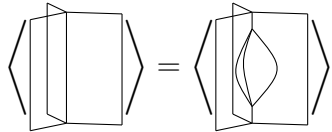
\begin{align*}
\brak{\NB{\tikz[scale=0.5]{
\tdplotsetmaincoords{80}{130}
\begin{scope}[scale = 1.5, tdplot_main_coords]
  \tikzset{yxplane/.style={canvas is xy plane at z=#1}}
  \begin{scope}[yxplane=1]
    \coordinate (OT) at ( 0, 0);
    \coordinate (AT) at ({cos(  0)}, {sin(  0)});
    \coordinate (BT) at ({cos(120)}, {sin(120)});
    \coordinate (CT) at ({cos(240)}, {sin(240)});
    \draw (OT) -- (AT);
    \draw (OT) -- (BT);
    \draw (OT) -- (CT);
  \end{scope}
  \begin{scope}[yxplane=-1]
    \coordinate (OB) at (0, 0);
    \coordinate (AB) at ({cos(  0)}, {sin(  0)});
    \coordinate (BB) at ({cos(120)}, {sin(120)});
    \coordinate (CB) at ({cos(240)}, {sin(240)});
    \draw (OB) -- (AB);
    \draw (OB) -- (BB);
    \draw (OB) -- (CB);
  \end{scope}
  \draw (OB) -- (OT);
  \draw (AB) -- (AT);
  \draw (BB) -- (BT);
  \draw (CB) -- (CT);
\end{scope}
}}}
=
\brak{\NB{\tikz[scale=0.5]{
\tdplotsetmaincoords{80}{130}
\begin{scope}[scale = 1.5, tdplot_main_coords]
  \tikzset{yxplane/.style={canvas is xy plane at z=#1}}
  \begin{scope}[yxplane=1]
    \coordinate (OT) at ( 0, 0);
    \coordinate (AT) at ({cos(  0)}, {sin(  0)});
    \coordinate (BT) at ({cos(120)}, {sin(120)});
    \coordinate (CT) at ({cos(240)}, {sin(240)});
    \draw (OT) -- (AT);
    \draw (OT) -- (BT);
    \draw (OT) -- (CT);
  \end{scope}
  \begin{scope}[yxplane=0]
    \coordinate (OB) at (0, 0);
    \coordinate (AM) at ({0.6*cos(  0)}, {0.6*sin(  0)});
    \coordinate (BM) at ({0.6*cos(120)}, {0.6*sin(120)});
    \coordinate (CM) at ({0.6*cos(240)}, {0.6*sin(240)});
  \end{scope}
  \begin{scope}[yxplane=-1]
    \coordinate (OB) at (0, 0);
    \coordinate (AB) at ({cos(  0)}, {sin(  0)});
    \coordinate (BB) at ({cos(120)}, {sin(120)});
    \coordinate (CB) at ({cos(240)}, {sin(240)});
    \draw (OB) -- (AB);
    \draw (OB) -- (BB);
    \draw (OB) -- (CB);
  \end{scope}
  \coordinate (M1) at (0, 0, -0.7);
  \coordinate (M2) at (0, 0,  0.7);
  \draw (OB) -- (M1);
  \draw (OT) -- (M2);
  \draw (M1)  .. controls (AM) and (AM) .. (M2);
  \draw (M1)  .. controls (BM) and (BM) .. (M2);
  \draw (M1)  .. controls (CM) and (CM) .. (M2);
  \draw (AB) -- (AT);
  \draw (BB) -- (BT);
  \draw (CB) -- (CT);
\end{scope}
}}}
\end{align*}
\end{prop}

\begin{proof}
	Let us denote by $F$ the foam on the left-hand side and by $G$ the foam on the right-hand side. There is a canonical bijection between colorings of $F$ and  $G$. Denoting by $c'$ the coloring of $G$ that corresponds to a coloring $c$ of $F$, we have:
    \begin{align*}
    		P(F,c) = P(G,c'), \quad
            Q(F,c) = Q(G,c'),
    \end{align*}
so that $\brak{F,c} =\brak{G,c'}$ and 
$\brak{F} =\brak{G}$.
\end{proof}
The same argument gives the following proposition.
\begin{prop}[Vertices removal relation]\label{prop:verticesremoval}
 The following local identity holds:
 \begin{align*}
 \brak{\NB
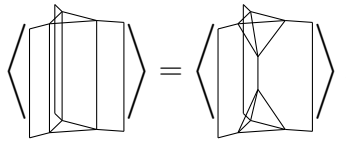
{\tikz[scale=0.5]{\tdplotsetmaincoords{75}{135}
\begin{scope}[scale = 1.5, tdplot_main_coords]
  \tikzset{yxplane/.style={canvas is xy plane at z=#1}}
  \begin{scope}[yxplane=1]
    \coordinate (AT) at ({cos(  0)}, {sin(  0)});
    \coordinate (BT) at ({cos(120)}, {sin(120)});
    \coordinate (CT) at ({cos(240)}, {sin(240)});
    \coordinate (aT) at ({0.5*cos(  0)}, {0.5*sin(  0)});
    \coordinate (bT) at ({0.5*cos(120)}, {0.5*sin(120)});
    \coordinate (cT) at ({0.5*cos(240)}, {0.5*sin(240)});
    \draw (aT) -- (bT);
    \draw (bT) -- (cT);
    \draw (cT) -- (aT);
    \draw (aT) -- (AT);
    \draw (bT) -- (BT);
    \draw (cT) -- (CT);
  \end{scope}
  \begin{scope}[yxplane=-1]
    \coordinate (AB) at ({cos(  0)}, {sin(  0)});
    \coordinate (BB) at ({cos(120)}, {sin(120)});
    \coordinate (CB) at ({cos(240)}, {sin(240)});
    \coordinate (aB) at ({0.5*cos(  0)}, {0.5*sin(  0)});
    \coordinate (bB) at ({0.5*cos(120)}, {0.5*sin(120)});
    \coordinate (cB) at ({0.5*cos(240)}, {0.5*sin(240)});
    \draw (aB) -- (bB);
    \draw (bB) -- (cB);
    \draw (cB) -- (aB);
    \draw (aB) -- (AB);
    \draw (bB) -- (BB);
    \draw (cB) -- (CB);
  \end{scope}
  \draw (AB) -- (AT);
  \draw (BB) -- (BT);
  \draw (CB) -- (CT);
  \draw (aB) -- (aT);
  \draw (bB) -- (bT);
  \draw (cB) -- (cT);
\end{scope}
}}}= 
 \brak{\NB{\tikz[scale=0.5]{\tdplotsetmaincoords{75}{135}
\begin{scope}[scale = 1.5, tdplot_main_coords]
  \tikzset{yxplane/.style={canvas is xy plane at z=#1}}
  \begin{scope}[yxplane=1]
    \coordinate (AT) at ({cos(  0)}, {sin(  0)});
    \coordinate (BT) at ({cos(120)}, {sin(120)});
    \coordinate (CT) at ({cos(240)}, {sin(240)});
    \coordinate (aT) at ({0.5*cos(  0)}, {0.5*sin(  0)});
    \coordinate (bT) at ({0.5*cos(120)}, {0.5*sin(120)});
    \coordinate (cT) at ({0.5*cos(240)}, {0.5*sin(240)});
    \draw (aT) -- (bT);
    \draw (bT) -- (cT);
    \draw (cT) -- (aT);
    \draw (aT) -- (AT);
    \draw (bT) -- (BT);
    \draw (cT) -- (CT);
  \end{scope}
  \begin{scope}[yxplane=-1]
    \coordinate (AB) at ({cos(  0)}, {sin(  0)});
    \coordinate (BB) at ({cos(120)}, {sin(120)});
    \coordinate (CB) at ({cos(240)}, {sin(240)});
    \coordinate (aB) at ({0.5*cos(  0)}, {0.5*sin(  0)});
    \coordinate (bB) at ({0.5*cos(120)}, {0.5*sin(120)});
    \coordinate (cB) at ({0.5*cos(240)}, {0.5*sin(240)});
    \draw (aB) -- (bB);
    \draw (bB) -- (cB);
    \draw (cB) -- (aB);
    \draw (aB) -- (AB);
    \draw (bB) -- (BB);
    \draw (cB) -- (CB);
  \end{scope}
  \coordinate (M1) at (0,0, -0.3);
  \coordinate (M2) at (0,0,  0.3);
  \draw (M1) -- (M2);
  \draw (AB) -- (AT);
  \draw (BB) -- (BT);
  \draw (CB) -- (CT);
  \draw (aB) -- (M1);
  \draw (bB) -- (M1);
  \draw (cB) -- (M1);
  \draw (aT) -- (M2);
  \draw (bT) -- (M2);
  \draw (cT) -- (M2);
\end{scope}
}}}.
 \end{align*}
\end{prop}

Similar computations establish the next two propositions. 

\begin{prop}\label{prop:IHI}
 The following local identity holds:
 \begin{align*}
 \brak{\NB{\tikz[scale=0.5]{\tdplotsetmaincoords{70}{20}
\begin{scope}[scale = 1.5, tdplot_main_coords]
  \tikzset{yxplane/.style={canvas is xy plane at z=#1}}
  \begin{scope}[yxplane=1]
    \coordinate (AT) at ({cos(  45)}, {sin(  45)});
    \coordinate (BT) at ({cos(135)}, {sin(135)});
    \coordinate (CT) at ({cos(225)}, {sin(225)});
    \coordinate (DT) at ({cos(315)}, {sin(315)});
    \coordinate (aT) at (0 ,  0.3);
    \coordinate (bT) at (0 , -0.3);
    \draw (aT) -- (AT);
    \draw (aT) -- (BT);
    \draw (aT) -- (bT);
    \draw (CT) -- (bT);
    \draw (DT) -- (bT);
  \end{scope}
  \begin{scope}[yxplane=-1]
    \coordinate (AB) at ({cos(  45)}, {sin(  45)});
    \coordinate (BB) at ({cos(135)}, {sin(135)});
    \coordinate (CB) at ({cos(225)}, {sin(225)});
    \coordinate (DB) at ({cos(315)}, {sin(315)});
    \coordinate (aB) at (0 ,  0.3);
    \coordinate (bB) at (0 , -0.3);
    \draw (aB) -- (AB);
    \draw (aB) -- (BB);
    \draw (aB) -- (bB);
    \draw (CB) -- (bB);
    \draw (DB) -- (bB);
  \end{scope}
  \begin{scope}[yxplane=0, densely dotted]
    \coordinate (AM) at ({cos(  45)}, {sin(  45)});
    \coordinate (BM) at ({cos(135)}, {sin(135)});
    \coordinate (CM) at ({cos(225)}, {sin(225)});
    \coordinate (DM) at ({cos(315)}, {sin(315)});
    \coordinate (aM) at ( 0.3,0);
    \coordinate (bM) at (-0.3,0);
    \draw (aM) -- (AM);
    \draw (aM) -- (DM);
    \draw (aM) -- (bM);
    \draw (CM) -- (bM);
    \draw (BM) -- (bM);
  \end{scope}
  \coordinate (OT) at (0,0, 0.5);
  \coordinate (OB) at (0,0,-0.5);
  \draw (AB) -- (AT);
  \draw (BB) -- (BT);
  \draw (CB) -- (CT);
  \draw (DB) -- (DT);
  \draw (aB) -- (OB);
  \draw (bB) -- (OB);
  \draw (aT) -- (OT);
  \draw (bT) -- (OT);
  \draw (aM) -- (OB);
  \draw (bM) -- (OB);
  \draw (aM) -- (OT);
  \draw (bM) -- (OT);
\end{scope}
}}}= 
 \brak{\NB{\tikz[scale=0.5]{\tdplotsetmaincoords{70}{20}
\begin{scope}[scale = 1.5, tdplot_main_coords]
  \tikzset{yxplane/.style={canvas is xy plane at z=#1}}
  \begin{scope}[yxplane=1]
    \coordinate (AT) at ({cos(  45)}, {sin(  45)});
    \coordinate (BT) at ({cos(135)}, {sin(135)});
    \coordinate (CT) at ({cos(225)}, {sin(225)});
    \coordinate (DT) at ({cos(315)}, {sin(315)});
    \coordinate (aT) at (0 ,  0.3);
    \coordinate (bT) at (0 , -0.3);
    \draw (aT) -- (AT);
    \draw (aT) -- (BT);
    \draw (aT) -- (bT);
    \draw (CT) -- (bT);
    \draw (DT) -- (bT);
  \end{scope}
  \begin{scope}[yxplane=-1]
    \coordinate (AB) at ({cos(  45)}, {sin(  45)});
    \coordinate (BB) at ({cos(135)}, {sin(135)});
    \coordinate (CB) at ({cos(225)}, {sin(225)});
    \coordinate (DB) at ({cos(315)}, {sin(315)});
    \coordinate (aB) at (0 ,  0.3);
    \coordinate (bB) at (0 , -0.3);
    \draw (aB) -- (AB);
    \draw (aB) -- (BB);
    \draw (aB) -- (bB);
    \draw (CB) -- (bB);
    \draw (DB) -- (bB);
  \end{scope}
  \draw (AB) -- (AT);
  \draw (BB) -- (BT);
  \draw (CB) -- (CT);
  \draw (DB) -- (DT);
  \draw (aB) -- (aT);
  \draw (bB) -- (bT);
  \fill ($(BT)!0.5!(bT)$) circle (0.5mm);
\end{scope}
}}}+
 \brak{\NB{\tikz[scale=0.5]{\tdplotsetmaincoords{70}{20}
\begin{scope}[scale = 1.5, tdplot_main_coords]
  \tikzset{yxplane/.style={canvas is xy plane at z=#1}}
  \begin{scope}[yxplane=1]
    \coordinate (AT) at ({cos(  45)}, {sin(  45)});
    \coordinate (BT) at ({cos(135)}, {sin(135)});
    \coordinate (CT) at ({cos(225)}, {sin(225)});
    \coordinate (DT) at ({cos(315)}, {sin(315)});
    \coordinate (aT) at (0 ,  0.3);
    \coordinate (bT) at (0 , -0.3);
    \draw (aT) -- (AT);
    \draw (aT) -- (BT);
    \draw (aT) -- (bT);
    \draw (CT) -- (bT);
    \draw (DT) -- (bT);
  \end{scope}
  \begin{scope}[yxplane=-1]
    \coordinate (AB) at ({cos(  45)}, {sin(  45)});
    \coordinate (BB) at ({cos(135)}, {sin(135)});
    \coordinate (CB) at ({cos(225)}, {sin(225)});
    \coordinate (DB) at ({cos(315)}, {sin(315)});
    \coordinate (aB) at (0 ,  0.3);
    \coordinate (bB) at (0 , -0.3);
    \draw (aB) -- (AB);
    \draw (aB) -- (BB);
    \draw (aB) -- (bB);
    \draw (CB) -- (bB);
    \draw (DB) -- (bB);
  \end{scope}
  \draw (AB) -- (AT);
  \draw (BB) -- (BT);
  \draw (CB) -- (CT);
  \draw (DB) -- (DT);
  \draw (aB) -- (aT);
  \draw (bB) -- (bT);
  \fill ($(CT)!0.15!(bB)$) circle (0.5mm);
\end{scope}
}}}.
 \end{align*}
\end{prop}

\begin{prop}\label{prop:unzipzip}
 The following local identity holds:
 \begin{align*}
 \brak{\NB{\tikz[scale=0.5]{\tdplotsetmaincoords{70}{20}
\begin{scope}[scale = 1.5, tdplot_main_coords]
  \tikzset{yxplane/.style={canvas is xy plane at z=#1}}
  \begin{scope}[yxplane=1]
    \coordinate (AT) at ({cos(  45)}, {sin(  45)});
    \coordinate (BT) at ({cos(135)}, {sin(135)});
    \coordinate (CT) at ({cos(225)}, {sin(225)});
    \coordinate (DT) at ({cos(315)}, {sin(315)});
    \coordinate (aT) at (0 ,  0.3);
    \coordinate (bT) at (0 , -0.3);
    \draw (aT) -- (AT);
    \draw (aT) -- (BT);
    \draw (aT) -- (bT);
    \draw (CT) -- (bT);
    \draw (DT) -- (bT);
  \end{scope}
  \begin{scope}[yxplane=-1]
    \coordinate (AB) at ({cos(  45)}, {sin(  45)});
    \coordinate (BB) at ({cos(135)}, {sin(135)});
    \coordinate (CB) at ({cos(225)}, {sin(225)});
    \coordinate (DB) at ({cos(315)}, {sin(315)});
    \coordinate (aB) at (0 ,  0.3);
    \coordinate (bB) at (0 , -0.3);
    \draw (aB) -- (AB);
    \draw (aB) -- (BB);
    \draw (aB) -- (bB);
    \draw (CB) -- (bB);
    \draw (DB) -- (bB);
  \end{scope}
  \begin{scope}[yxplane=0, densely dotted]
    \coordinate (AM) at ({cos(  45)}, {sin(  45)});
    \coordinate (BM) at ({cos(135)}, {sin(135)});
    \coordinate (CM) at ({cos(225)}, {sin(225)});
    \coordinate (DM) at ({cos(315)}, {sin(315)});
    \coordinate (aM) at (0, 0.3);
    \coordinate (bM) at (0, -0.3);
    \draw (AM) .. controls (aM) and (bM) .. (DM);
	\draw (BM) .. controls (aM) and (bM) .. (CM);
\end{scope}
  \coordinate (OT) at (0,0, 0.3);
  \coordinate (OB) at (0,0,-0.3);
  \draw (aT) .. controls +(0,0,-0.3) and +(0,0.3,0) .. (OT) .. controls +(0,-0.3, 0) and + (0,0,-0.3) .. (bT);
  \draw (aB) .. controls +(0,0, 0.3) and +(0,0.3,0) .. (OB) .. controls +(0,-0.3, 0) and + (0,0, 0.3) .. (bB);
\draw (AB) -- (AT);
  \draw (BB) -- (BT);
  \draw (CB) -- (CT);
  \draw (DB) -- (DT);
\end{scope}
}}}= 
 \brak{\NB{\tikz[scale=0.5]{\tdplotsetmaincoords{70}{20}
\begin{scope}[scale = 1.5, tdplot_main_coords]
  \tikzset{yxplane/.style={canvas is xy plane at z=#1}}
  \begin{scope}[yxplane=1]
    \coordinate (AT) at ({cos(  45)}, {sin(  45)});
    \coordinate (BT) at ({cos(135)}, {sin(135)});
    \coordinate (CT) at ({cos(225)}, {sin(225)});
    \coordinate (DT) at ({cos(315)}, {sin(315)});
    \coordinate (aT) at (0 ,  0.3);
    \coordinate (bT) at (0 , -0.3);
    \draw (aT) -- (AT);
    \draw (aT) -- (BT);
    \draw (aT) -- (bT);
    \draw (CT) -- (bT);
    \draw (DT) -- (bT);
  \end{scope}
  \begin{scope}[yxplane=-1]
    \coordinate (AB) at ({cos(  45)}, {sin(  45)});
    \coordinate (BB) at ({cos(135)}, {sin(135)});
    \coordinate (CB) at ({cos(225)}, {sin(225)});
    \coordinate (DB) at ({cos(315)}, {sin(315)});
    \coordinate (aB) at (0 ,  0.3);
    \coordinate (bB) at (0 , -0.3);
    \draw (aB) -- (AB);
    \draw (aB) -- (BB);
    \draw (aB) -- (bB);
    \draw (CB) -- (bB);
    \draw (DB) -- (bB);
  \end{scope}
  \draw (AB) -- (AT);
  \draw (BB) -- (BT);
  \draw (CB) -- (CT);
  \draw (DB) -- (DT);
  \draw (aB) -- (aT);
  \draw (bB) -- (bT);
  \fill ($(AT)!0.2!(aB)$, 0) circle (0.5mm);
\end{scope}
}}}+
 \brak{\NB{\tikz[scale=0.5]{}}}.
 \end{align*}
\end{prop}

\begin{prop}[Handle removal]\label{prop:handleremoval}
 The following local identity holds:
 \begin{align*}
 \brak{\!\!\NB{\tikz[scale=0.8]{\begin{scope}
  \coordinate (A) at (0,0);
  \coordinate (B) at (2,0);
  \coordinate (C) at (2.5,0.5);
  \coordinate (D) at (0.5,0.5);
  \coordinate (d) at (1, 0.5);
  \coordinate (c) at (1.5, 0.5);
  \coordinate (m1) at (0.5, 0.25);
  \coordinate (m2) at (2, 0.25);
  \coordinate (T) at (1.25, 1.25);
  \coordinate (M1) at (0.9, 1);
  \coordinate (M2) at (1.6, 1);
  \draw[dotted] (c) -- (d);
  \draw (d) -- (D) -- (A) -- (B) -- (C) -- (c);
  \draw (m1)  .. controls +(0.2,0) and +(0, -0.2) .. (d)  .. controls +(0,0.2) and +(-1.5, 0) .. (T) 
   .. controls +(1.5,0) and +(0, 0.2) .. (c)  .. controls +(0,-0.2) and +(-0.2, 0) .. (m2);
  \draw (M1)  .. controls +(0.2,-0.1) and +(-0.2, -0.1) .. (M2) coordinate [pos =0.2] (n1) coordinate [pos= 0.8] (n2);
  \draw (n1)  .. controls +(0.16, 0.08) and +(-0.16, 0.08) .. (n2);
\end{scope}}}\!\!}= 
 \brak{\decorateddisk{\bullet\,\bullet}{0.7}{0.8}}
 +E_2
 \brak{\decorateddisk{ }{0.7}{0.8}}.
 \end{align*}
\end{prop}

\begin{proof}
The relation follows immediately from the neck-cutting relation.
\end{proof}

This proposition shows that a handle on a facet can be removed at the cost of multiplying  polynomial floating on the facet by $\bullet^2+E_2$. 

\begin{prop}[Bubble removal] \label{prop:bubbleremoval} Let  $F_{n,m}$ be obtained from an admissible pre-foam $F$ by adding a bubble that flows on a facet of $F$, with $n$ and $m$ dots, respectively,  
on the new facets, for $n,m\le 2$. 
Let $F_n$ be foam $F$ with $n$ dots added to the same facet of $F$. Then 
\begin{align*}
 \brak{F_{n,n}} & =  0, \ \ n\ge 0,   \\
  \brak{F_{0,1}} & =  \brak{F},  \\
 \brak{F_{0,2}} & =   \brak{F_1} + E_1 \brak{F},  \\
  \brak{F_{1,2}} & =  \brak{F_2}+ E_1 \brak{F_1} + E_2 \brak{F}, 
  \end{align*}
  where  corresponding facets of pre-foams $F_n$ and $F_{n,m}$ are depicted in Figure~\ref{fig:FiFnm}.
  \begin{figure}[ht]
  	\centering	
    \tikz{
    	\begin{scope}
    	\begin{scope}[scale = 0.5]
\draw (0,0) circle (1.5 and 0.3);
\draw (1.5, 0) arc (0:180:1.5);
\draw (1.5, 0) arc (0:-180:1.5);
\draw[white, dashed, thick] (1.5, 0) arc (0:-19:1.5);
\draw[white, dashed, thick] (-1.5, 0) arc (180:199:1.5);
\draw (-2,0.5) -- (3,0.5) -- (2, -0.5) -- (-3, -0.5) -- cycle;
\draw[white, thick, dashed] (160.3:1.49)-- (19.7:1.49);
\fill (0,-1) circle (1mm) node [scale = 0.7, right] {$n$};
\fill (0,1) circle (1mm) node [scale = 0.7, right] {$m$};
\end{scope}
       	\end{scope}
    	\begin{scope}[xshift = 5cm, yshift = -0.25cm]
    	\input{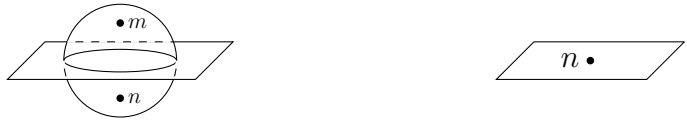}
       	\end{scope}
        }
    \caption{On the left, pre-foam $F_{n,m}$; on the right, pre-foam $F_n$.}
    \label{fig:FiFnm}
  \end{figure}
\end{prop} 
\begin{proof} Straightforward. 
\end{proof}

\begin{prop}\label{prop:handleremoval2}
 The following local identity holds:
 \begin{align*}
 \brak{\!\!\NB{\tikz[scale=0.8]{\begin{scope}
  \coordinate (A) at (0,0);
  \coordinate (B) at (2,0);
  \coordinate (C) at (2.5,0.5);
  \coordinate (D) at (0.5,0.5);
  \coordinate (d) at (1, 0.5);
  \coordinate (c) at (1.5, 0.5);
  \coordinate (m1) at (0.5, 0.25);
  \coordinate (m2) at (2, 0.25);
  \coordinate (T) at (1.25, 1.25);
  \coordinate (M1) at (0.9, 1);
  \coordinate (M2) at (1.6, 1);
  \filldraw[fill = gray!50!white]  (c) arc (0:-180: 0.25cm and 0.08cm);
  \filldraw[densely dotted, fill = gray!50!white]  (c) arc (0:180: 0.25cm and 0.08cm);
  \draw (d) -- (D) -- (A) -- (B) -- (C) -- (c);
  \draw (m1)  .. controls +(0.2,0) and +(0, -0.2) .. (d)  .. controls +(0,0.2) and +(-1.5, 0) .. (T) 
   .. controls +(1.5,0) and +(0, 0.2) .. (c)  .. controls +(0,-0.2) and +(-0.2, 0) .. (m2);
  \draw (M1)  .. controls +(0.2,-0.1) and +(-0.2, -0.1) .. (M2) coordinate [pos =0.2] (n1) coordinate [pos= 0.8] (n2);
  \draw (n1)  .. controls +(0.16, 0.08) and +(-0.16, 0.08) .. (n2);
\end{scope}}}\!\!}=
 \brak{\decorateddisk{\bullet}{0.7}{0.8}}
+ E_1 
\brak{\decorateddisk{ }{0.7}{0.8}}.
 \end{align*}
\end{prop}

\begin{proof} Straightforward. 
\end{proof} 

\begin{prop}[Dot migration] \label{prop:dotmigration}Let $F$ be an admissible pre-foam with a seam edge. 
Label three portions of facets of $F$ bounding the edge by $1,2,3$ and 
denote by $F_{(n_1,n_2,n_3)}$ for $i\in \{1,2,3\}$ and $n_i\ge 0$ the  
pre-foam given by adding $n_i$ dots to the facet portion of $F$
labeled $i$, see Figure~\ref{fig:vertexIdots}. Then 
\begin{align*}
\brak{F_{(1,0,0)}} + \brak{F_{(0,1,0)}} + \brak{F_{(0,0,1)}} & =  
E_1 \brak{F}, \\
\brak{F_{(2,0,0)}} + \brak{F_{(0,2,0)}} + \brak{F_{(0,0,2)}} & =  
E_1^2 \brak{F}, \\
\brak{F_{(1,1,0)}} + \brak{F_{(1,0,1)}} + \brak{F_{(0,1,1)}}& =  E_2 \brak{F}, \\
\brak{F_{(1,1,1)}} & =  E_3 \brak{F}. 
\end{align*}
The first relation is depicted diagrammatically in Figure~\ref{fig:singleMigration}. 
\begin{figure}[ht]
\tikz[scale=0.7]{\input{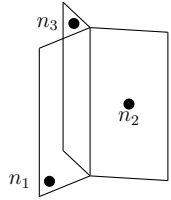}}
\caption{Pre-foam $F_{(n_1, n_2, n_3)}$.}\label{fig:vertexIdots}
\end{figure}

\begin{figure}[ht]
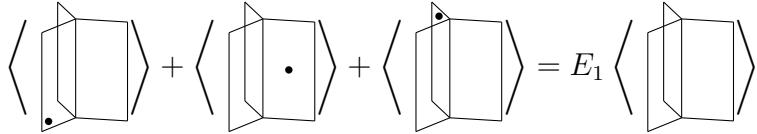

\[
\brak{\vertexIdottedonezerozero[0.7]} + 
\brak{\vertexIdottedzeroonezero[0.7]} + 
\brak{\vertexIdottedzerozeroone[0.7]} = E_1 \brak{\vertexInodot[0.7]}
\]
\caption{One of the dot migration relations}
\label{fig:singleMigration}
\end{figure}

\end{prop} 

\begin{proof} Direct computation. 
\end{proof} 

\begin{prop} \label{prop:3dots}
The following local identity holds:
 \begin{align*}
 \brak{\decorateddisk{\bullet\bullet\bullet}{0.7}{0.8}}
= 
E_1 \brak{\decorateddisk{\bullet\,\,\bullet}{0.7}{0.8}}
+ E_2\brak{\decorateddisk{\bullet}{0.7}{0.8}}
+
E_3
 \brak{\decorateddisk{ }{0.7}{0.8}}.
\end{align*}
\end{prop} 

\begin{proof} This follows from the identity 
$X_i^3= E_1 X_i^2 + E_2 X_i + E_3$ that holds in the ring $R$ for 
$i=1,2,3$. 
\end{proof}

\subsection{Kronheimer--Mrowka evaluation for 
 foams} 
\label{sec:combinatorialKM}\

Assume there is a homomorphism 
$\psi: R \longrightarrow S$ of commutative rings. Define 
the $\psi$-evaluation of a closed pre-foam $F$ as 
$\psi(\brak{F}) \in S$, by composing the evaluation with values in 
$R$ with the homomorphism $\psi$. We denote $\psi$-evaluation 
by $\brak{F}_{\psi}$ and also call it $S$-evaluation and denote 
$\brak{F}_S$ when it's clear what $\psi$ is from the context. 
In this subsection we'll use homomorphism $\psi: R \lra \kk$ into 
the two-element ground field $\kk$ with $\psi(E_i)=0$ for $i=1,2,3$. Note that 
$\psi$ is a grading-preserving homomorphism, with $\kk$ necessarily in degree zero. 
   
Kronheimer and Mrowka \cite[Section 8.3]{KM1} suggest a combinatorial counterpart of their 
homology for planar graphs and conjecture that it is well-defined. Here we briefly review their approach and relate it to foam evaluation.

Let $F$ be a closed pre-foam. Kronheimer-Mrowka's  algorithm aims to define an element $J^\flat(F)\in \kk$ associated with $F$. 
\begin{enumerate}
\item\label{item:seamV} If $s(F)$ is not bipartite, set $J^\flat(F) =0$. If $s(F)$ is bipartite, choose a perfect matching of $s(F)$ and cancel all the seam vertices using the relation in Proposition~\ref{prop:verticesremoval}. Hall's Marriage Theorem implies that any regular 
bipartite graph has a perfect matching~\cite[Theorem 2.1.2 and Corollary 2.1.4]{Diestel}.   
This results in a new pre-foam $F'$. Set $J^\flat(F) = J^\flat(F')$. Suppose from now on that the pre-foam $F$ has no seam vertices.
\item\label{item:monodromy} Now $s(F)$ is a collection of circles. If the monodromy of the three facets along one of the  circles is non-trivial, set $J^\flat(F) = 0$. Suppose from now on that a regular neighborhood of $s(F)$ is homeomorphic to a disjoint union of $Y\times \SS^1$, where $Y$ is the standard tripod. 
\item\label{item:neckcut} For each component of the seam having a neighborhood of the form $\SS^1\times Y$, apply neck-cutting  \cite[Proposition 6.1]{KM1} on the three circles parallel to the seam in the three neighboring facets. Neck-cutting in this algorithm 
is the same as the specialization of neck-cutting 
in Proposition~\ref{prop:neckcutting} to the quotient ring $\kk$, where $E_1=E_2=E_3=0$.
In particular, there are only three terms on the right hand side of the relation in Proposition~\ref{prop:neckcutting}, since $E_1=E_2=0$. 
The neck-cutting relation reduces computing $J^\flat(F)$ to 
the case when $F$ is a collection of dotted theta pre-foams and  dotted closed surfaces.
\item\label{item:ThetaSurface} Set $J^\flat$ to be multiplicative under the disjoint union. For theta pre-foams $\theta(n_1,n_2,n_3)$  with $n_1$, $n_2$, and $n_3$ dots on the three disks and 
$n_1\ge n_2 \ge n_3$, define $J^{\flat}(\theta(2,1,0))=1$
and $J^{\flat}(\theta(n_1,n_2,n_3))=0$ for all other triples of non-increasing numbers. 
Define $J^{\flat}$ to be $1$ on a sphere with two dots and on a dotless torus and to be $0$ on all other closed connected surfaces
(that may carry dots). In particular, any unorientable surface  evaluates to $0$ under $J^{\flat}$. 
\end{enumerate}

\begin{conjecture}[{\cite[Conjecture 8.9]{KM1}}]
\label{conj:JflatKM}
The quantity $J^\flat(F)$ is well-defined: it does not depend on the choices made in step~(\ref{item:seamV}).
\end{conjecture}

Recall that we denote by $\brak{F}_{\kk}$ 
the image of $\brak{F}\in R$ under the ring homomorphism $R\lra \kk$ sending $E_1, E_2, E_3$ to $0$. This homomorphism kills
$R$ in all positive degrees, keeping only the ground field $\kk$, 
which is exactly the degree zero part of $R$. 

\begin{theorem}
\label{thm:JflatFoams}
If $F$ is embeddable in $\R^3$, then $J^\flat(F)$ is well-defined and equal to $\brak{F}_{\kk}$.
\end{theorem}

\begin{proof} We start by proving a sequence of lemmas. 

\begin{lemma}\label{lem:Jflatdeg0}
If a pre-foam $F$ has a non-zero degree, $J^\flat(F)$ is well-defined and equal to 0. 
\end{lemma}

\begin{proof}
The non-deterministic rules given to evaluate $J^\flat$ respect the degree, i.e., at each step a pre-foam of a given degree is simplified into a linear combination of pre-foams of the same degree. At the last step, only elementary pre-foams of degree $0$ are evaluated to a non-zero value.
\end{proof}

\begin{lemma}\label{lem:JflatUnoriented}
If a pre-foam $F$ has a non-orientable facet, $J^\flat(F)$ is well-defined and equal to 0. 
\end{lemma}

\begin{proof}
The non-orientability of a facet is preserved by steps (\ref{item:seamV}), (\ref{item:monodromy}) and (\ref{item:neckcut}).  
\end{proof}

\begin{lemma}\label{lem:deg0foam}
Let $F$ be a connected pre-foam of degree $0$ with no seam vertices such that all its facets are orientable. Then one of the following holds:
\begin{itemize}
\item $F$ has a disk facet,
\item all facets of $F$ are annuli and $F$ carries no dots,
\item $F$ is a sphere with two dots,
\item $F$ is a dotless torus.
\end{itemize}
\end{lemma}

\begin{proof}
The degree of a pre-foam $F$ with no seam vertices is given by
\[
\deg(F) = 2\ |d(F)| - 2\sum_{f\in f(F)} \chi(f),
\]
see Definition~\ref{def:degree}. 
Let $F$ be a pre-foam as in the lemma.
If $F$ has no seam circles, it is a surface. This surface is orientable and therefore it is either a torus (with no dots) or a sphere with two dots. 
  
If $F$ is not a surface, all of its facets have boundaries. The only way for a facet to have positive Euler characteristic is to be a disk. Likewise, the only way for a facet to have zero Euler characteristic is to be an annulus. This shows that if none of the facets of $F$ is a disk, then all its facets are annuli and $F$ carries no dots. 
\end{proof}

\begin{lemma}\label{lem:collofann}
Let $F$ be a (non-empty) pre-foam without seam vertices such that every facet of $F$ is an annulus. Then $J^\flat(F)=0$. 
\end{lemma}

\begin{figure}[h]
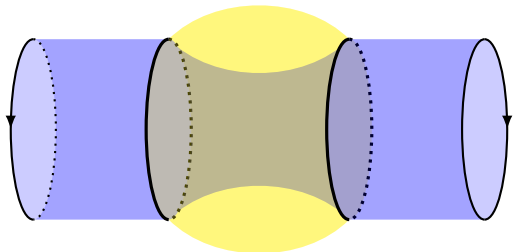

\centering
\tikz[scale = 0.6]{\begin{scope}[yscale =2,very thick, label distance = -0.1cm]
  \fill[opacity = 0.2, blue] (-1,1) -- (2,1) arc (90:270:0.5 and 1) --  (-1, -1) arc (270:90:0.5 and 1);
  \fill[opacity = 0.2, blue] (-1,1) -- (2,1) arc (90:-90:0.5 and 1) -- (-1, -1) arc (-90:90:0.5 and 1); 
  \draw[thick, black, ->-] (-1,1) arc (90:270:0.5 and 1);
  \draw[thick, black, dotted] (-1,1) arc (90:-90:0.5 and 1);
  \fill[opacity = 0.3, yellow] (2,1) .. controls +(1,  0.5) and +(-1, 0.5) .. (6,1) arc (90:-90:0.5 and 1) .. controls +(-1, -0.5) and +( 1, -0.5) .. (2, -1) arc (-90:90:0.5 and 1);
  \fill[opacity = 0.2, blue] (2,1) .. controls +(1, -0.5) and +(-1, -0.5) .. (6,1) arc (90:270:0.5 and 1) .. controls +(-1, +0.5) and +( 1, 0.5) .. (2, -1) arc (270:90:0.5 and 1);
  \draw[ black, dotted, name path= line 3] ( 2,1) arc (90:-90:0.5 and 1);
  \fill[opacity = 0.2, blue] (2,1) .. controls +(1, -0.5) and +(-1, -0.5) .. (6,1) arc (90:-90:0.5 and 1) .. controls +(-1, +0.5) and +( 1, 0.5) .. (2, -1) arc (-90:90:0.5 and 1);
  \fill[opacity = 0.3, yellow] (2,1) .. controls +(1, +0.5) and +(-1, +0.5) .. (6,1) arc (90:270:0.5 and 1) .. controls +(-1, -0.5) and +( 1, -0.5) .. (2, -1) arc (270:90:0.5 and 1);
  \draw[ black , name path = line 4] ( 2,1) arc (90:270:0.5 and 1) coordinate[pos= 0.25] (ep1);
  \fill[opacity = 0.2,blue] (6,1) -- (9,1) arc (90:-90:0.5 and 1) -- (6, -1) arc (-90:90:0.5 and 1);
  \draw[ black, dotted, name path = line 5] (6,1) arc (90:-90:0.5 and 1);
  \fill[opacity = 0.2, blue] (6,1) -- (9,1) arc (90:270:0.5 and 1) -- (6, -1) arc (270:90:0.5 and 1);
  \draw[ black, name path = line 6] (6,1) arc (90:270:0.5 and 1) coordinate[pos= 0.25] (ep2);
  \draw[black, thick, ->- ] (9,1) arc (90:-90:0.5 and 1);
  \draw[thick, black] (9,1) arc (90:270:0.5 and 1);
\end{scope}}
\caption{A Klein bottle (in blue) onto which is glued an annulus (in yellow) along two parallel circles.  This is an example of a pre-foam fulfilling the hypothesis of Lemma~\ref{lem:collofann} ( every facet is an annulus), which is not a web times $\SS^1$.}\label{fig:anotherKleinBottle}
\end{figure}

Before moving onto a proof, note that can take the direct product of a trivalent graph $G$, not necessarily planar, and $\SS^1$, to get such a pre-foam. More generally, choosing a trivalent graph $G$ and an element $\alpha$ of $H^1(G,\Z/2)$, one can form a pre-foam by taking 
the product of $\SS^1$ and the set of vertices of $G$ as the set of seam circles of pre-foam, and gluing annuli to the product, one for each edge of $G$. The gluing is such that for any 1-cycle $y$ in $G$ its preimage in the pre-foam is a 2-torus if $\alpha(y)=0$ and a Klein bottle in $\alpha(y)=1$. 

Such an example is depicted in Figure~\ref{fig:anotherKleinBottle}. The underlying graph is a $\theta$-graph. The cohomology class $\alpha$ is equal to $1$ on two simple cycles and zero on the third. 

\begin{proof}
We know that $J^{\flat}(F)$ is well-defined for pre-foams without seam vertices, since there are no choices to make in the evaluation algorithm.

Let $F$ be a pre-foam satisfying the hypothesis of the proposition. If $F$ carries some dots, then its degree is positive and therefore  $J^\flat(F) = 0$. If the monodromy of the facets along a circle is non-trivial, then $J^\flat(F)=0$ thanks to step~(\ref{item:monodromy}) in the algorithm.
Else, we consider the graph $G_F$ given by the following data:
\begin{itemize}
\item The vertices of $G_F$ are seam circles of $F$.
\item The edges of $G_F$ are facets of $F$. They join their two boundary components. 
\end{itemize}
Thus the graph $G_F$ is trivalent.
We allow a degenerate case of an annulus facet that bounds the same circle on both sides.
The closure of that facet is either a two-torus or a Klein bottle, but the facet itself it orientable).

Let us denote by $v$ (resp. $e$) the number of vertices (resp. edges) of $G_F$. We have $3v = 2e$. 
After performing step (\ref{item:neckcut}) in the Kronheimer--Mrowka algorithm, we end up with a sum $S$ of $3^{3v}$ terms. Each of these terms is a disjoint union of $v$ dotted theta pre-foams and $e$ dotted spheres. The evaluation of each of these terms is either $0$ or $1$. We want to prove that the number of terms which evaluate to $1$ is even.

A sphere evaluates to $0$ unless it carries exactly two dots and a theta pre-foam evaluates to $0$ unless its three facets carry exactly $0$, $1$ and $2$ dots. Hence $J^\flat(F)$ is equal (in $\kk$) to the number of terms in $S$ which are unions of spheres with two dots and $(2,1,0)$-theta pre-foams. 
  
On each annulus we perform two neck-cuttings, yielding nine terms with a dotted sphere. The sphere has two dots in only three out of this nine terms, and these three terms correspond to a single neck-cutting, in the middle of the annulus. Hence, instead of performing two neck-cutting operations per annulus, we can only perform one.

For each facet $f$ of $F$, let us encode the three terms in the neck-cutting relation by a semi-orientation of the corresponding edge $e$ in $G_F$. An oriented edge $e$ contributes to the cutting with the term which places two dots on the half-sphere bounding the circle corresponding to the target vertex relative to the orientation, and no  dots on
the opposite half-sphere. 
 If an edge is not oriented, it contributes one dot to each half-sphere into which the annulus  is split by neck-cutting. 
 
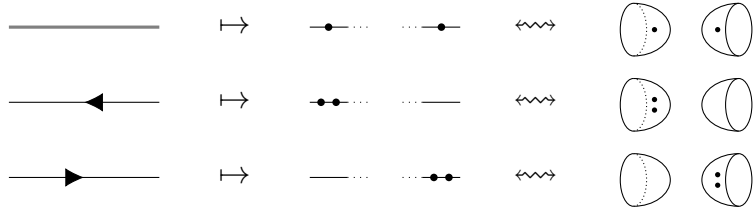
\[
\NB{
\tikz{
\begin{scope}[xshift= -0.5cm]
\tikzset{>={Latex[width=2.5mm,length=2.5mm]}}
\draw[->-] (0,0) -- +(2,0); 
\draw[->-] (2,1) -- +(-2,0); 
\draw[very thick, gray] (0,2) -- +(2,0); 
\end{scope}
\node at (2.5,0) {$\mapsto$};
\node at (2.5,1) {$\mapsto$};
\node at (2.5,2) {$\mapsto$};

\node at (6.5,0) {$\leftrightsquigarrow$};
\node at (6.5,1) {$\leftrightsquigarrow$};
\node at (6.5,2) {$\leftrightsquigarrow$};



\node at (8.5,0) {$\NB{\rotatebox{-90}{\thstwozero[0.7]}}$};
\node at (8.5,1) {$\NB{\rotatebox{-90}{\thszerotwo[0.7]}}$};
\node at (8.5,2) {$\NB{\rotatebox{-90}{\thsoneone[0.7]}}$};

\begin{scope}[xshift= 0.5cm]
\draw (3,0) -- +(0.5,0); 
\draw (3,1) -- +(0.5,0); 
\draw (3,2) -- +(0.5,0); 
\draw[dotted] (3.5,0) -- +(0.3,0); 
\draw[dotted] (3.5,1) -- +(0.3,0); 
\draw[dotted] (3.5,2) -- +(0.3,0); 
\draw[dotted] (4.5,0) -- +(-0.3,0); 
\draw[dotted] (4.5,1) -- +(-0.3,0); 
\draw[dotted] (4.5,2) -- +(-0.3,0); 
\draw (4.5,0) -- +(0.5,0); 
\draw (4.5,1) -- +(0.5,0); 
\draw (4.5,2) -- +(0.5,0); 
\fill (4.85,0) circle (0.5mm);
\fill (4.65,0) circle (0.5mm);
\fill (3.15,1) circle (0.5mm);
\fill (3.35,1) circle (0.5mm);
\fill (4.75,2) circle (0.5mm);
\fill (3.25,2) circle (0.5mm);
\end{scope}

}
}
\]

A term of $S$ which evaluates to $1$ is encoded by a semi-orientation of $G_F$ such that at each vertex one edge points in, one edge points out, and one edge is unoriented. We claim that the number of such semi-orientations is even. Indeed, there is an involution without fixed points on the set of such semi-orientations given by reversing the orientations of all edges (see Figure~\ref{fig:partialorientation}).  
  
\begin{figure}[ht]
\tikz{\input{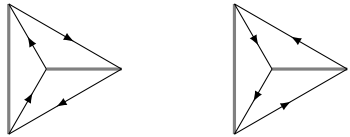}}
\caption{A graph $G_F$ and partial orientations related by the involution.}
\label{fig:partialorientation}
\end{figure}
\end{proof}

\begin{lemma}\label{lem:collofann2}
Let $F$ be a non-empty foam without seam vertices such that every facet of $F$ is an annulus. Then $\brak{F}_\kk=0$. 
\end{lemma}
  
\begin{proof}
If $F$ carries some dots, then its degree is non-zero and the result is immediate. Suppose that $F$ carries no dots. Then every bicolored surface is a collection of tori and have therefore Euler characteristic equal to $0$. This implies that $\brak{F}_\kk$ is equal to the number of admissible colorings modulo $2$. Since $F$ is non-empty, it contains at least one facet and hence contains one seam circle. There is an action of $S^3$ on the set of admissible colorings by permuting the colors. This action has no fixed point (it suffices to look at the color of the facets adjacent to a seam circle).
This proves that the number of admissible colorings of $F$ is even. Finally $\brak{F}_\kk=0$. 
\end{proof}
  
\begin{lemma}\label{lem:emb-equal}
For any pre-foam $F$ without seam vertices that admits an embedding in $\R^3$ 
\[J^\flat(F)= \brak{F}_{\kk}.\]
\end{lemma}

\begin{proof}  
Since both $J^\flat(\bullet)$ and $\brak{\bullet}_{\kk}$ are multiplicative with respect to the disjoint union, we can assume that the pre-foam $F$ is connected. Thanks to  Lemma~\ref{lem:Jflatdeg0}, we can suppose that $F$ has degree $0$.

We prove the lemma by induction on the number of seam circles. 
If there is no seam circle, $F$ is a collection of surfaces and the result is clear. Otherwise we apply Lemma~\ref{lem:deg0foam}: $F$ is either a theta pre-foam (in which case the result is clear), has only annulus-like facets (in this case the result follows from Lemmas~\ref{lem:collofann} and \ref{lem:collofann2}), or contains a disk. If it contains a disk, denote by $C$ the circle bounded by this disk. We can use the neck-cutting relation "non-abstractly", for foams rather than pre-foams,  on the three facets bounding $C$. This operation and the matching operation for $J^{\flat}(F)$ allows to express the values of $J^\flat(F)$ and $\brak{F}_{\kk}$ as sums of evaluations of the union of a dotted theta-foam and foams with fewer seam circles. We can now apply induction on the number of seam circles
to conclude that $J^\flat(F)=\brak{F}_{\kk}$ for $F$ as in the lemma. 
\end{proof}

To complete the proof of Theorem~\ref{thm:JflatFoams}, choose any foam $F$. If the graph $s(F)$ is not bipartite, 
$J^{\flat}(F)$ is well-defined, equal to $0$, and $\brak{F}_{\kk}=0$ as well, since $F$ has no admissible colorings thanks to Proposition~\ref{prop:admis2bipartite}. If $s(F)$ is bipartite, the evaluation $\brak{F}_{\kk}=\brak{F'}_{\kk}$ for any reduction of $F$ to a foam $F'$ without seam vertices via canceling of pairs of vertices along edges in a perfect matching in $s(F)$. In view of Lemma~\ref{lem:emb-equal} this matches the procedure in the algorithm, showing that 
$J^{\flat}(F) = \brak{F}_{\kk}$ for any such $F$. The theorem follows. 
\end{proof}

\begin{rmk} 
Instead of using a perfect matching in step (1) of the algorithm, one can choose to cancel pairs of vertices recursively, possibly via new edges created by earlier cancellations, slightly generalizing the algorithm and the invariance property of the evaluation. Theorem~\ref{thm:JflatFoams} holds for the generalized algorithm as well. 
\end{rmk}

\subsubsection*{Mismatched evaluations}

Consider the pre-foam $F$ constructed by adding three disks to a Klein bottle, as depicted in Figure~\ref{fig:counterexample0}. 
The tube depicts an annulus portion of the Klein bottle, and the oriented boundary circles are identified to match their orientations as shown in the figure, resulting in a Klein bottle. Three disks are attached to the Klein bottle, along the circles $c,c_1,c_2$ as shown. 
The circles $c_1$, $c_2$ are homotopic on the Klein bottle, while $c$ is contractible. The pre-foam $F$ carries a single dot, placed on the disk that bounds $c$. 
Facet portions $A_1$ and $A_2$ are parts of the same facet, denoted $A$. Note that all facets of $F$ are 
orientable: they consists of 7 disks and one annulus. 
   
The three circles have four intersection points, denoted $v_1$ through $v_4$. The graph $s(F)$ has four vertices $v_1, \dots, v_4$
and eight edges connecting them. We consider the edges denoted 
$e_1, e_2$, $e_3,e_4$, respectively, that span a four-cycle 
in the graph. These four edges together bound a square facet of the pre-foam and belong to circles $c, c_1, c, c_2$, respectively. The portion of the pre-foam shown is embeddable in $\R^3$, but not the entire pre-foam.  
   
\begin{figure}[ht]
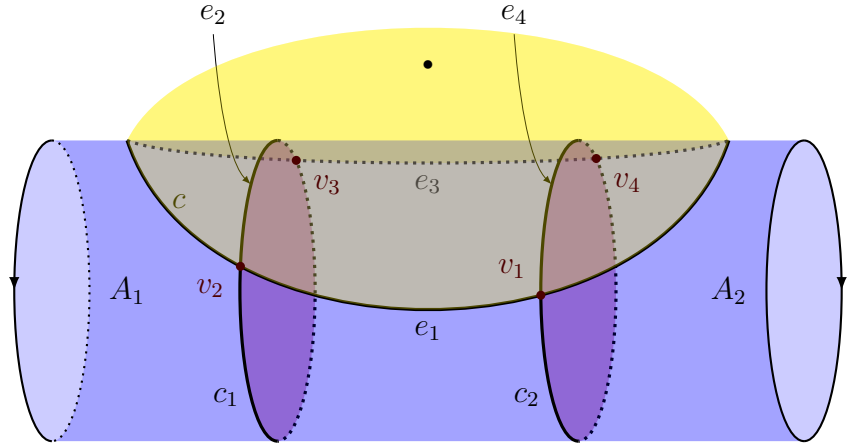

\tikz{\begin{scope}[yscale =2,very thick, label distance = -0.1cm]
\fill[yellow, opacity=0.3 ] (0,1) .. controls + (1, -0.2) and + (-1, -0.2) .. (8,1) .. controls +(-1, 1) and +(1, 1) .. (0,1) ;
  \draw[dotted, name path= line 1] (0,1) .. controls + (1, -0.2) and + (-1, -0.2) .. (8,1)
  node[pos =0.5, below] {$e_3$}; 
    \fill[opacity = 0.2, blue] (-1,1) -- (2,1) arc (90:-90:0.5 and 1) -- (-1, -1) arc (-90:90:0.5 and 1);
    \fill[opacity = 0.2, blue] (-1,1) -- (2,1) arc (90:-90:0.5 and 1) --  (-1, -1) arc (270:90:0.5 and 1);
    \fill[red, opacity = 0.2] (2,0) circle (0.5 and 1);
    \draw[thick, black, ->-] (-1,1) arc (90:270:0.5 and 1);
    \draw[thick, black, dotted] (-1,1) arc (90:-90:0.5 and 1);
    \fill[opacity = 0.2, blue] (2,1) -- (6,1) arc (90:-90:0.5 and 1) -- (2, -1) arc (-90:90:0.5 and 1);
    \draw[ black, dotted, name path= line 3] ( 2,1) arc (90:-90:0.5 and 1);
    \fill[opacity = 0.2, blue] (2,1) -- (6,1) arc (90:-90:0.5 and 1) -- (2, -1) arc (270:90:0.5 and 1);
    \draw[ black , name path = line 4] ( 2,1) arc (90:270:0.5 and 1) coordinate[pos= 0.25] (ep1) node[pos = 0.75, left] {$c_1$};;
    \draw[very thin, <-] (ep1) .. controls +(-0.4, 0.2) and +(0,0) .. +(-0.5,1) node [above] {$e_2$};
    \fill[red, opacity = 0.2] (6,0) circle (0.5 and 1);
    \fill[opacity = 0.2,blue] (6,1) -- (9,1) arc (90:-90:0.5 and 1) -- (6, -1) arc (-90:90:0.5 and 1);
    \draw[ black, dotted, name path = line 5] (6,1) arc (90:-90:0.5 and 1);
    \fill[opacity = 0.2, blue] (6,1) -- (9,1) arc (90:270:0.5 and 1) -- (6, -1) arc (270:90:0.5 and 1);
    \draw[ black, name path = line 6] (6,1) arc (90:270:0.5 and 1) coordinate[pos= 0.25] (ep2) node[pos = 0.75, left] {$c_2$};
    \draw[very thin, <-] (ep2) .. controls +(-0.4, 0.2) and +(0,0) .. +(-0.5,1) node [above] {$e_4$};
    \draw[black, thick, ->- ] (9,1) arc (90:-90:0.5 and 1);
    \draw[thick, black] (9,1) arc (90:270:0.5 and 1);
    \draw[name path= line 2] (0,1) .. controls + (1, -1.5) and + (-1, -1.5) .. (8,1) node[pos =0.5, below] {$e_1$} node[pos = 0.1, right] {$c$};
    \fill[yellow, opacity=0.3] (0,1) .. controls + (1, -1.5) and + (-1, -1.5) .. (8,1) .. controls +(-1, 1) and +(1, 1) .. (0,1) ;
 \path [name intersections={of=line 1 and line 3,by=v1}];
 \path [name intersections={of=line 1 and line 5,by=v2}];
 \path [name intersections={of=line 2 and line 4,by=v3}];
 \path [name intersections={of=line 2 and line 6,by=v4}];
\fill[red!30!black] (v1) circle (0.6mm and 0.3mm) node[label=-30:$v_3$] {};
\fill[red!30!black] (v2) circle (0.6mm and 0.3mm) node[label=-30:$v_4$] {};
\fill[red!30!black] (v3) circle (0.6mm and 0.3mm) node[label=-170:$v_2$] {};
\fill[red!30!black] (v4) circle (0.6mm and 0.3mm) node[label=120:$v_1$] {};
 \node[scale = 0.7] at (4,1.5) {$\bullet$}; 
 \node at (0,0) {$A_1$};
 \node at (8,0) {$A_2$};
\end{scope}}\caption{Pre-foam $F$; the annulus part of the Klein bottle is shown in blue and grey. Disk bounding $c$ and carrying a dot is shown in yellow.}
\label{fig:counterexample0}
\end{figure}

The graph $s(F)$ is bipartite. The edges $e_1,e_3$ constitute one 
possible perfect matching of $s(F)$, another is given by the edges $e_2,e_4$. 
Let us apply step~(\ref{item:seamV}) of the algorithm to the perfect matching $\{e_1,e_3\}$, 
canceling the vertices $(v_1,v_2)$ and $(v_3,v_4)$ in pairs along the seams $e_1$ and $e_3$. 
In the resulting pre-foam $F_0$, the facet $A$ acquires two additional strips, turning it into an unorientable surface with boundary, see Figure~\ref{fig:counterexample1}. Consequently, 
$J^{\flat}(F_0)=0$, since $F_0$ contains an unorientable facet.

\begin{figure}[ht]
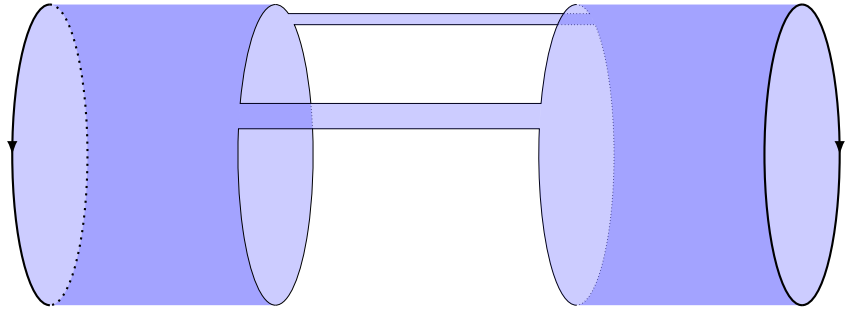

\tikz{\begin{scope}[yscale =2,very thick, label distance = -0.1cm]
    \fill[opacity = 0.2, blue] (-1,1) -- (2,1) arc (90:270:0.5 and 1) -- (-1, -1) arc (-90:90:0.5 and 1);
    \fill[opacity = 0.2, blue] (-1,1) -- (2,1) arc (90:-90:0.5 and 1) --  (-1, -1) arc (270:90:0.5 and 1);
    \draw[thick, black, ->-] (-1,1) arc (90:270:0.5 and 1);
    \draw[thick, black, dotted] (-1,1) arc (90:-90:0.5 and 1);
    \draw[ black, very thin ] ( 2,1) arc (90:70:0.5 and 1) coordinate (S1) -- ++ (3.67,0) ;
    \draw[ black, very thin, densely dotted] ( 6,1) arc (90:70:0.5 and 1) -- ++ (-0.33,0);
    \draw[ black, very thin, densely dotted] ( 6,-1) arc (-90:60:0.5 and 1)  -- ++ (-0.5,0);
    \draw[ black, very thin] ( 2,-1) arc (-90:10:0.5 and 1) coordinate (a1);%
    \draw[ black, very thin, densely dotted] (a1) arc (10:20:0.5 and 1) coordinate (a2);%
    \draw[ black, very thin] (a2) arc (20:60:0.5 and 1) -- ++ (3.5,0);
    \fill[opacity = 0.2, blue] (S1) arc (70:60:0.5 and 1) -- +(4,0) arc (60:70:0.5 and 1) -- +(-4,0);
    \draw[ black , very thin, name path = line 4] ( 2,1) arc (90:160:0.5 and 1) coordinate (S2) -- +(4,0) arc (160:90:0.5 and 1);
    \draw[ black , very thin, name path = line 4] ( 2,-1) arc (270:170:0.5 and 1) -- +(4,0) arc (170:270:0.5 and 1);
    \fill[opacity = 0.2, blue] (S2) arc (160:170:0.5 and 1) -- +(4,0) arc (170:160:0.5 and 1) -- +(-4,0);
    \fill[opacity = 0.2,blue] (6,1) -- (9,1) arc (90:-90:0.5 and 1) -- (6, -1) arc (-90:90:0.5 and 1);
    \fill[opacity = 0.2, blue] (6,1) -- (9,1) arc (90:270:0.5 and 1) -- (6, -1) arc (270:90:0.5 and 1);
    \draw[black, thick, ->- ] (9,1) arc (90:-90:0.5 and 1);
    \draw[thick, black] (9,1) arc (90:270:0.5 and 1);
\end{scope}}\caption{The facet $A$ in $F_0$.}
\label{fig:counterexample1}
\end{figure}

Now instead apply step~(\ref{item:seamV}) to the matching 
$\{e_2,e_4\}$, canceling vertices $(v_1,v_4)$ and $(v_2,v_3)$ in pairs, and denote the resulting pre-foam $F_1$, as shown in Figure~\ref{fig:counterexample2}. 

For step~(\ref{item:monodromy}), we see that the monodromy along each singular circle is trivial. We should next apply neck-cutting at the three circles near each singular circle. We can choose the order in which the neck-cutting is done. 

Note that $F_1$ has a $\Z/2$-symmetry $\tau$, which in the portion shown is given by reflecting about a vertical axis through the center. We'll be cutting along pairs of circles that are symmetric under $\tau$, each time resulting in nine possible terms that differ by numbers of dots. Six of these terms will come in $\tau$-symmetric pairs. Both terms of each pair will evaluate to the same element of the ground field $\kk$ and will always add up to $0$. 

Hence, each symmetric cutting along a pair of $\tau$-opposite circles only contributes three terms to the sum (some of which may be zero). Also, if a pre-foam has a facet with three or more dots, it evaluates to zero by definition of $J^\flat$. 

If we neck-cut along symmetric pairs in the order given in Figure~\ref{fig:counterexample2}, each time we sum reduces to exactly one non-trivial term. 

\begin{figure}[ht]
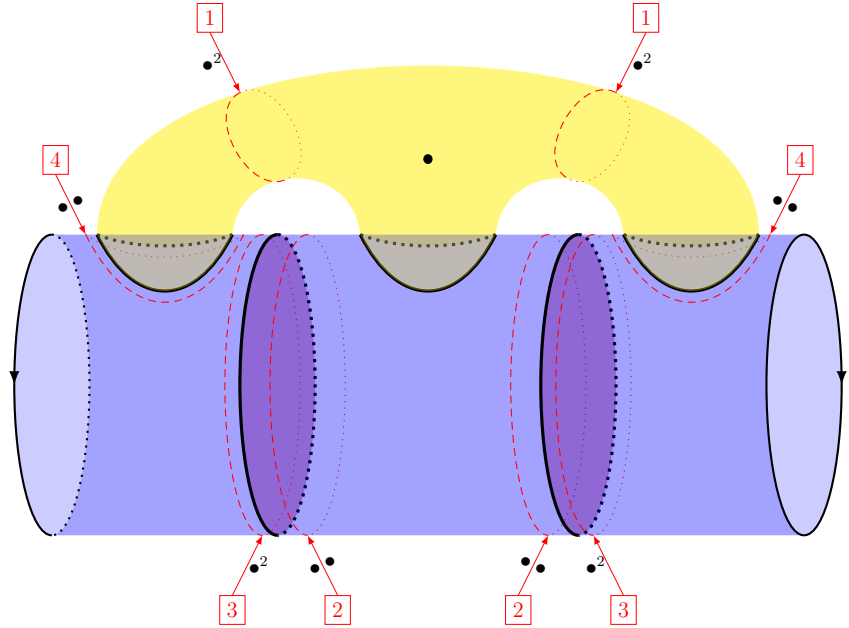

\tikz{\begin{scope}[yscale =2, very thick, label distance = -0.1cm]
\fill[yellow, opacity=0.3 ] (8.4,1) .. controls +(0, 1.5) and +(0, 1.5) .. 
   (-0.4,1) coordinate[pos=0.3] (RT) coordinate[pos=0.7] (LT) .. controls + (0.5, -0.1) and + (-0.5, -0.1) .. (1.4,1) 
   .. controls +(0.1, 0.5) and +(-0.1, 0.5) ..
   (3.1,1) coordinate[pos=0.45] (LB)  .. controls + (0.5, -0.1) and + (-0.5, -0.1) .. (4.9,1) 
   .. controls +(0.1, 0.5) and +(-0.1, 0.5) ..
   (6.6,1) coordinate[pos=0.55] (RB) .. controls + (0.5, -0.1) and + (-0.5, -0.1) .. (8.4,1); 
  \draw[dotted] (-0.4,1) .. controls + (0.5, -0.1) and + (-0.5, -0.1) .. (1.4,1);
  \draw[dotted] (3.1,1) .. controls + (0.5, -0.1) and + (-0.5, -0.1) .. (4.9,1); 
  \draw[dotted] (6.6,1) .. controls + (0.5, -0.1) and + (-0.5, -0.1) .. (8.4,1); 
  \draw[red, thin, dotted] (-0.55,1) .. controls + (0.6, -0.2) and + (-0.6, -0.2) .. (1.55,1);
  \draw[red, thin, dotted] ( 6.45,1) .. controls + (0.6, -0.2) and + (-0.6, -0.2) .. (8.55,1); 
  \draw[thin, dotted, red ] (LT)  .. controls +(0.5, 0.13) and +(0.5, 0.13) .. (LB);
  \draw[thin, dotted, red ] (RT)  .. controls +(0.5,-0.13) and +(0.5,-0.13) .. (RB);
    \fill[opacity = 0.2, blue] (-1,1) -- (2,1) arc (90:-90:0.5 and 1) -- (-1, -1) arc (-90:90:0.5 and 1);
    \fill[opacity = 0.2, blue] (-1,1) -- (2,1) arc (90:-90:0.5 and 1) --  (-1, -1) arc (270:90:0.5 and 1);
    \fill[red, opacity = 0.2] (2,0) circle (0.5 and 1);
    \draw[thick, black, ->-] (-1,1) arc (90:270:0.5 and 1);
    \draw[thick, black, dotted] (-1,1) arc (90:-90:0.5 and 1);
    \draw[red, thin, dotted] (1.8,1) arc (90:-90:0.5 and 1);
    \draw[red, thin, dotted] (2.4,1) arc (90:-90:0.5 and 1);
    \draw[red, thin, dotted] (5.6,1) arc (90:-90:0.5 and 1);
    \draw[red, thin, dotted] (6.2,1) arc (90:-90:0.5 and 1);
    \draw[red, very thin, <-] (RT) -- +(0.4, 0.4) node[sloped, above,draw, scale = 0.7] {1} coordinate[pos=0.5] (S1R);
    \draw[red, very thin, <-] (LT) -- +(-0.4, 0.4) node[above,draw, scale = 0.7] {1} coordinate[pos=0.5] (S1L);
    \draw[red, very thin, <-] (5.6,-1) -- +(-0.4, -0.4) node[below,draw, scale = 0.7] {2} coordinate[pos=0.5] (S2R);
    \draw[red, very thin, <-] (2.4,-1) -- +(0.4, -0.4) node[below,draw, scale = 0.7] {2} coordinate[pos=0.5] (S2L);
    \draw[red, very thin, <-] (6.2,-1) -- +(0.4, -0.4) node[below,draw, scale = 0.7] {3} coordinate[pos=0.5] (S3R);
    \draw[red, very thin, <-] (1.8,-1) -- +(-0.4, -0.4) node[below,draw, scale = 0.7] {3} coordinate[pos=0.5] (S3L);
    \draw[red, very thin, <-] (8.55, 1) -- +(0.4,  0.4) node[above,draw, scale = 0.7] {4} coordinate[pos=0.5] (S4R);
    \draw[red, very thin, <-] (-0.55, 1) -- +(-0.4,  0.4) node[above,draw, scale = 0.7] {4} coordinate[pos=0.5] (S4L);

    \fill[opacity = 0.2, blue] (2,1) -- (6,1) arc (90:-90:0.5 and 1) -- (2, -1) arc (-90:90:0.5 and 1);
    \draw[ black, dotted, name path= line 3] ( 2,1) arc (90:-90:0.5 and 1);
    \fill[opacity = 0.2, blue] (2,1) -- (6,1) arc (90:-90:0.5 and 1) -- (2, -1) arc (270:90:0.5 and 1);
    \draw[ black , name path = line 4] ( 2,1) arc (90:270:0.5 and 1) coordinate[pos= 0.25] (ep1);
    \fill[red, opacity = 0.2] (6,0) circle (0.5 and 1);
    \fill[opacity = 0.2,blue] (6,1) -- (9,1) arc (90:-90:0.5 and 1) -- (6, -1) arc (-90:90:0.5 and 1);
    \draw[ black, dotted, name path = line 5] (6,1) arc (90:-90:0.5 and 1);
    \fill[opacity = 0.2, blue] (6,1) -- (9,1) arc (90:270:0.5 and 1) -- (6, -1) arc (270:90:0.5 and 1);
    \draw[ black, name path = line 6] (6,1) arc (90:270:0.5 and 1) coordinate[pos= 0.25] (ep2);
    \draw[black, thick, ->- ] (9,1) arc (90:-90:0.5 and 1);
    \draw[thick, black] (9,1) arc (90:270:0.5 and 1);
  \draw (-0.4,1) .. controls + (0.5, -0.5) and + (-0.5, -0.5) .. (1.4,1);
  \draw ( 3.1,1) .. controls + (0.5, -0.5) and + (-0.5, -0.5) .. (4.9,1); 
  \draw ( 6.6,1) .. controls + (0.5, -0.5) and + (-0.5, -0.5) .. (8.4,1); 

  \draw[red, thin, densely dashed] (-0.55,1) .. controls + (0.6, -0.6) and + (-0.6, -0.6) .. (1.55,1);
  \draw[red, thin, densely dashed] ( 6.45,1) .. controls + (0.6, -0.6) and + (-0.6, -0.6) .. (8.55,1); 

    \fill[yellow, opacity=0.3 ] (8.4,1) .. controls +(0, 1.5) and +(0, 1.5) .. 
  (-0.4,1) .. controls + (0.5, -0.5) and + (-0.5, -0.5) .. (1.4,1)
   .. controls +(0.1, 0.5) and +(-0.1, 0.5) ..
  ( 3.1,1) .. controls + (0.5, -0.5) and + (-0.5, -0.5) .. (4.9,1) 
   .. controls +(0.1, 0.5) and +(-0.1, 0.5) ..
  ( 6.6,1) .. controls + (0.5, -0.5) and + (-0.5, -0.5) .. (8.4,1); 
  \draw[densely dashed, thin, red] (LT)  .. controls +(-0.5,-0.13) and +(-0.5,-0.13) .. (LB);
  \draw[densely dashed,  thin, red] (RT)  .. controls +(-0.5, 0.13) and +(-0.5, 0.13) .. (RB);
    \draw[red, thin, densely dashed] (1.8,1) arc (90:270:0.5 and 1);
    \draw[red, thin, densely dashed] (2.4,1) arc (90:270:0.5 and 1);
    \draw[red, thin, densely dashed] (5.6,1) arc (90:270:0.5 and 1);
    \draw[red, thin, densely dashed] (6.2,1) arc (90:270:0.5 and 1);
    \fill (4,1.5) circle (0.6mm and 0.3mm);
    \node[scale = 0.7] at (4,1.5) {$\bullet$};
  \node[scale = 0.7] at ($(S1R) + (0.15,0)$) {$\bullet^2$};
  \node[scale = 0.7] at ($(S1L) + (-0.17,0)$) {$\bullet^2$};

  \node[scale = 0.7] at ($(S2R) + (0.1,-0.025)$) {$\bullet$};
  \node[scale = 0.7] at ($(S2L) + (0.1,0.025)$) {$\bullet$};
  \node[scale = 0.7] at ($(S2R) + (-0.1,0.025)$) {$\bullet$};
  \node[scale = 0.7] at ($(S2L) + (-0.1,-0.025)$) {$\bullet$};

  \node[scale = 0.7] at ($(S3R) + (-0.17,0)$) {$\bullet^2$};
  \node[scale = 0.7] at ($(S3L) + (0.15,0)$) {$\bullet^2$};

  \node[scale = 0.7] at ($(S4R) + (0.1,-0.025)$) {$\bullet$};
  \node[scale = 0.7] at ($(S4L) + (0.1,0.025)$) {$\bullet$};
  \node[scale = 0.7] at ($(S4R) + (-0.1,0.025)$) {$\bullet$};
  \node[scale = 0.7] at ($(S4L) + (-0.1,-0.025)$) {$\bullet$};

\end{scope}}\caption{The pre-foam $F_1$. The red dashed lines are where the neck-cutting are performed, the number indicates the order, the $\bullet$ indicates the only terms which are non-zero when applying the neck-cutting relations.}
\label{fig:counterexample2}
\end{figure}

We start by cutting along the circle pair labeled $1$. The three non-canceling terms differ by distribution of dots, with either $0$, $1$ or $2$ dots added on each side to the cut central disk in the top of the picture. This results in adding $0$, $2$ or $4$ dots to a disk that already carries a dot. Unless no dots are added, the resulting pre-foam has a facet with at least 3 dots and evaluates to $0$. Consequently, only the term where two dots are added to each of the outer yellow disks survives in the sum. This situation is depicted by placing $\bullet^2$ on the corresponding side of the cut circle. 

We continue by cutting along the pair of circles labeled $2$. After the cuts there is a theta-foam in the middle, with one and zero dots, respectively, on the top and middle facets, requiring exactly two dots on the bottom facet for a nonzero evaluation. Hence, the only possible distribution is to place one dot on each side of each circle labeled $2$ upon the cuts. 

Next, performing cuts along circles labeled $3$ splits off two theta-foams, in a symmetric fashion. These theta-foams already have facets with $0$ and $1$ dots, requiring two dots to appear on the new facets 
after the cuts. This determines the unique distribution of dots for the third pair of cuts as well. 

The same argument shows that for cuts number $4$ the only distribution is to place one dot on each side of the cuts. These cuts will produce two theta-foams, each evaluating to $1$ (with these dot distributions) and a two-dotted sphere, evaluating to $1$ as well. 

Notice that the assumption that the order of cuts is inessential is built into the definition and the algorithm. We also bypass cutting along circles that already bounds disks after the previous cuts, since consistency for such cuts is an easy exercise going back to~\cite{SL3}. 

The computation results in  $J^{\flat}(F_1)=1$, 
which differs from $J^{\flat}(F_0)=0$. This shows that for pre-foam $F$ the value produced by the above algorithm depends on the choices made in step~(\ref{item:seamV}). Consequently, Conjecture 8.9 in~\cite{KM1} needs to be augmented for the evaluation to be well-defined. Theorem~\ref{thm:JflatFoams}  implies that one possible modification is to restrict to pre-foams embeddable in $\R^3$.

\section{Homology of webs} 
\label{sec:homology}
\subsection{Webs and their homology}\ 

A \emph{closed web}, or just a \emph{web}, is a trivalent oriented graph $\Gamma$, 
possibly with vertexless loops, embedded 
in $\R^2$ piecewise-linearly. 

We say that an oriented plane $T\cong\R^2$ in $\R^3$ intersects a (closed) foam 
$F$ \emph{generically} if $F\cap T$ is a web $\Gamma$ in $T$, no dots of $F$ are on $T$ and for a tubular neighborhood $N$ of $T$, $(N\cap F, N)$ is PL-homeomorphic to $(\Gamma \times ( -\epsilon, \epsilon ) , \R^2 \times (-\epsilon, \epsilon ) )$. 
Define a foam with boundary $V$ 
as the intersection of a closed foam $F$ and $T\times [0,1]\subset \R^3$ such 
that $T\times \{0\}$ and $T\times \{1\}$ intersect $F$ generically. 
We view foam $V$ with boundary as a cobordism between webs $\partial_i V \define 
V\cap T\times \{i\}$ for $i=0,1$ and assume the  standard embedding of 
$\R^2\times [0,1]\cong T\times[0,1]$ into $\R^3$. Sometimes we will call a foam with boundary simply a foam. 
Two foams are isomorphic if they are isotopic in $\R^2\times [0,1]$ through an 
isotopy which fixes all boundary points. 

For example, a closed foam is a foam with the empty boundary and gives a cobordism 
from the empty web to itself. 

The notions of admissible and pre-admissible coloring extend without difficulty to foams with boundary. A pre-admissible coloring of a foam $F$ induces a Tait coloring of its boundary. Note that since foams are properly embedded in $\R^2\times [0,1]$, any pre-admissible coloring of a foam with boundary is admissible. 

If $U$ and $V$ are two foams such that the webs $\partial_0 U, \partial_1 V$ are identical, define 
the composition $UV$ in the obvious way, by concatenating $U$ and $V$ along their common boundary (and rescaling). In this way we obtain a category $\catF$ 
with webs as objects and isomorphism classes of foams with boundary as morphisms. 

A foam $U$ is a morphism from $\partial_0U$ to $\partial_1 U$. If $\partial_0U$ 
is the empty foam, we say that $U$ is a foam or cobordism into $\partial_1U$. 
If $\partial_1U$ is the empty foam, we say that $U$ is a foam  out of 
$\partial_0U$. 

The category $\catF$ has an anti-involution $\omega$, which acts as the identity on objects and  on morphisms is given by 
reflecting a foam about $\R^2\times \{ \frac{1}{2}\}$. The category $\catF$ also has an involution given by reflecting a foam about 
$\ell \times [0,1]$, where $\ell$ is a line in $\R^2$. 

For a foam $U$ let $d(U)$ denote the set of dots on $U$, so that $|d(U)|$ is the total number of dots on $U$. Likewise, $|v(U)|$ is the number of 
seam vertices of $U$. 

Define the degree of a foam $U: \Gamma_0 \to \Gamma_1$ by 
\begin{align}
\label{eq:degree-with-bdy}
\deg(U) =  2 \ |d(U)| - 2 \ \chi(U) - \chi(s(U))
\end{align}
In particular, for any web $\Gamma$, foam  $\Id_\Gamma= \Gamma\times [0,1]$ has degree $0$.
Remark~\ref{rmk:degree} remains true in the context of foams with boundary.

\begin{prop} \label{prop:degree} For composable foams $U$ and $V$, 
\[\deg(UV) = \deg(U) + \deg(V).\] 
\end{prop}
\begin{proof} 
Consider foams $U:\Gamma_1 \to \Gamma_2$ and $V: \Gamma_0 \to \Gamma_1$. For a finite CW-complex $C$ obtained by gluing two CW-complexes $C_1$ and $C_2$ along a common CW-subcomplex $C_3$  one has $\chi(C)= \chi(C_1) + \chi(C_2) - \chi(C_3)$. Since $\Gamma_1$ is a trivalent graph, we have: $2\chi(\Gamma_1) = -|2v(\Gamma_1)|$. This gives:
\begin{align*}
\deg(UV) = & 2 |d(UV)| - 2\chi(UV) - \chi(s(UV)) \\
	     = & 2 |d(UV)| - 2\chi(U) -2 \chi(V) + 2 \chi(\Gamma_1) - \chi(s(U)) - \chi(s(V)) + |v(\Gamma_1)|\\
         = 	& \deg(U) + \deg(V) + 2\chi(\Gamma_1) + |v(\Gamma_1)|\\
         = &\deg(U) + \deg(V).         
\end{align*}
  
  \end{proof}
The proposition says that the degree of foams is well-behaved under composition.  The antiinvolution $\omega$ preserves the degree,  
$\deg(\omega(U))= \deg(U)$. 

We next define \emph{homology} or \emph{state space} $\brak{\Gamma}$ of a  web $\Gamma$ 
as a graded $R$-module spanned by all foams into $\Gamma$, modulo the evaluation relation. This definition, called the \emph{universal construction}, goes back to \cite{BHMV} and was used in~\cite{SL3} in the $sl(3)$ foam framework.

\begin{defn} The state space $\brak{\Gamma}$ is an $R$-module generated by symbols $\brak{U}$ 
for all foams $U$ from the empty foam $\emptyset$ to $\Gamma$. A relation 
$\sum_i a_i \brak{U_i}=0$ for $a_i\in R$ and $U_i \in \Hom_{\catF}(\emptyset,\Gamma)$ holds in $\brak{\Gamma}$ if and only if 
\[ \sum_i a_i \brak{V U_i} = 0 \]
for any foam $V$ from $\Gamma$ to the empty web. Here $\brak{VU_i}\in R$ is the 
evaluation of the closed foam $VU_i$. 
\end{defn}

It follows from the definition that the homology of the empty web is naturally isomorphic to the free $R$-module $R$, with the generator given by the empty foam. 

\begin{defn} Let $\Fo(\Gamma)$ be the free  $R$-module $\Fo(\Gamma)$ with the basis given by all foams into $\Gamma$, including foams decorated with dots and those which have connected components disjoint from $\Gamma$. 
\end{defn}

Assigning to a  foam its degree extends to a grading on $\Fo(\Gamma)$ and $\brak{\Gamma}$, turning them into graded $R$-modules over the graded ring $R$. 
$\Fo(\Gamma)$ is a free graded $R$-module. 

There is a canonical surjective graded $R$-module 
homomorphism 
\begin{align}\label{eq:h_gamma}
h_{\Gamma} \ : \ \Fo(\Gamma) \lra \brak{\Gamma}
\end{align}
induced by sending a foam $U$ into $\Gamma$ to 
$\brak{U}\in \brak{\Gamma}$. In particular,  $\brak{\Gamma}$ is isomorphic to a quotient of 
the free $R$-module $\Fo(\Gamma)$.

Given two foams $U$ and $V$ into $\Gamma$, consider the closed foam $\omega(V)U$ and evaluate it to $\brak{\omega(V)U} \in R$, where $\omega$ is the 
anti-involution, defined earlier, that reflects a foam about a horizontal plane. Extending bilinearly, 
one gets a map 
\[ \Fo(\Gamma) \times \Fo(\Gamma) \longrightarrow R\] 
that factors through the tensor product over $R$, 
\[ \Fo(\Gamma) \times \Fo(\Gamma) \longrightarrow \Fo(\Gamma)\otimes_R \Fo(\Gamma) \longrightarrow R, \]
equipping $\Fo(\Gamma)$ with a symmetric $R$-valued bilinear form $(,)$. 
This bilinear form is degree-preserving, relative to the above gradings on $\Fo(\Gamma)$ and $R$. 
The kernel $\mathrm{ker}((,))$ of this bilinear form is a graded 
$R$-submodule of $\Fo(\Gamma)$. 

\begin{prop} Homomorphism $h_{\Gamma}$ identifies 
$\brak{\Gamma}$ with the quotient of $\Fo(\Gamma)$ 
by the kernel $\mathrm{ker}((,))$ of the bilinear form:
\[ \brak{\Gamma} \cong \Fo(\Gamma)/\mathrm{ker}((,)).\]
\end{prop}

The proposition is immediate from the definitions.
The form descends to a symmetric $R$-bilinear degree-preserving form
\[ (,) \ : \ \brak{\Gamma} \otimes_R \brak{\Gamma} \longrightarrow R \]
on $\brak{\Gamma}$ with values in $R$.  
This form is non-degenerate, that is, for any $a\in \brak{\Gamma},$ $a\not=0$ there is $b$ such that 
$(a,b)\not= 0$.

Bilinear form $(,): \lG \otimes_R \lG \lra R$ has degree $0$, when 
viewed as a map between graded $R$-modules, due to 
Proposition~\ref{prop:degree}. 

\begin{rmk}
We don't know whether the form is a perfect pairing for any $\Gamma$, that is, whether $\brak{\Gamma}$ is always a free graded $R$-module of finite rank with a homogeneous basis $b_1, \dots, b_m$ and a dual basis $b_1^{\ast}, \dots, b_m^{\ast}$ such that $(b_i,b_j^{\ast}) = \delta_{i,j}$. 
\end{rmk}

Denote the assignment of $\brak{\Gamma}$ to $
\Gamma$ by $\brak{\bullet}$. 
We can promote $\brak{\bullet}$ to a functor from the category $\catF$ of 
foams to the category of graded $R$-modules and homogeneous module homomorphisms. It assigns  
a graded $R$-module $\brak{\Gamma}$ to a web $\Gamma$ and a homogeneous 
$R$-module map $\brak{\partial_0 U}\lra \brak{\partial_1 U}$ of degree 
$\deg{U}$ to a foam $U$. This map can be first 
defined on the level of free modules, as the 
map 
$\Fo(\partial_0 U) \lra \Fo(\partial_1 U)$ taking 
a foam $V$ into $\Gamma$ (a basis element of 
$\Fo(\partial_0 U)$) to the composition 
$UV$, which is an element of the basis of $\Fo(\partial_1 U)$ and then extending by 
linearity. This homomorphism of free $R$-modules descends to the quotient map 
\[ \brak{U} \ : \ \brak{\partial_0 U} \lra \brak{\partial_1 U} .\] 

Relative to the bilinear form on $\brak{\Gamma}$, for various $\Gamma$, the $R$-linear map $\brak{U}$ is adjoint to the 
map $\brak{\omega(U)}: \brak{\partial_1 U} \lra \brak{\partial_0 U}$, since 
$ (\omega(U)W,V) = (W,UV) $
for any foam $V$ into $\partial_0 U$ and any foam $W$ into $\partial_1 U$.

\subsection{Boundary colorings and finitely-generated property}
\label{sec:fin-gen}\ 

We now extend the formula for evaluation of closed foams to foams $U$ with boundary, at least when the boundary is on one side of the foam, and use this extension to show that the state space $\brak{\Gamma}$ is a finitely-generated $R$-module. 
We can fix a Tait coloring $t$ of the boundary and form a suitable sum over all extensions of the coloring $t$  to a pre-admissible coloring $c$ of the foam. 

Just like in the closed case, by a pre-admissible coloring of a foam $U$ with boundary $\Gamma$ we mean an assignment of colors $\{1,2,3\}$ to components of $U$ such that along each seam edge the colors are distinct. A pre-admissible coloring of $U$ induces a Tait coloring of its boundary $\Gamma$. 

By admissible coloring of a foam $U$ with boundary, we mean an admissible coloring such that all bicolored surfaces are orientable. 

Note that, for any pre-admissible coloring $c$ of a foam $U\in \R^2\times[0,1]$ with boundary (even when both boundaries $\partial_0 U$, $\partial_1 U$ are non-empty), all surfaces $F_{ij}(c)$ are orientable, although some may have boundary. That's because we can compose $U$ with its reflection, forming the foam $\omega(U)U$ with identical top and bottom boundary $\partial_0 U$, and then closing it up into a foam $\widetilde{U}$ without boundary. The coloring $c$ extends to a pre-admissible coloring of $\widetilde{U}$, which is then necessarily admissible, since $\widetilde{U}$ is closed. Consequently, $c$ is admissible as well. 

Thus, for foams with boundary there is  no difference between pre-admissible and admissible colorings. Denote by $\adm(U)$ the set of admissible colorings of a foam $U$ with boundary. 

The surfaces $F_{ij}(c)$ are no longer always closed, although still orientable, and their Euler characteristic may be odd. In the extension of the formula, we would need to form square roots $(X_i+X_j)^{\frac{1}{2}}$ and their inverses. In characteristic two 
\[\sqrt{X_i + X_j} = \sqrt{X_i} + \sqrt{X_j}\] 
and 
\[\frac{1}{\sqrt{X_i+X_j}} = \frac{\sqrt{X_i}+\sqrt{X_j}}{X_i + X_j},\]
so it's enough to introduce square roots of generators $X_1, X_2, X_3$.

Recall that so far we have been using the chain of rings $R\subset R'\subset R''$, where 
\begin{align*}
R & =  \kk[E_1,E_2,E_3], \\
R' & =  \kk[X_1,X_2,X_3], \\
R'' & =  R'[(X_1+X_2)^{-1},(X_1+X_3)^{-1},(X_2+X_3)^{-1}].
\end{align*}

Form the ring $\widetilde{R}'$ by extending $R'$ by adding square roots of $X_1,X_2, X_3$, 
\[ \widetilde{R}'  = \kk[X_1^{ \frac{1}{2}}, 
X_2^{\frac{1}{2}},X_3^{\frac{1}{2}}].\]
Similarly, let 
\[\widetilde{R}''  =  \kk[X_1^{ \frac{1}{2}}, 
X_2^{\frac{1}{2}},X_3^{\frac{1}{2}},
(X_1+X_2)^{-1}, (X_1+X_3)^{-1}, (X_2+X_3)^{-1}].\] 
The ring $\widetilde{R}''$ is a free graded $R''$-module 
with a basis $\{X_1^{\epsilon_1}X_2^{\epsilon_2}X_3^{\epsilon_3}\}$, where $\epsilon_i\in \{0,\frac{1}{2}\}$, $i=1,2,3$. Indeed, this set generates $\widetilde{R}''$ as an $R''$-module and it is $R''$-linearly independent. To see this,  suppose that a $R''$-linear combination of these eight elements is zero. Multiplying by a power of $(X_1+X_2)(X_1+ X_3)(X_2+X_3)$, we can suppose that it is an $R'$-linear combination of elements of $\widetilde{R}'$. The result follows since $R''$ and $R'$ are domains and the above set is a basis of the free $R'$-module $\widetilde{R}'$. 
    
The diagram below depics inclusions of these five rings. 
\[
\begin{CD}
 &   &  \ \widetilde{R}' \ & \ \subset \ & \ \widetilde{R}'' \\
 &   &     \cup   &    &  \cup  \\
 R \ & \ \subset \ & \ R' \ & \ \subset \ & \ R'' 
\end{CD}
\]    
The ring $\widetilde{R}''$ is naturally isomorphic to the ring 
\[ \kk[Y_1,Y_2,Y_3,(Y_1+Y_2)^{-1},(Y_1+Y_3)^{-1},(Y_2+Y_3)^{-1}]\]
via the map that sends $Y_i$ to $X_i^{\frac{1}{2}}$ and 
$(Y_i+Y_j)^{-1}$ to $\frac{\sqrt{X_i}+\sqrt{X_j}}{X_i + X_j}$. 

Given an admissible coloring $c$ of a foam $U$ with 
boundary, we can form the monomial $P(U,c)$ as before, as product of $X_{c(f)}^{d(f)}$ over all facets $f$ of $U$. Likewise, define
\begin{align}\label{eq:Q2}
Q(U,c) = \prod_{1\le i < j \le 3} (X_i + X_j)^{\frac{\chi(F_{ij}(c))}{2}}  \ \in \widetilde{R}'' 
\end{align}
as an element of the bigger ring $\widetilde{R}''$ (for closed forms the product lies in the smaller ring $R''$). The ratio 
\[ \brak{U,c} = \frac{P(U,c)}{Q(U,c)}\]
is an element of $\widetilde{R}''$.

We write $c\supset t$ to indicate that a pre-admissible coloring $c$ of $U$ extends a Tait coloring $t$ of the web $\partial U = \partial_0 U \cup \partial_1 U$. 
Define 
\[\brak{U,t}_{\partial} =  \sum_{c \supset t}\brak{U, c}\in \widetilde{R}''.\]
This formula generalizes (\ref{eq:eval}) to 
foams $U$ with boundary. If $t$ does not extend to an admissible coloring of $U$ then $\brak{U,t}_{\partial}=0$. 

We now specialize to the case when $U$ has boundary only at the top, that is $U$ is a foam 
into a web $\Gamma= \partial_1 U$, with $\partial_0 U = \emptyset$. Fix a web $\Gamma$ and choose a Tait coloring $t$ of $\Gamma$. Consider any foam $U$ into $\Gamma$.  

The subgraph of $\Gamma$ which consists of all the vertices of $\Gamma$ and edges of $\Gamma$ which are colored $i$ or $j$ by $t$ is a collection of cycles, called the $ij$-cycles of $t$.

The ring $\widetilde{R}''$ contains $R'$ as a subring, and, when viewed as an $R'$-module, 
contains a collection of $R'$-submodules generated by elements 

\begin{align}\label{eq:u}
u(n_1,n_2,n_3) = (X_1+X_2)^{-\frac{n_1}{2}}(X_1+X_3)^{-\frac{n_2}{2}}(X_2 + X_3)^{-\frac{n_3}{2}} 
\end{align}
for any $n_1,n_2,n_3\in \Z$. 

\begin{prop} For any $\Gamma$, $t$ and $U$ as above, $\brak{U,t}_{\partial}\in R'u(m_{12},m_{13},m_{23})$, where $m_{ij}$ is the number of $ij$-cycles in $t$. 
\end{prop} 

\begin{proof}
The proof is very similar to that of Theorem~\ref{thm:evl-sym-pol}. Indeed, if $\Gamma$ is empty and $U$ is then a closed foam, the proposition simply says that $\brak{U}$ is a polynomial. 

By definition, 
\[
\brak{U,t}_\partial= \sum_{c\supset t} \brak{U,c}
= \sum_{c\supset t} \frac{P(U,c)}{
(X_1 + X_2)^{\frac{\chi(U_{12},c)}{2}}
(X_1 + X_3)^{\frac{\chi(U_{13},c)}{2}}
(X_2 + X_3)^{\frac{\chi(U_{23},c)}{2}}
}
.
\]

For a pair of colors $(i,j)$ and $c\supset t$ form 
the surface $U_{ij}(c)$. Its boundary is the union of 
edges of $\Gamma$ colored by $i$ or $j$ by $t$. Denote by $U_{ij}^\partial(c)$, respectively $U_{ij}^\mathrm{o}(c)$, the union of all connected components of $U_{ij}(c)$ with non-empty, respectively empty, boundary. We have $\chi(U_{ij}^\partial)\leq m_{ij}$, since the Euler characteristic of a disk is $1$, and any other connected compact surface with boundary has Euler characteristic $0$ or less. Moreover, $\chi(U_{ij}^\partial(c))$ and $m_{ij}$
have the same parity. 

We have:
\begin{align*}
&  \brak{U,t}_\partial
 = \sum_{c\supset t} \brak{U,c} \\
&   =   \sum_{c\supset t} 
\frac{P(U,c)}{
(X_1 + X_2)^{\frac{\chi(U_{12},c)}{2}}
(X_1 + X_3)^{\frac{\chi(U_{13},c)}{2}}
(X_2 + X_3)^{\frac{\chi(U_{23},c)}{2}}
}   
\\
&  = u \sum_{c\supset t}
\frac{P(U,c) (X_1 + X_2)^{\frac{m_{12}-\chi(U^\partial_{12}(c))}{2}}
(X_1 + X_3)^{\frac{m_{13}-\chi(U^\partial_{13}(c))}{2}}
(X_2 + X_3)^{\frac{m_{23}-\chi(U^\partial_{23}(c))}{2}}
}{
(X_1 + X_2)^{\frac{\chi(U^{\mathrm{o}}_{12}(c))}{2}}
(X_1 + X_3)^{\frac{\chi(U^{\mathrm{o}}_{13}(c))}{2}}
(X_2 + X_3)^{\frac{\chi(U^{\mathrm{o}}_{23}(c))}{2}}
}
,
\end{align*}
where $u = u(m_{12}, m_{13}, m_{23})$ is given by formula (\ref{eq:u}). Note that 
each exponent in the numerator is non-negative, 
since $m_{ij}\ge \chi(U_{ij}^{\partial}(c)).$ 

Let $r$ be the number of connected components of $U_{ij}^{\mathrm{o}}(c)$ for a given $c\supset t$. We apply Kempe moves along these components and combine together $2^r$ terms in the above sum for 
the $2^r$ $ij$-Kempe-related colorings to 
pull out $(X_i+X_j)^r$ and cancel  
potentially positive exponent $(X_i+X_j)^\frac{\chi(U^{\mathrm{o}}_{ij}(c))}{2}$ 
in the denominator. 

Consequently, $\brak{U,t}_{\partial}$ belongs 
to the $R''_{ij}$-submodule of $\widetilde{R}''$ generated 
by $u$, where, recall,  
\[R''_{ij} = R' \left[\frac{1}{X_i+X_k},\frac{1}{X_j+X_k}\right] \]
and $\{ i,j,k\} = \{ 1, 2, 3\}$. 
Since the triple intersection of the rings $R''_{12}$, $R''_{13}$, and $R''_{23}$ is $R'$, the sum $\brak{U,t}_{\partial}$ belongs to 
$uR'$. 
\end{proof} 

Let $N(\Gamma)$ be the maximal number of 12-colored cycles in any Tait coloring of $\Gamma$ and 
$$u(\Gamma) = ((X_1+X_2)^{1/2}(X_1+X_3)^{1/2}(X_1+X_3)^{1/2})^{-N(\Gamma)}.$$ 
\begin{corollary} Fix a web $\Gamma$. 
For any foam $U$ into $\Gamma$ and any Tait coloring $t$ of $\Gamma$, the evaluation 
$\brak{U,t}_\partial$ belongs to $\widetilde{R}' u(\Gamma)$, that is, 
the $\widetilde{R}'$-submodule of $\widetilde{R}''$ generated by 
$u(\Gamma)$. In particular, $\brak{U,t}_\partial$, over all $t$, 
belong to a finitely-generated (and free of rank eight) $R'$-submodule 
of $\widetilde{R}''$. The degree of $\brak{U,t}_{\partial}$ is 
bounded below by $-3N(\Gamma)$. 
\end{corollary}

Recall that $\Fo(\Gamma)$ has a basis $\{ [F]\}_F$ given by all possible foams $F$ 
from the empty foam $\emptyset$ to $\Gamma$. 
Degree of $[F]$ is given by formula (\ref{eq:degree-with-bdy}). 

\vspace{0.1in}

For a web $\Gamma$ denote by $\adm(\Gamma)$ the set of Tait colorings of $\Gamma$. Consider the free 
graded $\widetilde{R}''$-module $M(\Gamma)$ of rank $|\adm(\Gamma)|$ with a basis $\{ 1_t\}_{t\in \adm(\Gamma)}$. We place each basis element in degree $0$. 

Assume that $U$ is a foam into $\Gamma$. Let us define   
\[\brak{U}_{\partial} = \sum_{t\in \adm(\Gamma)} \, \sum_{\substack{c\in \adm(U) \\ c \supset t}}\brak{U, \partial c}1_{t} \in M(\Gamma).\]
In this formula, each admissible coloring $c$ of $U$ contributes to the coefficient of $1_t$, where $t$ is the restriction of $c$ to $\Gamma$. 

Consider a symmetric bilinear form $(,)_M$ on $M(\Gamma)$ with values 
in $\widetilde{R}''$ which is orthogonal in the basis of $1_t$'s, so 
that $(1_t, 1_s)_M = \delta_{t,s}$. 

\begin{prop} \label{prop:bilin-equal} For foams $U$ and $U_1$ into $\Gamma$ one has 
\begin{align}\label{eq:bilin-equal}
(\brak{U_1}_{\partial}, \brak{U}_{\partial})_M = 
 ( \brak{U_1}, \brak{U}) = \brak{\omega(U_1)U} \in R.
 \end{align}
\end{prop} 
In particular, the inner product for the braket evaluation takes 
values in the subring $R$ of $\widetilde{R}''$. 

\begin{proof} The evaluation $\brak{\omega(U_1)U}$ is given by summing over all admissible colorings of $\omega(U_1)U$. Each of these colorings restricts to a Tait coloring of $\Gamma$, which is the middle cross-section of $\omega(U_1)U$. Vice versa, a pair of colorings of $U$ and $U_1$ that restrict to the same coloring on their boundaries give rise to an admissible coloring of $\omega(U_1)U$. Each such pair of compatible colorings of $U$ and $U_1$ contributes the same quantity to the LHS and the RHS of the formula in the proposition. 
\end{proof} 
  
Now consider three graded $R$-modules: $\Fo(\Gamma)$, $M(\Gamma)$, 
and $\brak{\Gamma}$. Each of these comes with a symmetric bilinear 
form on it, which is $(,)_M$ for the second module and is given 
by the evaluation $\brak{\omega(U_1)U}$ on generating pairs for the
first and the third modules. The form takes values in $R$ for the 
first and third spaces and values in the bigger ring $\widetilde{R}''$ 
for the second module. The third space is the quotient of the first 
by the kernel of the bilinear form, and the forms on the first and second spaces are related by the formula (\ref{eq:bilin-equal}). 

Consider the $R$-submodule $M_R(\Gamma)$ of $M(\Gamma)$ generated over $R$ by $\brak{U}_{\partial}$ over all foams $U$ into $\Gamma$. Due 
to (\ref{eq:bilin-equal}), the restriction of the bilinear form 
$(,)_M$ to this submodule takes values in $R$ rather than in the 
bigger ring $\widetilde{R}''$. 

$M_R(\Gamma)$ is an $R$-submodule of 
the $\widetilde{R}'$-submodule 
\[\bigoplus_{t\in \adm(\Gamma)}\widetilde{R}' \ u(\Gamma)1_t\]
of $M(\Gamma)$. 

The latter submodule is a 
finitely-generated graded $R$-module (and also a free $R$-module), 
being a finite direct sum of free $\widetilde{R}'$-submodules 
generated by $u(\Gamma)1_t$, over all $t$. 

Finitely-generated property follows by considering the chain of 
subrings $R\subset R'\subset \widetilde{R}'$ and observing that 
$R'$ is a free graded finitely-generated $R$-module (of rank six), 
and $\widetilde{R}'$ is graded finitely-generated $R'$-module 
(in fact, a free rank eight module). 

Since $\Fo(\Gamma)$ is a free graded $R$-module generated by 
foams into $\Gamma$, there is a surjective $R$-module map 
$\Fo(\Gamma) \lra M_R(\Gamma)$ given by sending foam $U$ to $\brak{U}_\partial$, for all $U$. This homomorphism respects the bilinear 
forms, in view of Proposition~\ref{prop:bilin-equal}. Furthermore, 
all the bilinear forms considered respect the grading of our 
modules. 

Consequently, there is a unique homomorphism of graded $R$-modules
$\gamma_{\Gamma}: M_R(\Gamma) \lra \brak{\Gamma}$ that takes $\brak{U}_{\partial}$ 
to $\brak{U}$ for all foams $U$ into $\Gamma$, due to 
$\brak{\Gamma}$ being the quotient of $\Fo(\Gamma)$ by the kernel 
of the bilinear form. This homomorphism is surjective, leading 
at once to the following result. 

\begin{prop} \label{prop:fin-gen} Graded $R$-module $\brak{\Gamma}$ is finitely-generated, for any web $\Gamma$. 
\end{prop}

\begin{proof} The $R$-module $\brak{\Gamma}$ is a quotient of the finitely-generated 
graded $R$-module $M(\Gamma)$. 
\end{proof}
We collect the modules and maps from the proof into the diagram below
\[
\begin{CD}
 &   &  \ M(\Gamma) \ &  &  \\
 &   &     \cup   &    &    \\
 \Fo(\Gamma) & \ \longrightarrow \ & \ M_R(\Gamma) \ & \ \stackrel{\gamma_{\Gamma}}{\longrightarrow} \ &  \lG  . 
\end{CD}
\]    

\vspace{0.1in}

Each element $b$ of $\brak{\Gamma}$ determines an 
$R$-linear map $\Fo(\Gamma)\longrightarrow R$ taking $a$ to $(b,a)$. The form $(,)$ is non-degenerate on $\brak{\Gamma}$ and this assignment is an injective $R$-module homomorphism 
$$\brak{\Gamma}\longrightarrow \Fo(\Gamma)^{\ast} = 
\Hom_R (\Fo(\Gamma),R)$$  
Since $\brak{\Gamma}$ is finitely generated over $R$, choose a finite collection of homogeneous generators $b_1, \dots, b_m$ of this $R$-module, 
giving a surjective $R$-module map $R^m \lra 
\brak{\Gamma}$. Then 
assigning to $a\in \brak{\Gamma}$ the element 
\[((b_1,a),\dots,(b_m,a))^T\in R^m\] 
is an injective $R$-module map
\[ \brak{\Gamma} \ \lra \ R^m.\]
The map is that of graded $R$-modules if we assign to the generator $(b_i,\ast)$ of $R^m$ degree $-\deg(b_i)$, $i=1, \dots, m$. We frame this into a proposition. 

\begin{prop} 
$\brak{\Gamma}$, for any web $\Gamma$, is isomorphic to a submodule of a free graded $R$-module of finite rank.  
\end{prop}

\begin{corollary} Finitely-generated graded $R$-module $\brak{\Gamma}$ has no torsion. 
It's equipped with a symmetric graded $R$-valued bilinear form with the trivial kernel. 
\end{corollary}

\subsection{Direct sum decompositions} 
\label{sub:dir_sum}\

In this subsection we will translate some of the relations satisfied by the local evaluation of foam given in Section~\ref{sec:rel-btwn-ev} into local relation satisfied by the homology. 

\begin{prop}
\label{prop:dec-circle}
If a graph $\Gamma'$ is obtained from a graph $\Gamma$ by adding an innermost circle, then there is a canonical isomorphism 
\[
\brak{\Gamma'} \simeq \brak{\Gamma}\{2\} \oplus \brak{\Gamma} \oplus \brak{\Gamma}\{-2\}
\]
given by maps in Figure~\ref{fig:dec-circle}.
\begin{figure}[ht]
\centering
\begin{tikzpicture}[scale=0.8]
\input{\imagesfolder/km_dec-circle}
\end{tikzpicture}
\caption{}
\label{fig:dec-circle}
\end{figure}
\end{prop}
\begin{proof}
This follows directly from Proposition~\ref{prop:neckcutting} and Corollary~\ref{cor:sphere-eval}.  
\end{proof}

\begin{prop}
\label{prop:dec-digon}
If a graph $\Gamma'$ is obtained from a graph $\Gamma$ by adding a digon region, then there is a canonical isomorphism 
\[
\brak{\Gamma'} \simeq \brak{\Gamma}\{1\} \oplus \brak{\Gamma}\{-1\}
\]
given by Figure~\ref{fig:dec-digon}.
\begin{figure}[ht]
\centering
\begin{tikzpicture}[scale=0.8]
\begin{scope}
\node (I) at (5,0) {$\begin{array}{c}  \NB{\tikz{\draw (-0.5,0) -- (0.5,0)}} \, \{1\} \\ \oplus \\ \NB{\tikz{\draw (-0.5,0) -- (0.5,0)}}\, \{-1\}  \end{array}$};
  \node (Gamma) at (-5,0) {${\tikz{\draw (-0.5,0) -- (-0.3, 0) .. controls (0,0.3) .. (0.3, 0) -- (0.5, 0); \draw (-0.3, 0) .. controls (0,-0.3) .. (0.3, 0) ;
}}$};
  \draw[<-] (I) .. controls +(-2,1) and + (2,1) .. (Gamma) coordinate[pos=0.5] (top) {};
  \draw[->] (I) .. controls +(-2,-1) and + (2,-1) .. (Gamma) node[pos=0.5] (bottom) {};
  \node[above] (TOP) at (top) 
{
$\left(
\begin{array}{l}
  \digoncapone  \\[1em] \digoncapzero
\end{array} 
\right)
$
};
  \node[below] (BOT) at (bottom) {
$\left(
\digoncupzero \quad  \digoncupone
\right)$
};
\end{scope}
\end{tikzpicture}
\caption{}
\label{fig:dec-digon}
\end{figure}
\end{prop}
\begin{proof}
This follows directly from Propositions~\ref{prop:digonrel} and~\ref{prop:bubbleremoval}.
\end{proof}

\begin{prop}
\label{prop:dec-square}
Suppose a graph $\Gamma$ contains a square. Denote by $\Gamma_1$ and $\Gamma_2$ the two smoothings of the square of $\Gamma$. Then there is a canonical isomorphism 
\[
\brak{\Gamma} \simeq \brak{\Gamma_1} \oplus \brak{\Gamma_2}
\]
given by Figure~\ref{fig:dec-square}.
\begin{figure}[ht]
\centering
\begin{tikzpicture}
\begin{scope}
\node (Smoothings) at (5,0) {
$\begin{array}{c} 
\NB{\tikz[rotate= -20]{
\draw (-0.5, -0.5) ..controls +(0.2, 0.2) and +(0.2, -0.2).. (-0.5, 0.5);  
\draw ( 0.5, -0.5) ..controls +(-0.2, 0.2) and +(-0.2, -0.2).. ( 0.5, 0.5);  
}} \\ 
\oplus \\ 
\NB{\tikz[rotate= -20]{
\draw (-0.5, -0.5) ..controls +(0.2, 0.2) and +(-0.2, 0.2).. ( 0.5, -0.5);  
\draw (-0.5,  0.5) ..controls +( 0.2,-0.2) and +(-0.2, -0.2).. ( 0.5, 0.5);  
}}
\end{array}
$
};
\node (Square) at (-5,0) {
${\tikz[rotate= -20]{
\draw (-0.3,-0.3) -- (-0.3, 0.3) -- (0.3,0.3) -- (0.3, -0.3) -- cycle;
\draw (-0.3,-0.3) -- (-0.5, -0.5);
\draw (-0.3, 0.3) -- (-0.5,  0.5);
\draw ( 0.3,-0.3) -- ( 0.5, -0.5);
\draw ( 0.3, 0.3) -- ( 0.5,  0.5);
}}$
};
  \draw[<-] (Smoothings) .. controls +(-2,1) and + (2,1) .. (Square) coordinate[pos=0.5] (top) {};
  \draw[->] (Smoothings) .. controls +(-2,-1) and + (2,-1) .. (Square) node[pos=0.5] (bottom) {};
  \node[above] (TOP) at (top) 
{
$
\left(
\begin{array}{c} 
\squareTOsmoothone  \\[1em]
\squareTOsmoothtwo
\end{array} 
\right)
$
 };
  \node[below] (BOT) at (bottom) {
 $\left(
 \smoothoneTOsquare \quad  \smoothtwoTOsquare
 \right)$
};
\end{scope}
\end{tikzpicture}
\caption{}
\label{fig:dec-square}
\end{figure}
\end{prop}
\begin{proof}
This follows directly from Propositions~\ref{prop:squarerel} and~\ref{prop:bubbleremoval}.
\end{proof}

\begin{prop}
\label{prop:dec-triangle}
If a graph $\Gamma'$ is obtained from a graph $\Gamma$ by replacing a vertex by a triangle, then there is a canonical isomorphism 
\[
\brak{\Gamma'} \simeq \brak{\Gamma}
\]
given by Figure~\ref{fig:dec-triangle}.
\begin{figure}[ht]
\centering
\begin{tikzpicture}
\begin{scope}
\node (vertex) at (3,0) {
$
\NB{\tikz[rotate= -20]{
\draw (0,0) -- (  0:0.5cm);
\draw (0,0) -- (120:0.5cm);
\draw (0,0) -- (240:0.5cm);
}
}
$
};
\node (triangle) at (-3,0) {
${\tikz[rotate= -20]{
\draw (  0:0.25cm) -- (  0:0.5cm);
\draw (120:0.25cm) -- (120:0.5cm);
\draw (240:0.25cm) -- (240:0.5cm);
\draw  (120:0.25cm) -- (240:0.25cm) -- (0:0.25cm) -- (120:0.25cm);
}}$
};
  \draw[<-] (vertex) .. controls +(-2,1) and + (2,1) .. (triangle) coordinate[pos=0.5] (top) {};
  \draw[->] (vertex) .. controls +(-2,-1) and + (2,-1) .. (triangle) node[pos=0.5] (bottom) {};
  \node[above] (TOP) at (top) 
{
$
\triangleTOvertex[0.5]
$
 };
  \node[below] (BOT) at (bottom) {
 $
\vertexTOtriangle[0.5]
$
};
\end{scope}
\end{tikzpicture}
\caption{}
\label{fig:dec-triangle}
\end{figure}
\end{prop}
\begin{proof}
This follows directly from Propositions~\ref{prop:trivalentbubble} and \ref{prop:verticesremoval}.
\end{proof}

An edge in a graph $\Gamma$ is called a \emph{bridge} if removing the edge increases the number of connected components of $\Gamma$ (by one). 

\begin{prop} If a planar trivalent graph $\Gamma$ has a bridge, then $\brak{\Gamma}=0$. 
\end{prop}
\begin{proof} Such a graph $\Gamma$ has no Tait colorings. 
Consequently, for any two foams $U,U_1$ into $\Gamma$, the foam $\omega(U_1)U$ has no admissible colorings, since an admissible coloring of $\omega(U_1)U$ would restrict to a Tait coloring of $\Gamma$. The bilinear form on $\Fo(\Gamma)$ is 
identically $0$, and the  
state space $\brak{\Gamma}=0$. 
\end{proof}

Original definition of $sl(3)$-link homology~\cite{SL3} included constructing state spaces $\mathrm{H}(\Gamma)$ for planar trivalent bipartite graphs $\Gamma$ as an intermediate step. In that case, the state spaces are graded free abelian groups and their graded rank is the quantum $sl(3)$ invariant of the graph, a.k.a. the Kuperberg bracket of $\Gamma$~\cite{Ku}. 

An equivariant extension of $sl(3)$ link homology and of these state spaces has been constructed by Mackaay and Vaz~\cite{MaV}. Let 
\[R_{\Z} = \Z[E_1, E_2, E_3]\]
be the integral version of the ring $R$. Mackaay 
and Vaz~\cite{MaV} denote $E_1=a$, $E_2=b$ and $E_3=c$, so the ring $R_{\Z}\cong \Z[a,b,c]$.  

For a planar trivalent bipartite graph $\Gamma$
Mackaay-Vaz state space $\HMV(\Gamma)$ is a free graded  $R_{\Z}$-module of graded rank (over $R_{\Z}$) equal 
to the Kuperberg bracket of $\Gamma$. 

\begin{prop} For bipartite webs $\Gamma$ there are canonical isomorphisms of graded $R$-modules, 
respectively graded $\kk$-vector spaces 
\begin{align*} 
\brak{\Gamma}_{\kk} &\cong \mathrm{H}(\Gamma)\otimes_\Z \kk, \\
\brak{\Gamma} &\cong \HMV(\Gamma)\otimes_{R_{\Z}} R .
\end{align*}
 These isomorphisms commute with maps between these spaces induced by oriented foams in $\R^2\times [0,1]$. 
\end{prop} 

\begin{proof}  
The space $\mathrm{H}(\Gamma)$ (resp. $\HMV(\Gamma)$) is obtained by quotienting out the free graded $\Z$-module (resp. $R_\Z$-module) generated by foams without seam vertices and with suitable orientability conditions on seam lines by the kernel of a bilinear form. Just like in this paper, the bilinear form is given by the evaluation of closed foams. This evaluation is $\Z$-valued (resp. $R_\Z$-valued) and is given by an algorithm rather than a formula. One can easily show that, after reducing coefficients from $\Z$ to $\Z/2$, this evaluation is precisely the same as the one given by $\brak{\bullet}_\kk$ (resp. by$\brak{\bullet}$), restricted to foams without seam vertices and with orientability conditions.

The isomorphisms 
$$
\Psi^\Gamma: \mathrm{H}(\Gamma) \otimes_\Z \kk \to \brak{\Gamma}_\kk, \ \  \Psi^\Gamma_\mathrm{MV}: \HMV(\Gamma)\otimes_{R_\Z} R \to \brak{\Gamma}
$$
are given by mapping foams generating $\mathrm{H}(\Gamma)$, respectively  $\HMV(\Gamma)$, to foams seen as elements of $\brak{\Gamma}_\kk$, respectively $\brak{\Gamma}$.

To prove that these morphisms are well-defined, we need to show that all relations valid in $\mathrm{H}(\Gamma)\otimes_\Z \kk$ (resp. $\HMV(\Gamma)\otimes_\Z \kk$) are valid in $\brak{\Gamma}_\kk$ (resp. $\brak{\Gamma}$). 

This follows from  the direct sum decomposition given by Propositions~\ref{prop:dec-circle}, \ref{prop:dec-digon} and \ref{prop:dec-square}. Indeed, this proves that a relation $\sum_i a_i F_i=0$ holds in $\brak{\Gamma}_\kk$ (resp. $\brak{\Gamma}$) if and only if 
$\sum_i a_i \brak{G\circ F_i}_\kk=0$ (resp. $\sum_i a_i \brak{G\circ F_i}=0$) holds for $G$ without seams vertices and respecting the orientability conditions. This proves as well that the maps are injective.

The maps $\Psi^\Gamma$ and $\Psi^\Gamma_\mathrm{MV}$ are  isomorphisms, since the same direct sum decomposition results show that these maps are surjective.
\end{proof}

A related but different approach to equivariant 
$sl(3)$ link homology has been sketched by  Morrison and Nieh~\cite[Appendix]{MoNi}. Morrison and Nieh avoid dots on foam's facets at the cost of inverting $2$ and $3$, while Mackaay and Vaz~\cite{MaV} utilize dots and use $\Z$ as the degree zero term of the ground ring of the theory. For this reason the match of our state spaces $\brak{\Gamma}$ with those in Mackaay--Vaz~\cite{MaV} modulo two for  bipartite planar graphs is immediate, while the relation to Morrison--Nieh's approach seems less straightforward. 
Division by three in their formulas is not an issue when reducing modulo two, but formula (3.5) of \cite{MoNi} contains division by two, obstructing a naive attempt to define a version of their construction modulo two. 

\section{Base change, inverting the discriminant, and graded dimensions} 
\label{sec:basechange}

\subsection{Base change}\

One of the immediate questions that we can't answer is whether 
$\lG$ is a free graded $R$-module for any 
planar trivalent graph $\Gamma$. This is one of the 
reasons to introduce base changes and work over different 
commutative rings. 

Assume there is a homomorphism 
$\psi: R \longrightarrow S$ of commutative rings. Recall that 
we defined $\psi$-evaluation $\brak{F}_{\psi}$, also called  $S$-evaluation $\brak{F}_{S}$, by composing the
evaluation $\brak{F}$ of a closed foam $F$ with the homomorphism 
$\psi$. This can be naturally extended to define 
$\psi$-state spaces of graphs $\Gamma$. 

Consider the free $S$-module $\Fo(\Gamma)_{\psi}$, also 
denoted $\Fo(\Gamma)_S$, to have a basis of all 
foams from the empty graph into $\Gamma$. There is a natural 
isomorphism of $S$-modules 
$\Fo(\Gamma) \otimes_R S \cong \Fo(\Gamma)_S.$ 
If $S$ is $\Z$-graded and $\psi$ is a grading-preserving homomorphism, then $\Fo(\Gamma)_S$ is a free graded $S$-module, 
with the degree of the foam given by the same formula as in 
the original case of $\Fo(\Gamma)$. 

There is a symmetric $S$-valued bilinear form on $\Fo(\Gamma)_S$ 
given by 
\[(U_1,U)_S = \brak{\omega(U_1)U}_S = \psi(\brak{\omega(U_1)U})\]
on foams $U_1, U$ into $\Gamma$. Define $\brak{\Gamma}_{\psi}$ as 
the quotient of $\Fo(\Gamma)_S$ by the kernel of this bilinear 
form. Another notation for $\brak{\Gamma}_{\psi}$ is 
$\brak{\Gamma}_S$. 
  
Just like in the original case, a relation $\sum a_i \brak{U_i} = 0$ holds in $\brak{\Gamma}_S$, with 
$a_i \in S$, if for any foam $U$ out of $\Gamma$, 
 $\sum a_i \psi(\brak{UU_i})=0$. 
If the ring $S$ is graded and $\psi$ is a grading-preserving homomorphism, then $\lG_S$ is naturally 
a graded $S$-module. 

Note that any element of $\brak{\Gamma}_S$ is a linear 
combination of foams into $\Gamma$ with coefficients in $S$. 
Consequently, the isomorphism $\Fo(\Gamma) \otimes_R S \stackrel{\cong}{\lra} 
\Fo(\Gamma)_S$ descends to a surjective map 
 \begin{align}\label{eq:surjective}
 \brak{\Gamma} \otimes_R S \ \lra \ \brak{\Gamma}_S.
 \end{align}    
The map is surjective since all generators $\brak{U}$ of the $S$-module 
on the right-hand side are in the image of the homomorphism, coming from the corresponding set of generators on the left-hand side. 
      
We are mostly interested in the case where $S$ is $\Z$-graded and 
homomorphism $\psi$ preserves the grading. We refer to this 
as a \emph{graded} base change $\psi$ or $S$. In this case  $\brak{\Gamma}_S$ is naturally a $\Z$-graded $S$-module. 

\begin{prop} The $S$-module $\lG_S$ is finitely-generated for any 
base change $(\psi,S)$. 
For a graded base change $(\psi,S)$, $\lG_{S}$ is a finitely-generated graded $S$-module.  
\end{prop} 

\begin{proof} The $R$-module on the left-hand side of map 
(\ref{eq:surjective}) is finitely generated. Choose a finite set of 
generators for that module. Their images under $\psi$ will 
span the $S$-module on the right-hand side. In the graded 
case, generators can be chosen to be homogeneous. 
\end{proof} 

\begin{prop} All direct sum decompositions for state spaces 
derived in Section~\ref{sub:dir_sum} hold with state spaces $\brak{\bullet}_S$  for any graded 
base change $(\psi,S)$. With grading shifts dropped from 
the relations, they hold 
for any base change. 
\end{prop}

\begin{proof} All identities on foams used to prove direct sum decompositions hold under any homomorphism $\psi$ as well. 
\end{proof} 

One base change that we have already encountered is 
the graded homomorphism $\psi_0: R \lra \kk$ with $\psi_0(E_i)=0$ 
for $i=1,2,3$, see Section~\ref{sec:combinatorialKM}. For this base change the $\kk$-state space $\lG_{\kk}$ is 
a finite-dimensional graded $\kk$-vector space. Our proof of 
Kronheimer--Mrowka Conjecture 8.9 for foams in $\R^3$ implies 
the following result. 

\begin{prop} There is a functorial isomorphism 
\[J^{\flat}(\Gamma) \cong \brak{\Gamma}_{\kk}\]
for all planar trivalent graphs $\Gamma$. 
\end{prop}

The isomorphisms are functorial relative to maps in these 
two homology theories induced by foams with boundary. The homology 
theory $J^{\flat}$ for planar trivalent graphs is defined in~\cite[Section 8.3]{KM1} assuming their Conjecture 8.9 for foams in $\R^3$.

\begin{prop} For any base change $(\psi,S)$ the $S$-module 
$\brak{\Gamma}_S$ is a submodule of a free 
$S$-module of finite rank. For a graded base change, 
$\brak{\Gamma}_S$ is a submodule of a graded free $S$-module 
of finite rank. 
\end{prop} 

\begin{proof} Fix a planar graph $\Gamma$ and 
choose a collection of homogeneous generators $a_1, \dots, a_n$ of 
the $R$-module $\brak{\Gamma}$. The elements 
$b_1, \dots, b_n$, where $b_i = \psi(a_i)$, generate the
$S$-module $\brak{\Gamma}_S$. The bilinear pairing $(,)_S$ on 
$\brak{\Gamma}_S$ is non-degenerate, and an element $b\in \brak{\Gamma}_S$ 
is determined by its couplings $(b,b_i)_S \in S$ over 
$i=1, \dots, n$. Thus, to each $b$ in $b\in \brak{\Gamma}_S$ 
we can assign an element of $S^n$, namely
\[((b,b_1)_S, (b,b_2)_S, \dots, (b,b_n)_S)^T.\]
This assignment is an injective $S$-module map 
$\brak{\Gamma}_S \lra S^n$, realizing $\brak{\Gamma}_S$ as 
a submodule of a free $S$-module of finite rank. 

If $(\psi,S)$ is graded, the inclusion is that of graded 
modules. One can be more precise and write $S^f$ instead of 
$S^n$ where  
$f = \sum_{i=1}^n q^{-m_i}$, with $m_i$ the degree of $a_i$. 
Here the degrees of generators of a free module are encoded via 
sum of powers of $q$.  
\end{proof}

\vspace{0.2in}

If a commutative ring $S$ has no zero divisors, then 
$S$-module $\brak{\Gamma}_S$ is torsion-free, that is, 
$am=0$ for $a\in S$ and $m\in \brak{\Gamma}_S$ implies that 
$a=0$ or $m=0$. 

Recall that PID stands for 'principal ideal domain'. 

\begin{prop} \label{prop:S_free} If $S$ is a (graded) PID then $\lG_{S}$ is a finitely-generated (graded) free $S$-module. 
\end{prop} 

\begin{proof} $\lG_{S}$ is a finitely-generated $S$-module 
with no torsion, necessarily free of finite rank. 
\end{proof}

In Section~\ref{sec:base_change_pid} we will consider base changes into graded principal ideal domains. In the next section we will use the graded base change 
\[\psi_{\D} \ : \ R \lra R[\D^{-1}],\]
where $\D=E_1E_2 + E_3$ is the discriminant of the 
polynomial $x^3+E_1 x^2 + E_2 x + E_3 \in R[x]$. 
This polynomial factors into $(x+X_1)(x+X_2)(x+X_3)$ in 
the larger ring $R'[x]$, and its discriminant $\D$ is given by:
\begin{align}\label{eq:discriminant}
\D \define (X_1+X_2)(X_1+X_3)(X_2+X_3) = E_1E_2 + E_3. 
\end{align}
We denote the state space for this base change by $\brak{\Gamma}_{\D}$.

\subsection{Facet decorations of negative degrees}\ 
\label{sec:facet_neg}

In this subsection, we introduce some additional decoration which can float on faces of foams. For doing so we need to work over a ring slightly larger than $R$. We invert
the discriminant $\D$ given by equation (\ref{eq:discriminant}) and work over the 
ring 
\begin{align}\label{eq:discriminant-ring}
R_\D\define R[\D^{-1}] = \kk[E_1, E_2, E_3, \D^{-1}].
\end{align} 

We introduce two additional decorations on foams which, just like dots, freely float on facets: the \emph{triangle} (denoted by $\triangle$) and the \emph{square} (denoted by $\square$). We extend the evaluation to foams having these extra decorations. The triangle and 
square 
invert the expressions 
$\bullet + E_1$ and $\bullet^2+E_2$, respectively, 
where $\bullet$ denotes a dot on a facet: 
\[
\triangle \leftrightsquigarrow \frac{1}{
\bullet + E_1} \quad \textrm{and} \quad
\square \leftrightsquigarrow \frac{1}{
\bullet^2 + E_1}.
\]  
For a given coloring $c$, each triangle on a 
facet colored $i$ contributes $(X_{i}+E_1)^{-1}$ to the 
product and each square on an $i$-colored facet contributes 
$(X_i^2+E_2)^{-1}$ to the product term for $c$. 
Note that 
\begin{align*}
X_i + E_1   = & X_j + X_k, \\
X_i^2+E_2  = &  (X_i + X_j)(X_i + X_k), 
\end{align*}
where $j,k$ are the remaining colors, 
since 
\begin{align*}
E_1 = X_i + X_j + X_k, \ \ E_2 = X_i X_j + X_i X_k + X_j X_k.
\end{align*} 
The product 
\begin{align}\label{eq:product-D}
(X_i + E_1)(X_i^2+E_2) = (X_j + X_k)(X_i + X_j)(X_i + X_k)  = E_1 E_2 + E_3 = \D 
\end{align}
is symmetric in $X_i,X_j,X_k$ and equals the discriminant 
$\D$. This discriminant will appear in the denominators 
of our product terms, so to make sense out of floating triangles and squares it suffices to invert $\D$ in 
the ring $R$ and work in the localized ring $R_{\D}$.
Note that the localized ring is still $\Z$-graded, with $\D^{-1}$ in degree $-6$. As a graded $R$-module or even $R_{\D}$-module, ring $R_{\D}$ is periodic with period $6$, via the multiplication by $\D^{\pm 1}$. 
  
Allowing floating triangles and squares on facets, the definition (\ref{eq:Q}) of $Q(F,c)$ remains unchanged, while the definition of $P(F,c)$ becomes:
\[
P(F,c) = \prod_{f  \in f(F)}  \frac{X^{d(f)}_{c(f)}}{(X_{c(f)} + E_1)^{t(f)}(X_{c(f)}^2 + E_2)^{s(f)}}, 
\]
where $t(f)$ and $s(f)$ are respectively the number of triangles and squares on facet $f$.

Finally, define 
\[\brak{F}_{\D} = \sum_{c \in \adm(F)} \frac{P(F,c)}{Q(F,c)}. \]

The degrees of $\triangle$ and $\square$ are $-2$ and $-4$ respectively. For a foam $F$ without triangles or squares, $\brak{F}_{\D}= \brak{F}$. 

Note that $P(F,c)$ is no longer a polynomial, hence Theorem~\ref{thm:evl-sym-pol} does not hold anymore for $\brak{\bullet}_\D$. However, if $F$ is a foam of degree $d$, $\brak{F}_{\D}$ is an homogeneous element of $R_\D$ of degree $d$.

\begin{example}
A sphere with one triangle evaluates to $\D^{-1}$. 
\end{example}

The next lemma says that the triangle and the square decoration can be expressed as a linear combination of dots (provided $\D^{-1}$ exists). Hence they do not enriched the theory, but as we shall see, it is convenient to have them in computations. 

\begin{lemma}\label{lem:simplify-square-triangle}
The following local relations hold:
\begin{align*}
\brak{\decorateddisk{\triangle}{0.5}{0.8} }_\D &=
\frac{1}{\D} \left(
\brak{\decorateddisk{\bullet\,\,\bullet}{0.7}{0.8} }_\D 
+ E_2
\brak{\decorateddisk{ }{0.7}{0.8} }_\D
\right), 
\\
\brak{\decorateddisk{\square}{0.5}{0.8} }_\D &=
\frac{1}{\D} \left(
\brak{\decorateddisk{\bullet }{0.7}{0.8} }_\D 
+ E_1
\brak{\decorateddisk{ }{0.7}{0.8} }_\D
\right), \\
\brak{\decorateddisk{\square\,\triangle}{0.5}{0.8} }_\D &=
\frac{1}{\D} 
\brak{\decorateddisk{}{0.7}{0.8} }_\D. 
\end{align*}
\end{lemma}
\begin{proof}
These relations follow directly from the identity
\begin{align*}
\D = (E_1 + X_i)(E_2 + X_i^2),
\end{align*}
for any $i$ in $\{1,2,3\}$. The same 
equation implies the last relation.  
\end{proof}

The next lemma says that the square decoration $\square$ added to a facet can be interpreted as the inverse of forming the  connected sum with a two-torus along the facet, that is, adding a handle. Likewise, 
adding the triangle $\triangle$ to a facet is 
the inverse of connected sum with a 
capped torus. In other words, it's the inverse of 
forming a connected sum with a torus and gluing on a disk along the connecting circle. 
\begin{lemma}
The following local relations hold:
\begin{align*}
\brak{\!
\NB{\tikz[scale = 0.8]{\node[scale= 0.7] at(1.25, 0.2) {$\square$}; }}
\!}_\D &= \brak{\decorateddisk{}{0.7}{0.8}}_\D,
\\
\brak{\!
\NB{\tikz[scale = 0.8]{\node[scale= 0.7] at(1.25, 0.18) {$\triangle$}; }}
\!}_\D &= \brak{\decorateddisk{}{0.7}{0.8}}_\D, \\
\brak{\,
\NB{\tikz[scale = 0.8]{\begin{scope}
  \coordinate (A) at (0,0);
  \coordinate (B) at (2,0);
  \coordinate (C) at (2.5,0.5);
  \coordinate (D) at (0.5,0.5);
  \coordinate (d) at (1, 0.5);
  \coordinate (c) at (1.5, 0.5);
  \coordinate (dm) at (0.5, 0.875);
  \coordinate (cm) at (2, 0.875);
  \coordinate (dT) at (1, 1.25);
  \coordinate (cT) at (1.5, 1.25);
  \coordinate (dM) at (0.5, 1.625);
  \coordinate (cM) at (2, 1.625);
  \coordinate (TT) at (1.25, 2);
  \coordinate (m1) at (0.5, 0.25);
  \coordinate (m2) at (2, 0.25);
  \coordinate (mt1) at (0.5, 0.25);
  \coordinate (mt2) at (2, 0.25);
  \coordinate (p1) at (1.5, 1.25);
  \coordinate (p2) at (1, 1.25);
  \coordinate (M1) at (0.9, 0.9);
  \coordinate (M2) at (1.6, 0.9);
  \coordinate (M3) at (0.9, 1.65);
  \coordinate (M4) at (1.6, 1.65);
  \draw[dotted] (c) --(d);
  \fill[fill = gray!50!white] (cT) arc (0:361:0.25cm and 0.08cm);
  \draw[]  (cT) arc (0:-180: 0.25cm and 0.08cm);
  \draw[densely dotted, fill = gray!50!white]  (cT) arc (0:180: 0.25cm and 0.08cm);
  \draw (d) -- (D) -- (A) -- (B) -- (C) -- (c);
  \draw (m1)  .. controls +(0.2,0) and 
+(0, -0.2) .. (d)  .. controls +(0,0.2) and 
+(0, -0.2) .. (dm)  .. controls +(0,0.2) and 
+(0, -0.2) .. (dT)  .. controls +(0,0.2) and 
+(0, -0.2) .. (dM)  .. controls +(0,0.2) and 
+(-0.5, 0) .. (TT)  .. controls +(0.5,0) and 
+(0, 0.2) .. (cM)  .. controls +(0,-0.2) and 
+(0, 0.2) .. (cT)  .. controls +(0,-0.2) and 
+(0, 0.2) .. (cm)  .. controls +(0,-0.2) and 
+(0, 0.2) .. (c)   .. controls +(0,-0.2) and 
+(-0.2, 0) .. (m2);  
\draw (M1)  .. controls +(0.2,-0.1) and +(-0.2, -0.1) .. (M2) coordinate [pos =0.2] (n1) coordinate [pos= 0.8] (n2);
  \draw (n1)  .. controls +(0.16, 0.08) and +(-0.16, 0.08) .. (n2);
  \draw (M3)  .. controls +(0.2,-0.1) and +(-0.2, -0.1) .. (M4) coordinate [pos =0.2] (n3) coordinate [pos= 0.8] (n4);
  \draw (n3)  .. controls +(0.16, 0.08) and +(-0.16, 0.08) .. (n4);
\end{scope} }}
\,}_\D &= \D\brak{\decorateddisk{}{0.7}{0.8}}_\D.
\end{align*}
\end{lemma}
\begin{proof}
Straightforward
\end{proof}
Notice that the connected sum with both torus and capped torus is equivalent to connected sum with a genus two surface capped by a disk along a separating curve in the middle. This operation, upon foam evaluation, is equivalent to multiplication by $\D= E_1E_2+E_3$, which is an element in $R$, so it does not matter which facet to apply it to.  

\begin{lemma}\label{lem:triangle-mig}
The following local relation holds:
\[
\brak{\vertexIdottedonetrianglezero}_\D +
\brak{\vertexIdottedzerotriangleone}_\D =
\brak{\vertexInodot}_\D
\]
\end{lemma}
\begin{proof}
This follows immediately from the definition of the evaluation of closed foam. Indeed for a fixed coloring  the identity reads:
\[
\frac{X_i}{E_1 + X_k} + \frac{X_j}{E_1+ X_k} =
\frac{X_i}{X_i + X_j} + \frac{X_j}{X_i+ X_j} 
= 1, 
\]
for $\{i,j,k\} =\{1,2,3\}$ and $k$ the color of the triangle-decorated facet.
One can as well deduce this identity from Lemma~\ref{lem:simplify-square-triangle} and Proposition~\ref{prop:dotmigration}.
\end{proof}
\begin{lemma} \label{lem:doublesaddle}
The following local relation holds:
\begin{align*}
\brak{\Gamma_0}_\D &= 
\brak{\Gamma_1}_\D +
\brak{\Gamma_2}_\D,
\end{align*}
where:
\[
\NB{\tikz{\begin{scope}[scale =0.47]
  \begin{scope}[xshift = -8cm]
\begin{scope}
    \draw[draw= gray, line width = 2mm] (-1.5,-4.5) -- +(0,9); 
    \draw[draw= gray, line width = 2mm] ( 1.5,-4.5) -- +(0,9); 
    \draw[draw= gray, very thick] (-1.5,-4.5) -- +(3,0); 
    \draw[draw= gray, very thick] (-1.5,-1.5) -- +(3,0); 
    \draw[draw= gray, very thick] (-1.5, 1.5) -- +(3,0); 
    \draw[draw= gray, very thick] (-1.5, 4.5) -- +(3,0); 
    \draw[draw= white, dotted, line width =1.2mm]  (-1.5,-4.5) -- +(0,9); 
    \draw[draw= white, dotted, line width =1.2mm]  ( 1.5,-4.5) -- +(0,9); 
  \end{scope}
\node at (-3.7,0) {$\Gamma_0 :=$};
\node at (2, 0) {,};
  \begin{scope}[yshift= 3 cm]
    \coordinate (A) at (-1,-1);
    \coordinate (B) at (1,-1);
    \coordinate (C) at (1,1);
    \coordinate (D) at (-1,1);
    \coordinate (a) at (-.5,-.5);
    \coordinate (b) at (.5,-.5);
    \coordinate (c) at (.5,.5);
    \coordinate (d) at (-.5,.5);
    \draw (A) ..controls (a) and (b) .. (B);
    \draw (C) ..controls (c) and (d) .. (D);
  \end{scope}
 \begin{scope}[yshift= 0 cm]
    \coordinate (A) at (-1,-1);
    \coordinate (B) at (1,-1);
    \coordinate (C) at (1,1);
    \coordinate (D) at (-1,1);
    \coordinate (a) at (-.5,-.5);
    \coordinate (b) at (.5,-.5);
    \coordinate (c) at (.5,.5);
    \coordinate (d) at (-.5,.5);
    \draw (A) ..controls (a) and (b) .. (B);
    \draw (C) ..controls (c) and (d) .. (D);
  \end{scope}
 \begin{scope}[yshift= -3 cm]
    \coordinate (A) at (-1,-1);
    \coordinate (B) at (1,-1);
    \coordinate (C) at (1,1);
    \coordinate (D) at (-1,1);
    \coordinate (a) at (-.5,-.5);
    \coordinate (b) at (.5,-.5);
    \coordinate (c) at (.5,.5);
    \coordinate (d) at (-.5,.5);
    \draw (A) ..controls (a) and (b) .. (B);
    \draw (C) ..controls (c) and (d) .. (D);
  \end{scope}    
  \end{scope}
\begin{scope}
    \draw[draw= gray, line width = 2mm] (-1.5,-4.5) -- +(0,9); 
    \draw[draw= gray, line width = 2mm] ( 1.5,-4.5) -- +(0,9); 
    \draw[draw= gray, very thick] (-1.5,-4.5) -- +(3,0); 
    \draw[draw= gray, very thick] (-1.5,-1.5) -- +(3,0); 
    \draw[draw= gray, very thick] (-1.5, 1.5) -- +(3,0); 
    \draw[draw= gray, very thick] (-1.5, 4.5) -- +(3,0); 
    \draw[draw= white, dotted, line width =1.2mm]  (-1.5,-4.5) -- +(0,9); 
    \draw[draw= white, dotted, line width =1.2mm]  ( 1.5,-4.5) -- +(0,9); 
  \end{scope}
\node at (-3.7,0) {$\Gamma_1 :=$};
  \begin{scope}[yshift= 3 cm]
    \coordinate (A) at (-1,-1);
    \coordinate (B) at (1,-1);
    \coordinate (C) at (1,1);
    \coordinate (D) at (-1,1);
    \coordinate (a) at (-.5,-.5);
    \coordinate (b) at (.5,-.5);
    \coordinate (c) at (.5,.5);
    \coordinate (d) at (-.5,.5);
    \draw (A) ..controls (a) and (b) .. (B) node[scale = 0.7, midway] {$\square$};
    \draw (C) ..controls (c) and (d) .. (D);
  \end{scope}
 \begin{scope}[yshift= 0 cm]
    \coordinate (A) at (-1,-1);
    \coordinate (B) at (1,-1);
    \coordinate (C) at (1,1);
    \coordinate (D) at (-1,1);
    \coordinate (a) at (-.5,-.5);
    \coordinate (b) at (.5,-.5);
    \coordinate (c) at (.5,.5);
    \coordinate (d) at (-.5,.5);
    \draw (A) ..controls (a) and (d) .. (D);
    \draw (C) ..controls (c) and (b) .. (B);
  \end{scope}
 \begin{scope}[yshift= -3 cm]
    \coordinate (A) at (-1,-1);
    \coordinate (B) at (1,-1);
    \coordinate (C) at (1,1);
    \coordinate (D) at (-1,1);
    \coordinate (a) at (-.5,-.5);
    \coordinate (b) at (.5,-.5);
    \coordinate (c) at (.5,.5);
    \coordinate (d) at (-.5,.5);
    \draw (A) ..controls (a) and (b) .. (B);
    \draw (C) ..controls (c) and (d) .. (D);
  \end{scope}
  \begin{scope}[xshift = 10cm]
    \begin{scope}
    \draw[draw= gray, line width = 2mm] (-1.5,-4.5) -- +(0,9); 
    \draw[draw= gray, line width = 2mm] ( 1.5,-4.5) -- +(0,9); 
    \draw[draw= gray, very thick] (-1.5,-4.5) -- +(3,0); 
    \draw[draw= gray, very thick] (-1.5,-1.5) -- +(3,0); 
    \draw[draw= gray, very thick] (-1.5, 1.5) -- +(3,0); 
    \draw[draw= gray, very thick] (-1.5, 4.5) -- +(3,0); 
    \draw[draw= white, dotted, line width =1.2mm]  (-1.5,-4.5) -- +(0,9); 
    \draw[draw= white, dotted, line width =1.2mm]  ( 1.5,-4.5) -- +(0,9); 
  \end{scope}
\node at (-5.9,0) {and};
\node at (-3.7,0) {$\Gamma_2 :=$};
  \begin{scope}[yshift= 3 cm]
    \coordinate (A) at (-1,-1);
    \coordinate (B) at (1,-1);
    \coordinate (C) at (1,1);
    \coordinate (D) at (-1,1);
    \coordinate (a) at (-.5,-.5);
    \coordinate (b) at (.5,-.5);
    \coordinate (c) at (.5,.5);
    \coordinate (d) at (-.5,.5);
    \draw (A) ..controls (a) and (b) .. (B);
    \draw (C) ..controls (c) and (d) .. (D);
  \end{scope}
 \begin{scope}[yshift= 0 cm]
    \coordinate (A) at (-1,-1);
    \coordinate (B) at (1,-1);
    \coordinate (C) at (1,1);
    \coordinate (D) at (-1,1);
    \coordinate (a) at (-.5,0);
    \coordinate (b) at (.5,-.5);
    \coordinate (c) at (.5, 0);
    \coordinate (d) at (-.5,.5);
    \draw (A)  -- (a) -- (D);
    \draw (C) -- (c) -- (B);
    \draw (c) -- (a) node[scale = 0.7, midway] {$\triangle$};
  \end{scope}
 \begin{scope}[yshift= -3 cm]
    \coordinate (A) at (-1,-1);
    \coordinate (B) at (1,-1);
    \coordinate (C) at (1,1);
    \coordinate (D) at (-1,1);
    \coordinate (a) at (-.5,-.5);
    \coordinate (b) at (.5,-.5);
    \coordinate (c) at (.5,.5);
    \coordinate (d) at (-.5,.5);
    \draw (A) ..controls (a) and (b) .. (B);
    \draw (C) ..controls (c) and (d) .. (D);
  \end{scope}
\node at (2, 0) {.};
  \end{scope}
\end{scope}

}}
\]
In other words, $\Gamma_0$ is locally the identity on two strands, $\Gamma_1$ is a square-decorated double-saddle on these two strands and $\Gamma_2$ is the composition of a zip and an unzip, with a triangle decoration on the inner disk.
\end{lemma}
\begin{proof}
One can prove the relation directly. The computations are similar to the ones in the proof of Proposition~\ref{prop:squarerel}.  
Alternatively, via Lemma~\ref{lem:simplify-square-triangle}, one can rewrite the square and the triangle in terms of dots and $\D^{-1}$, then use the local relations of Subsection~\ref{sec:rel-btwn-ev} to complete the proof.
\end{proof}

\subsection{The square of four-end graphs} \ 
\label{sec:four_end}

Let us consider four webs $(\Gamma_i)_{i\in \{I,H,=,||\}}$ which are identical except in a small 2-dimensional ball where they are given by:
\[
\Gamma_I = 
\NB{\tikz[scale = 0.5]{\draw[dotted] (0,0) circle (1cm); 
\draw   (45:1cm) -- (0, 0.3); 
\draw  (135:1cm) -- (0, 0.3); 
\draw   (0,-0.3) -- (0, 0.3);
\draw  (-45:1cm) -- (0,-0.3);
\draw (-135:1cm) -- (0,-0.3);
}}
,\quad
\Gamma_H = 
\NB{\tikz[scale = 0.5]{\draw[dotted] (0,0) circle (1cm); 
\draw   (45:1cm) -- ( 0.3, 0); 
\draw  (-45:1cm) -- ( 0.3, 0); 
\draw   (-0.3,0) -- ( 0.3, 0);
\draw ( 135:1cm) -- (-0.3, 0);
\draw (-135:1cm) -- (-0.3, 0);
}}
,\quad
\Gamma_{=} = 
\NB{\tikz[scale = 0.5]{\draw[dotted] (0,0) circle (1cm); 
\draw   (45:1cm) ..controls (0.3, 0) and (-0.3,0) ..  (135:1cm);
\draw  (-45:1cm) ..controls (0.3, 0) and (-0.3,0) .. (-135:1cm);
}}
\quad \textrm{and} \quad
\Gamma_{||} = 
\NB{\tikz[scale = 0.5]{\draw[dotted] (0,0) circle (1cm); 
\draw   (45:1cm) ..controls (0, 0.3) and (0,-0.3) ..  (-45:1cm);
\draw  (135:1cm) ..controls (0, 0.3) and (0,-0.3) .. (-135:1cm);
}}.
\]
We consider four cobordisms:
\begin{itemize}
\item the neighborhood of a seam vertex from $\Gamma_I$ to $\Gamma_H$ denoted by $F_{I\to H}$, of degree~$1$,
\item an unzip from $\Gamma_H$ to $\Gamma_=$ denoted by $F_{H\to =}$, of degree $1$,
\item a saddle from $\Gamma_=$ to $\Gamma_{||}$ denoted by $F_{=\to ||}$, of degree $2$,
\item and a zip from $\Gamma_{||}$ to $\Gamma_I$ denoted by $F_{||\to I}$, of degree $1$.
\end{itemize}
These foams are depicted in Figure~\ref{fig:foam-for-exact-square}.
\begin{figure}[ht]
\centering
\tikz[scale = 1]{
\input{
\imagesfolder/km_foam-for-exact-square
}
}
\caption{From left to right: $F_{I\to H}$, $F_{H\to =}$, $F_{=\to ||}$ and $F_{|| \to I}$.}\label{fig:foam-for-exact-square}
\end{figure}

\begin{lemma}\label{lem:compinsuqareis0}
	Compositions $F_{H\to =} \circ F_{I\to H}$, $F_{=\to ||} \circ F_{H\to =}$, $F_{||\to I} \circ F_{=\to ||}$ and $F_{I\to H} \circ F_{||\to I}$ are mapped to $0$ by the functor \brak{\bullet}.
\end{lemma}

\begin{proof}
	In each case, the foam obtained by composing the two cobordisms has no admissible coloring. 
\end{proof}

We consider the following square of webs and web cobordisms 
\begin{align}
\ensuremath{\vcenter{\hbox{
\begin{tikzpicture}[scale =2]
  \node (A) at (0,1) {$\Gamma_I$};
  \node (B) at (2,1) {$\Gamma_H$};
  \node (C) at (2,0) {$\Gamma_=$};
  \node (D) at (0,0) {$\Gamma_{||}$};
  \draw[-to] (A)--(B) node [midway,above] {$F_{I\to H}$};
  \draw[-to] (D)--(A) node [midway,left] {$F_{||\to I}$};
  \draw[-to] (B)--(C) node [midway,right] {$F_{H \to =}$};
  \draw[-to] (C)--(D) node [midway,below] {$F_{=\to ||}$};
\end{tikzpicture}
}}}\label{eq:square0}
\end{align}

Applying the functor $\brak{\bullet}$ results 
in a 4-periodic complex 
 \begin{align}
\ensuremath{\vcenter{\hbox{
\begin{tikzpicture}[scale =2]
  \node (A) at (0,1) {$\brak{\Gamma_I}$};
  \node (B) at (2,1) {$\brak{\Gamma_H}$};
  \node (C) at (2,0) {$\brak{\Gamma_=}$};
  \node (D) at (0,0) {$\brak{\Gamma_{||}}$};
  \draw[-to] (A)--(B) node [midway,above] {$\brak{F_{I\to H}}$};
  \draw[-to] (D)--(A) node [midway,left] {$\brak{F_{||\to I}}$};
  \draw[-to] (B)--(C) node [midway,right] {$\brak{F_{H \to =}}$};
  \draw[-to] (C)--(D) node [midway,below] {$\brak{F_{=\to ||}}$};
\end{tikzpicture}
}}}\label{eq:square}
\end{align}
of graded $R$-modules. Differential in this complex is homogeneous relative to internal grading, of degrees $1, 1, 1, 2$ respectively, going clockwise starting from the map on the left. 
It is not clear whether this square is always 
exact.

A similar square (a 4-periodic complex) can be obtained by applying the functor $\brak{\bullet}_S$ to square (\ref{eq:square0}) for any base change $\psi:R\lra S$. 

\begin{prop}\label{prop:exact_sq}
The square obtained for the base change $(\psi_\D, R[\D^{-1}])$ is exact.
\end{prop}

\begin{proof}
We define four cobordisms (or foams) $G_{H\to I}$, $G_{I\to ||}$, $G_{=\to H}$, and $G_{||\to =}$ (the source and the target of these cobordisms should be clear from the notation) and we prove the following identities:
\begin{align}
\brak{\Id_{\Gamma_I   }}_\D &= \brak{G_{H \to I } \circ F_{I \to H }}_\D + \brak{F_{|| \to I }\circ G_{I \to ||}}_\D, \label{eq:idI}\\
\brak{\Id_{\Gamma_H   }}_\D &= \brak{G_{= \to H } \circ F_{H \to = }}_\D + \brak{F_{I  \to H }\circ G_{H \to I }}_\D, \label{eq:idH}\\
\brak{\Id_{\Gamma_=   }}_\D &= \brak{G_{||\to = } \circ F_{= \to ||}}_\D + \brak{F_{H  \to = }\circ G_{= \to H }}_\D, \label{eq:idh}\\
\brak{\Id_{\Gamma_{||}}}_\D &= \brak{G_{I \to ||} \circ F_{||\to I }}_\D + \brak{F_{=  \to ||}\circ G_{||\to = }}_\D. \label{eq:idv}
\end{align}
\begin{itemize}
 \item The foam $G_{H \to I }$ is the neighborhood of a seam vertex with one triangle on the facet bounding the internal edge of the $H$. It has degree $-1$.
 \item The foam $G_{I \to ||}$ is an unzip with one triangle on the facet bounding the internal edge of the $I$. It has degree $-1$.
 \item The foam $G_{|| \to =}$ is a saddle with one square. It has degree $-2$.
 \item The foam $G_{= \to H}$ is a zip with one triangle on the facet bounding the internal edge of the $I$. It has degree $-1$.
\end{itemize}
These foams are depicted in Figure~\ref{fig:homotopies}. 
\begin{figure}[ht]
\centering
\tikz[scale = 1]{
\input{
\imagesfolder/km_homot-for-exact-square
}
}
\caption{From left to right: $G_{H\to I}$, $G_{I\to ||}$, $G_{||\to =}$ and $G_{= \to H}$.}\label{fig:homotopies}
\end{figure}
  
The identities (\ref{eq:idv}) and (\ref{eq:idh}) are given by Lemma~\ref{lem:doublesaddle}.
The identities (\ref{eq:idI}) and (\ref{eq:idH}) are essentially the same, so we only prove (\ref{eq:idI}). 
  
In this proof, we consider (pieces of) foams which are diffeomorphic to (pieces of) webs times an interval. In order to represent such foams in the computations, we will only draw the (pieces of) webs and indicate the dots on the edges.

Thanks to Proposition~\ref{prop:IHI}, we have:
\[
\brak{G_{H \to I } \circ F_{I \to H }} = 
\brak{
\NB{\tikz[scale = 0.6]{\draw[dotted] (0,0) circle (1cm); 
\draw  (135:1cm) -- (0, 0.3) coordinate[pos= 0.5] (P); 
\draw   (45:1cm) -- (0, 0.3); 
\draw   (0,-0.3) -- (0, 0.3) coordinate[pos= 0.5] (M); 
\draw  (-45:1cm) -- (0,-0.3);
\draw (-135:1cm) -- (0,-0.3);
\fill (P) circle (0.1cm);
\node[scale =0.5] at (M) {$\triangle$};
}}
}
+
\brak{
\NB{\tikz[scale = 0.6]{\draw[dotted] (0,0) circle (1cm); 
\draw   (45:1cm) -- (0, 0.3);
\draw  (135:1cm) -- (0, 0.3); 
\draw   (0,-0.3) -- (0, 0.3)  coordinate[pos= 0.5] (M); 
\draw (-135:1cm) -- (0,-0.3)  coordinate[pos= 0.5] (P); 
\draw  (-45:1cm) -- (0,-0.3);
\fill (P) circle (0.1cm);
\node[scale =0.5] at (M) {$\triangle$};
}}
}
.
\]
Thanks to Proposition~\ref{prop:unzipzip}, 
\[
\brak{F_{|| \to I } \circ G_{I \to || }} = 
\brak{
\NB{\tikz[scale = 0.6]{\draw[dotted] (0,0) circle (1cm); 
\draw  (135:1cm) -- (0, 0.3) coordinate[pos= 0.5] (P); 
\draw   (45:1cm) -- (0, 0.3); 
\draw   (0,-0.3) -- (0, 0.3) coordinate[pos= 0.5] (M); 
\draw  (-45:1cm) -- (0,-0.3);
\draw (-135:1cm) -- (0,-0.3);
\fill (P) circle (0.1cm);
\node[ scale =0.5] at (M) {$\triangle$};
}}
}
+
\brak{
\NB{\tikz[scale = 0.6]{\draw[dotted] (0,0) circle (1cm); 
\draw   (45:1cm) -- (0, 0.3);
\draw  (135:1cm) -- (0, 0.3); 
\draw   (0,-0.3) -- (0, 0.3) coordinate[pos= 0.5] (M); 
\draw (-135:1cm) -- (0,-0.3) ;
\draw  (-45:1cm) -- (0,-0.3) coordinate[pos= 0.5] (P); 
\fill (P) circle (0.1cm);
\node[ scale =0.5] at (M) {$\triangle$};
}}
}
.
\]
This gives:
\begin{align*}
\brak{G_{H \to I } \circ F_{I \to H }} + \brak{F_{|| \to I }\circ G_{I \to ||}} &=
\brak{
\NB{\tikz[scale = 0.6]{\draw[dotted] (0,0) circle (1cm); 
\draw  (135:1cm) -- (0, 0.3);
\draw   (45:1cm) -- (0, 0.3); 
\draw   (0,-0.3) -- (0, 0.3) coordinate[pos= 0.5] (M); 
\draw  (-45:1cm) -- (0,-0.3);
\draw (-135:1cm) -- (0,-0.3) coordinate[pos= 0.5] (P); 
\fill (P) circle (0.1cm);
\node[scale =0.5] at (M) {$\triangle$};
}}
}
+
\brak{
\NB{\tikz[scale = 0.6]{\draw[dotted] (0,0) circle (1cm); 
\draw   (45:1cm) -- (0, 0.3);
\draw  (135:1cm) -- (0, 0.3); 
\draw   (0,-0.3) -- (0, 0.3) coordinate[pos= 0.5] (M); 
\draw (-135:1cm) -- (0,-0.3) ;
\draw  (-45:1cm) -- (0,-0.3) coordinate[pos= 0.5] (P); 
\fill (P) circle (0.1cm);
\node[scale =0.5] at (M) {$\triangle$};
}}
} 
=  \brak{\Id_{\Gamma_I}},
\end{align*}
where the second equality comes from Proposition~\ref{lem:triangle-mig}. 
\end{proof}

Let $Q(R)$ be the field of fractions of $R$. 
Note that $Q(R)$ is naturally isomorphic to the 
field of fractions of $R_{\D}$ as well. 
Given a projective $R_{\D}$-module $P$, define 
its rank $\rk(P)$ as the dimension 
of the $Q(R)$-vector space $P\otimes_{R_{\D}}Q(R)$, 
\[ \rk(P)  = \dim_{Q(R)}(P\otimes_{R_{\D}}Q(R)).\] 
 $P$ is finitely-generated iff it's rank is finite. 

\begin{prop} \label{prop:proj_rank} The state space $\lG_{\D}$ is a projective $R_{\D}$-module 
of rank equal to the number $|\adm(\Gamma)|$ of Tait colorings of $\Gamma$. 
\end{prop} 
\begin{proof} The data of maps and homotopies 
in the localized theory, see proof of Proposition~\ref{prop:exact_sq}, can be encoded by 
the diagram below, with $R_{\D}$-modules $V_0, \dots, V_3$. 

\[
\NB{
\tikz[scale =2]{
\node (V1) at (-1, 1) {$V_1$};
\node (V2) at ( 1, 1) {$V_2$};
\node (V3) at ( 1,-1) {$V_3$};
\node (V4) at (-1,-1) {$V_0$};
\draw[-to] (V1) ..controls +( 0.4, 0.2) and +(-0.4, 0.2) .. (V2) node[midway, above] {$\alpha_1$};
\draw[to-] (V1) ..controls +( 0.4,-0.2) and +(-0.4,-0.2) .. (V2) node[midway, below] {$\beta_1$};
\draw[-to] (V2) ..controls +( 0.2,-0.4) and +( 0.2, 0.4) .. (V3) node[midway, right] {$\alpha_2$};
\draw[to-] (V2) ..controls +(-0.2,-0.4) and +(-0.2, 0.4) .. (V3) node[midway, left] {$\beta_2$};
\draw[-to] (V3) ..controls +(-0.4,-0.2) and +( 0.4,-0.2) .. (V4) node[midway, below] {$\alpha_3$};
\draw[to-] (V3) ..controls +(-0.4, 0.2) and +( 0.4, 0.2) .. (V4) node[midway, above] {$\beta_3$};
\draw[-to] (V4) ..controls +(-0.2, 0.4) and +(-0.2,-0.4) .. (V1) node[midway, left] {$\alpha_0$};
\draw[to-] (V4) ..controls +( 0.2, 0.4) and +( 0.2,-0.4) .. (V1) node[midway, right] {$\beta_0$};
}
}
\]
with homogeneous maps $\alpha_i,\beta_i$ 
 and index $i\in \Z/4$ understood modulo $4$, that satisfy 
 \begin{align*}
 \alpha_{i+1}\alpha_i   = &  0, \\
 \beta_i\beta_{i+1}  = & 0, \\ 
 \beta_{i+1}\alpha_{i+1} + \alpha_i \beta_i   = & 1_{V_i} . 
 \end{align*}
 Both $\alpha$'s and $\beta$'s are differentials, and the four-periodic complex is both $\alpha$-exact and $\beta$-exact. 
Maps $\beta_{i}\alpha_i$ and $\alpha_{i-1}\beta_{i-1}$ are mutually-orthogonal idempotents in $\mathrm{End}(V_i)$. 
These projections decompose $V_i$ into the direct sum of two subspaces, 
\[ V_i \cong \mathrm{im}(\beta_{i}\alpha_i) \oplus \mathrm{im}(\alpha_{i-1}\beta_{i-1}). \]
The complex decomposes into the direct sum of four exact 
complexes 
\[ 0 \lra \mathrm{im}(\beta_{i}\alpha_i) \stackrel{\cong}{\lra} \mathrm{im}(\alpha_i\beta_i) 
\lra 0. \]
There is a canonical isomorphism of 
$R_{\D}$-modules 
\[V_0 \oplus V_2 \ \cong \ V_1 \oplus V_3 ,\]
given by the mutually-inverse matrices  of maps 
\[\left( {\begin{array}{cc}
   \alpha_0 & \beta_1 \\
   \beta_3 & \alpha_2 \\
  \end{array} } \right) , \ \ 
  \left( {\begin{array}{cc}
   \beta_0 & \alpha_3\\
   \alpha_1 & \beta_2 \\
  \end{array} } \right) . 
  \]
Furthermore, $V_i$ is isomorphic to the direct sum of two terms that are isomorphic to direct summands of $V_{i-1}$ and $V_{i+1}$,  and the 
following lemma holds. 

\begin{lemma} If three out of the four $V_i$'s are projective graded $R_{\D}$-modules, then the fourth one is a projective graded $R_{\D}$-module as well. The following equality on their ranks holds 
\begin{align}\label{eq:sum-eq}
\rk(V_0) + \rk(V_2)  =  
 \rk(V_1) + \rk(V_3).
 \end{align}
\end{lemma}

We can now prove Proposition~\ref{prop:proj_rank} by induction on the number of vertices of $\Gamma$.  Proposition~\ref{prop:proj_rank} is 
clear for $\Gamma$ with no vertices (such graph is a union of circles). Such $\Gamma$ has $3^m$ 
Tait colorings, where $m$ is the number of circles 
in $\Gamma$, and the rank of free $R_{\D}$-module 
$\lG_{\D}$ is $3^m$, in view of Proposition~\ref{prop:dec-circle}, which holds in the localized theory as well. 

Hence, for a graph $\Gamma$ without vertices, 
$ \rk(\lG_{\D}) = |\adm(\Gamma)|$. 

If graph $\Gamma$ 
has a bridge, $ \rk(\lG_{\D}) = 0 = |\adm(\Gamma)|$, since the localized state space 
$\lG_{\D}=0$, as well as the state space itself, 
$\lG=0$. At the same time, $\Gamma$ has no 
Tait colorings. 

If $\Gamma$ has $n>0$ vertices and a region with at most $4$ sides, propositions 
in Section~\ref{sub:dir_sum}, which remain true in the localized theory, show that $\lG_{\D}$ is 
isomorphic to $\brak{\Gamma'}_{\D}$ or a direct sum of two such state spaces for graphs $\Gamma'$ with fewer vertices. 

Otherwise, any region of $\Gamma$ has at least five sides. The graph $\Gamma$ being planar, there necessarily exists a region with exactly five sides. Take one of the edges of a pentagon region, and modify a neighborhood of this edge to form three other graphs in the square (\ref{eq:square0}) so that $\Gamma_I = \Gamma$. 
Then graph $\Gamma_{H}$ contains a square region, 
so the statement of the proposition holds for it. 
Likewise, the remaining two graphs have fewer vertices than $\Gamma$ and satisfy the 
statement of the proposition. This implies the same property for $\Gamma$. 
  
The degree equality follows from the equation (\ref{eq:sum-eq}). The same equation is satisfied 
by the number of Tait colorings of the four graphs:
\[ |\adm(\Gamma_{||})| + |\adm(\Gamma_{H})| = 
|\adm(\Gamma_{I})| + |\adm(\Gamma_{=})|, \]
see also \cite[Definition 2.1]{FenKru2009} for 
the corresponding defining relations on the chromatic polynomial of planar trivalent graphs 
at $Q=4$. 

These observations together complete the induction base and step and prove the proposition. 
\end{proof}

 \subsection{Graded dimension}\ 
 \label{sec:qdim}
 
To get a numerical invariant out of our construction, we assign to a planar trivalent graph $\Gamma$ the 
 graded dimension of the graded finite-dimensional $\kk$-vector space $\brak{\Gamma}_{\kk}$, 
 \[ \qdim(\Gamma) = \gdim(\brak{\Gamma}_{\kk}).\] 
The graded dimension of a graded finite-dimensional vector space $V = \oplusop{n\in \Z} V_n$ is given by
 \[ \gdim(V) = \sum_{n\in \Z} \dim(V_n)q^n.\]
  We call the invariant $\qdim(\Gamma)$ the \emph{quantum dimension} of $\Gamma$ (or 
  quantum $SO(3)$ dimension). 
 Quantum dimension takes values in the semiring $\Z_+[q,q^{-1}]$. Direct sum decompositions 
 from Section~\ref{sub:dir_sum} allow to express the quantum dimension of a graph that contains a facet with four or fewer edges as a linear combination of quantum dimensions of its simplifications. 
 
If a graph $\Gamma$ is bipartite, it contains a facet with at most four edges and its reductions are bipartite as well, so that $\qdim(\Gamma)$ can be computed inductively. 

\begin{prop} For any bipartite web $\Gamma$, 
the quantum dimension $\qdim(\Gamma)$ equals the 
Kuperberg bracket of $\Gamma$. 

\end{prop} 
\begin{proof} 
This is immediate, since the recursive relations are identical. Kuperberg bracket of $\Gamma$, defined in~\cite{Ku}, is normalized here to lie in $\Z_+[q,q^{-1}]$, the 
same normalization as in~\cite{SL3}. 
Kuperberg bracket is also the graded dimension of the $sl(3)$ link homology groups of $\Gamma$, 
see~\cite{SL3}. 
The latter space can be defined over any field, with the graded dimension independent of the field. 
 \end{proof}

When $\Gamma$ is bipartite, its quantum dimension lies either in $\Z_+[q^2,q^{-2}]$ or in $q\Z_+[q^2,q^{-2}]$, that is, either only even or only odd powers of $q$ have nonzero coefficients. The parity equals the parity of 
$v(\Gamma)/2$, where $v(\Gamma)$ is the number of vertices of $\Gamma$, necessarily even. 

\begin{example} Here is an example 
of a graph where graded dimension fails the
parity property. 

\begin{align*}
\qdim\left(\brak{\NB{
\!\!\!\begin{tikzpicture}[scale=0.25]
\coordinate (A) at (-2,1.5);
\coordinate (B) at (2,1.5);
\coordinate (C) at (2,-1.5);
\coordinate (D) at (-2,-1.5);
\coordinate (a) at (-1,.5);
\coordinate (b) at (1,.5);
\coordinate (c) at (1,-.5);
\coordinate (d) at (-1,-.5);
\coordinate (mt) at (0,.5);
\coordinate (mb) at (0,-.5);
\coordinate (ML) at (-2, 0);
\coordinate (MR) at (2, 0);
\draw (a) -- (mt) -- (b) --(c) -- (mb) -- (d) -- (a);
\draw (A) -- (B) -- (MR) -- (C) -- (D) -- (ML) -- (A);
\draw (a) -- (A);
\draw (b) -- (B);
\draw (c) -- (C);
\draw (d) -- (D);
\draw (mt) -- (mb);
\draw (ML) ..  controls +(-2, 0) and +(-4,0) .. (0,2.5)  .. controls +(4,0) and +(2,0) .. (MR);
\end{tikzpicture}\!\!\!
}}_\kk\right)
&=
\qdim\left(\brak{\NB{\!\!\!
\begin{tikzpicture}[scale=0.25]
\coordinate (A) at (-2,1.5);
\coordinate (B) at (2,1.5);
\coordinate (C) at (2,-1.5);
\coordinate (D) at (-2,-1.5);
\coordinate (a) at (-1,.5);
\coordinate (b) at (1,.5);
\coordinate (c) at (1,-.5);
\coordinate (d) at (-1,-.5);
\coordinate (mt) at (0,.5);
\coordinate (mb) at (0,-.5);
\coordinate (ML) at (-2, 0);
\coordinate (MR) at (2, 0);
\draw (b) -- (c);
\draw (A) -- (B) -- (MR) -- (C) -- (D) -- (ML) -- (A);
\draw (A) .. controls (a) and (d) .. (D);
\draw (b) .. controls (mt) and (mb) .. (c);
\draw (b) -- (B);
\draw (c) -- (C);
\draw (ML) ..  controls +(-2, 0) and +(-4,0) .. (0,2.5)  .. controls +(4,0) and +(2,0) .. (MR);
\end{tikzpicture} \!\!\!
}}_\kk\right)
+
\qdim\left(\brak{\NB{\!\!\!
\begin{tikzpicture}[scale=0.25]
\coordinate (A) at (-2,1.5);
\coordinate (B) at (2,1.5);
\coordinate (C) at (2,-1.5);
\coordinate (D) at (-2,-1.5);
\coordinate (a) at (-1,.5);
\coordinate (b) at (1,.5);
\coordinate (c) at (1,-.5);
\coordinate (d) at (-1,-.5);
\coordinate (mt) at (0,.5);
\coordinate (mb) at (0,-.5);
\coordinate (ML) at (-2, 0);
\coordinate (MR) at (2, 0);
\draw (b) -- (c);
\draw (A) -- (B) -- (MR) -- (C) -- (D) -- (ML) -- (A);
\draw (A) .. controls (a) and (mt) .. (b);
\draw (c) .. controls (mb) and (d) .. (D);
\draw (b) -- (B);
\draw (c) -- (C);
\draw (ML) ..  controls +(-2, 0) and +(-4,0) .. (0,2.5)  .. controls +(4,0) and +(2,0) .. (MR);
\!\!\!\end{tikzpicture}
}
}_\kk\right)
\\ 
&=
 [2]
\qdim\left(\brak{\NB{
\!\!\!\begin{tikzpicture}[scale=0.25]
\coordinate (A) at (-2,1.5);
\coordinate (B) at (2,1.5);
\coordinate (C) at (2,-1.5);
\coordinate (D) at (-2,-1.5);
\coordinate (a) at (-1,.5);
\coordinate (b) at (1,.5);
\coordinate (c) at (1,-.5);
\coordinate (d) at (-1,-.5);
\coordinate (mt) at (0,.5);
\coordinate (mb) at (0,-.5);
\coordinate (ML) at (-2, 0);
\coordinate (MR) at (2, 0);
\draw (A) -- (B) -- (MR) -- (C) -- (D) -- (ML) -- (A);
\draw (A) .. controls (a) and (d) .. (D);
\draw (B) .. controls (b) and (c) .. (C);
\draw (ML) ..  controls +(-2, 0) and +(-4,0) .. (0,2.5)  .. controls +(4,0) and +(2,0) .. (MR);
\end{tikzpicture}\!\!\!
}}_\kk\right) +
\qdim\left(\brak{\NB{
\!\!\!\begin{tikzpicture}[scale=0.25]
\coordinate (A) at (-1.5,1.5);
\coordinate (B) at (1.5,1.5);
\coordinate (C) at (1.5,-1.5);
\coordinate (D) at (-1.5,-1.5);
\coordinate (a) at (-1,.5);
\coordinate (b) at (1,.5);
\coordinate (c) at (1,-.5);
\coordinate (d) at (-1,-.5);
\coordinate (mt) at (0,1.5);
\coordinate (mb) at (0,-1.5);
\coordinate (ML) at (-1.5, 0);
\coordinate (MR) at (1.5, 0);
\draw (mt) -- (mb);
\draw (mt) .. controls (A) and (A) .. (ML);
\draw (mt) .. controls (B) and (B) .. (MR);
\draw (mb) .. controls (D) and (D) .. (ML);
\draw (mb) .. controls (C) and (C) .. (MR);
\draw (ML) ..  controls +(-2, 0) and +(-4,0) .. (0,2.5)  .. controls +(4,0) and +(2,0) .. (MR);
\end{tikzpicture}\!\!\!
}}_\kk\right)
\\ &=
([2] +1)
\qdim\left(\brak{\NB{
\!\!\!\begin{tikzpicture}[scale=0.25]
\coordinate (A) at (-1.5,1.5);
\coordinate (B) at (1.5,1.5);
\coordinate (C) at (1.5,-1.5);
\coordinate (D) at (-1.5,-1.5);
\coordinate (a) at (-1,.5);
\coordinate (b) at (1,.5);
\coordinate (c) at (1,-.5);
\coordinate (d) at (-1,-.5);
\coordinate (mt) at (0,1);
\coordinate (mb) at (0,-1);
\coordinate (ML) at (-1.5, 0);
\coordinate (MR) at (1.5, 0);
\draw (ML) ..  controls (A) and (B) .. (MR);
\draw (ML) ..  controls (D) and (C) .. (MR);
\draw (ML) ..  controls +(-2, 0) and +(-4,0) .. (0,2)  .. controls +(4,0) and +(2,0) .. (MR);
\end{tikzpicture}\!\!\!
}}_\kk\right)
\\ & = ([2] + 1)[2][3],
\end{align*}
where $[n] = \frac{q^n-q^{-n}}{q-q^{-1}}$.
In particular, the graded dimension of this graph above does not satisfy the parity property. 
\end{example}
  
There are potential variations on the quantum dimension $\qdim(\Gamma)$ given by using the 
original ring $R$ or base changes other than 
the graded homomorphism $R\lra \kk$. One could 
define the quantum dimension as the graded 
dimension of $\lG$ rather than $\lG_{\kk}$, 
normalized by dividing by the graded dimension 
of $R$; the latter is $((1-q^2)(1-q^4)(1-q^6))^{-1}$. One can also first resolve $\lG$ into 
a complex of free graded $R$-modules and then take its graded Euler characteristic. For $\Gamma$ such that $\lG$ is a free graded $R$-module all these definitions would result in the same graded dimension, but we don't know whether $\lG$ has this property for any $\Gamma$. Lacking enough information about the structure of graded $R$-module $\lG$ beyond the bipartite case we 
chose to restrict here to just one version of the quantum dimension. 
  
When a graph $\Gamma$ is reducible using the rules given by Propositions~\ref{prop:dec-circle}, \ref{prop:dec-digon}, \ref{prop:dec-square} and \ref{prop:dec-triangle}, $\brak{\Gamma}$ is a free graded $R$-module, and the reduced theory 
$\brak{\Gamma}_{\kk}$ has graded dimension 
equal to the graded rank of $\brak{\Gamma}$. The quantum dimension of such $\Gamma$ can be computed using the relations from the decategorified versions of the above propositions. 
  
Even restricting to such graphs, 
the quantum dimension is not given by the Yamada polynomial, which is the  invariant coming from planar networks built out of the 3-dimensional irreducible representation $V$ of quantum $sl(2)$ and the one-dimensional space of invariants in the third tensor power of $V$, see~\cite{FenKru2009,FenKru2010}.

\subsection{Base change into PID} 
\label{sec:base_change_pid}\

We don't know whether $\brak{\Gamma}$ is always a graded free module, and it makes sense to consider 
graded base changes $\psi: R\lra S$ with $S$ a PID. 
There is a family of such base changes given by 
taking $S=\kf[E]$ where $\kf$ is a field of characteristic two (a field extension of $\kk$) and 
$\deg(E)=2$. A degree-preserving homomorphism 
\[\psi_{\lambda} \ : \ R \lra \kf[E]\]
is given by 
\[\psi_{\lambda}(E_1)  =\lambda_1 E, \ \ 
\psi_{\lambda}(E_2)  =\lambda_2 E^2, \ \ 
\psi_{\lambda}(E_1)  =\lambda_3 E^3 \]
where $\lambda_1,\lambda_2,\lambda_3\in \kf$, 
and we denote $\lambda=(\lambda_1,\lambda_2,\lambda_3)$. Homomorphism $\psi_{\lambda}$ is 
surjective iff $\lambda_1 \not= 0$ and $\kf=\kk$. 

For the base change $\psi_{\lambda}$ we can form 
the corresponding state space $\lG_{\psi_{\lambda}}$ or just $\lG_{\lambda}$. It's a finitely-generated 
graded $\kf[E]$-module.  Propositions~\ref{prop:S_free} implies that 
$\lG_{\lambda}$ is a finitely-generated free graded 
$\kf[E]$-module. Its graded rank $\qdim_{\lambda}(\Gamma)$ is an invariant of $\Gamma$. 

When $\lambda=(0,0,0)$, which is a degenerate case, the map $\psi_{(0,0,0)}$ factors through $\kk$:
\[
\tikz{
\node (R) at (0,0) {$R$};
\node (k) at (2,0) {$\kk$};
\node (kp) at (4,0) {$\kk'$};
\node (kpE) at (6,0) {$\kk'[E]$};
\draw[-to] (R) -- (k) node[midway, above, font=\scriptsize] {$E_i \mapsto 0$};
\draw[right hook-to] (k)-- (kp);
\draw[right hook-to] (kp)-- (kpE);
\draw[-to] (R) .. controls +(1,-1)  and + (-1,-1) .. (kpE) node[below, midway] {$\psi_{(0,0,0)}$};
}
\]
Hence,  $\brak{\Gamma}_{(0,0,0)}\cong \brak{\Gamma}_\kk \otimes_{\kk} \kk'[E]$. If $\lG$ is a free 
graded $R$-module, the quantum dimension $\qdim_{\lambda}(\Gamma)$ is the same for all $\lambda$ 
and equals $\qdim(\Gamma)$ defined in 
Section~\ref{sec:qdim}. 

We say that $\psi_{\lambda}$ or $\lambda$ is $\D$-localizable if $\psi_{\lambda}(\D)\not=0$.  
This is equivalent to $\lambda_3+\lambda_1\lambda_2\not= 0$. 

\begin{prop} If $\lambda$ is $\D$-localizable, 
the state space $\lG_{\lambda}$ is a free $\kf[E]$-module of rank equal to the number $|\adm(\Gamma)|$ of Tait colorings of $\Gamma$. 
\end{prop} 

\begin{proof} 
Note that the proposition is about the rank of $\lG_{\lambda}$, not its graded rank. For a 
$\D$-localizable $\lambda$, compose $\psi_{\lambda}$ with the inclusion of rings $\kf[E]\subset \kf[E,E^{-1}]$ to get a homomorphism 
\[ \psi_{\lambda,\D} \ : \ R \lra \kf[E,E^{-1}].\]
This homomorphism factors through the inclusion 
$R\subset R[\D^{-1}]$, since the image of $\D$ under 
$\psi_{\lambda}$ is a nonzero multiple of $E^3$, so invertible in $\kf[E,E^{-1}]$. Hence, the analogues 
of results of Section~\ref{sec:four_end}, including 
Propositions~\ref{prop:exact_sq} and~\ref{prop:proj_rank} hold 
for the base change $\psi_{\lambda,\D}$. Any graded 
projective module over $\kf[E,E^{-1}]$ is graded free. In particular, $\lG_{\lambda}$ is a free 
module of rank -- the number of Tait colorings of $\Gamma$. 
\end{proof} 

A similar result holds for the homomorphism 
\[ \phi \ : \ R \lra \kf[E], \ \ \deg(E) = 6\] 
given by 
\[ \phi(E_1) = \phi(E_2) = 0, \ \  \phi(E_3) = E.\]
The image $\phi(\D) = \phi(E_3 + E_1 E_2) = E$ is invertible in the localized ring $\kf[E,E^{-1}]$ 

\begin{prop}  
The state space $\lG_{\phi}$ is a free $\kf[E]$-module of rank equal to the number $|\adm(\Gamma)|$ of Tait colorings of $\Gamma$. 
\end{prop}

It's a very interesting problem, related to the four-color theorem, to understand the graded ranks 
of $\lG_{\lambda}$, $\lG_{\phi}$, and, more generally, the structure of $R$-modules $\lG$.

\bibliographystyle{alphaurl}
\bibliography{biblio}

\begin{thebibliography}{BHMV95}

\bibitem[AH77]{ApHa}
Kenneth Appel and Wolfgang Haken.
\newblock Every planar map is four colorable. {I}. {D}ischarging.
\newblock {\em Illinois J. Math.}, 21(3):429--490, 1977.
\newblock URL: \url{http://projecteuclid.org/euclid.ijm/1256049011}.

\bibitem[AH89]{ApHaBook}
Kenneth Appel and Wolfgang Haken.
\newblock {\em Every planar map is four colorable}, volume~98 of {\em
  Contemporary Mathematics}.
\newblock American Mathematical Society, Providence, RI, 1989.
\newblock With the collaboration of J. Koch.
\newblock URL: \url{https://doi.org/10.1090/conm/098}, \href
  {http://dx.doi.org/10.1090/conm/098} {\path{doi:10.1090/conm/098}}.

\bibitem[AHK77]{ApHaKo}
Kenneth Appel, Wolfgang Haken, and John Koch.
\newblock Every planar map is four colorable. {II}. {R}educibility.
\newblock {\em Illinois J. Math.}, 21(3):491--567, 1977.
\newblock URL: \url{http://projecteuclid.org/euclid.ijm/1256049012}.

\bibitem[BHMV95]{BHMV}
Christian Blanchet, Nathan Habegger, Gregor Masbaum, and Pierre Vogel.
\newblock Topological quantum field theories derived from the {K}auffman
  bracket.
\newblock {\em Topology}, 34(4):883--927, 1995.
\newblock URL: \url{http://dx.doi.org/10.1016/0040-9383(94)00051-4}.

\bibitem[Die00]{Diestel}
Reinhard Diestel.
\newblock {\em Graph theory}, volume 173 of {\em Graduate Texts in
  Mathematics}.
\newblock Springer-Verlag, New York, second edition, 2000.
\newblock URL: \url{https://doi.org/10.1007/b100033}, \href
  {http://dx.doi.org/10.1007/b100033} {\path{doi:10.1007/b100033}}.

\bibitem[ETW17]{FunctorialitySLN}
Michael Ehrig, Daniel Tubbenhauer, and Paul Wedrich.
\newblock Functoriality of colored link homologies, 2017.
\newblock ArXiv E-print.
\newblock URL: \url{https://arxiv.org/abs/1703.06691}.

\bibitem[FK09]{FenKru2009}
Paul Fendley and Vyacheslav Krushkal.
\newblock Tutte chromatic identities from the {T}emperley-{L}ieb algebra.
\newblock {\em Geom. Topol.}, 13(2):709--741, 2009.
\newblock URL: \url{https://doi.org/10.2140/gt.2009.13.709}, \href
  {http://dx.doi.org/10.2140/gt.2009.13.709}
  {\path{doi:10.2140/gt.2009.13.709}}.

\bibitem[FK10]{FenKru2010}
Paul Fendley and Vyacheslav Krushkal.
\newblock Link invariants, the chromatic polynomial and the {P}otts model.
\newblock {\em Adv. Theor. Math. Phys.}, 14(2):507--540, 2010.
\newblock URL: \url{http://projecteuclid.org/euclid.atmp/1288619151}.

\bibitem[Kho04]{SL3}
Mikhail Khovanov.
\newblock sl(3) link homology.
\newblock {\em Algebr. Geom. Topol.}, 4:1045--1081, 2004.
\newblock URL: \url{https://doi.org/10.2140/agt.2004.4.1045}, \href
  {http://dx.doi.org/10.2140/agt.2004.4.1045}
  {\path{doi:10.2140/agt.2004.4.1045}}.

\bibitem[KM15]{KM1}
{Peter~B}. {Kronheimer} and {Tomasz~S.} {Mrowka}.
\newblock {Tait colorings, and an instanton homology for webs and foams}.
\newblock {\em To appear in Journal of the European Mathematical Society},
  August 2015.
\newblock ArXiv e-prints.
\newblock URL: \url{http://arxiv.org/abs/1508.07205}.

\bibitem[KM16]{KM2}
{Peter~B}. {Kronheimer} and {Tomasz~S.} {Mrowka}.
\newblock Exact triangles for {$SO(3)$} instanton homology of webs.
\newblock {\em J. Topol.}, 9(3):774--796, 2016.
\newblock URL: \url{https://doi.org/10.1112/jtopol/jtw010}, \href
  {http://dx.doi.org/10.1112/jtopol/jtw010} {\path{doi:10.1112/jtopol/jtw010}}.

\bibitem[KM17]{KM3}
{Peter~B}. {Kronheimer} and {Tomasz~S.} {Mrowka}.
\newblock {A deformation of instanton homology for webs}.
\newblock {\em ArXiv e-prints}, October 2017.
\newblock \href {http://arxiv.org/abs/1710.05002} {\path{arXiv:1710.05002}}.

\bibitem[Kup96]{Ku}
Greg Kuperberg.
\newblock Spiders for rank {$2$} {L}ie algebras.
\newblock {\em Comm. Math. Phys.}, 180(1):109--151, 1996.
\newblock URL:
  \url{http://projecteuclid.org/getRecord?id=euclid.cmp/1104287237}.

\bibitem[MN08]{MoNi}
Scott Morrison and Ari Nieh.
\newblock On {K}hovanov's cobordism theory for {$\mathfrak{su}_3$} knot
  homology.
\newblock {\em J. Knot Theory Ramifications}, 17(9):1121--1173, 2008.
\newblock URL: \url{https://doi.org/10.1142/S0218216508006555}, \href
  {http://dx.doi.org/10.1142/S0218216508006555}
  {\path{doi:10.1142/S0218216508006555}}.

\bibitem[MV07]{MaV}
Marco Mackaay and Pedro Vaz.
\newblock The universal {${\rm sl}_3$}-link homology.
\newblock {\em Algebr. Geom. Topol.}, 7:1135--1169, 2007.
\newblock URL: \url{https://doi.org/10.2140/agt.2007.7.1135}, \href
  {http://dx.doi.org/10.2140/agt.2007.7.1135}
  {\path{doi:10.2140/agt.2007.7.1135}}.

\bibitem[RW17]{RW1}
Louis-Hadrien Robert and Emmanuel Wagner.
\newblock A closed formula for the evaluation of $\mathfrak{sl}_n$-foams,
  February 2017.
\newblock ArXiv E-Print.
\newblock URL: \url{https://arxiv.org/abs/1702.04140}.

\bibitem[Tho98]{Thomas}
Robin Thomas.
\newblock An update on the four-color theorem.
\newblock {\em Notices Amer. Math. Soc.}, 45(7):848--859, 1998.

\end{thebibliography}

\end{document}